\documentclass[11pt]{article}      
\usepackage[margin=1in,bottom=1in]{geometry}  
\usepackage{libertine}
\usepackage{amsthm}
\usepackage{amsmath}

\usepackage{amssymb}
\usepackage{stmaryrd}
\usepackage{mathrsfs} 
\usepackage{mathtools}
\usepackage{thm-restate}
\usepackage{hyperref}
\usepackage[capitalize,nameinlink]{cleveref}
\usepackage{url}
\usepackage[size=tiny]{todonotes}
\usepackage{mathrsfs} 
\usepackage{soul}
\usepackage{relsize}
\usepackage{enumerate}
\usepackage[all,defaultlines=3]{nowidow}
\usepackage[mathlines]{lineno}
\usepackage{multicol}
\usepackage{enumitem}
\usepackage{textpos}
\setlist[itemize]{topsep=4pt,itemsep=3pt,parsep=0pt} 
\setlist[enumerate]{topsep=4pt,itemsep=3pt,parsep=0pt} 

\crefname{claim}{Claim}{Claims}
\crefname{figure}{Figure}{Figures}
\crefname{conjecture}{Conjecture}{Conjectures}
\renewcommand{\preceq}{\preccurlyeq}

\renewcommand{\npreceq}{\not\preccurlyeq}

\newtheorem{theorem}{Theorem}

\newtheorem{corollary}[theorem]{Corollary}
\newtheorem{conjecture}{Conjecture}

\newtheorem{lemma}[theorem]{Lemma}

\theoremstyle{definition}
\newtheorem{definition}{Definition}
\theoremstyle{plain}

\theoremstyle{definition}

\newtheorem*{example*}{Example}
\renewcommand{\cal}{\mathcal}

\hypersetup{
    colorlinks,
    linkcolor={violet!50!black},
    citecolor={blue!50!black},
    urlcolor={blue!80!black}
}

\def\phi{\varphi}
\newcommand{\N}{\mathbb{N}}

\newcommand{\Bb}{\mathscr{B}}
\newcommand{\Cc}{\mathscr{C}}
\newcommand{\Dd}{\mathscr{D}}
\newcommand{\Ee}{\mathscr{E}}

\newcommand{\Oof}{\mathcal{O}}

\newcommand{\Ball}{\mathrm{Ball}}
\newcommand{\dist}{\mathrm{dist}}
\newcommand{\Ff}{\mathcal{F}}

\newcommand\wcol{{\rm wcol}}

\newcommand{\Pp}{\mathcal P}

\def\N{\mathbb N} 
\def\epsilon{\varepsilon}
\def\eps{\varepsilon}
\newcommand{\Oh}{\Oof}
\renewcommand{\deg}{\mathsf{deg}}

\renewcommand{\leq}{\leqslant}
\renewcommand{\geq}{\geqslant}

\renewcommand{\setminus}{-}

\newcommand{\FO}{\mathsf{FO}}
\newcommand{\MSO}{\mathsf{MSO}}
\newcommand{\CMSO}{\mathsf{CMSO}}

\newcommand{\Graphs}{\mathsf{Graphs}}
\newcommand{\arity}{\mathsf{arity}}

\newcommand{\Af}{\mathbb{A}}
\newcommand{\adj}{\mathsf{adj}}

\newcommand{\tup}{\bar}

\tikzstyle{vertex}=[circle,draw=black,fill=white,minimum size=0.15cm,inner sep=0pt, very thick]

\begin{document}

\newcommand{\funding}{
This work is a part of project BOBR that received funding from the European Research Council (ERC) under the European Union’s Horizon 2020 research and innovation programme (grant agreement No. 948057).}

\title{Graph classes through the lens of logic\thanks{\funding}}
\date{}
\author{
  Michał Pilipczuk\thanks{Institute of Informatics, University of Warsaw, Poland. E-mail: \href{michal.pilipczuk@mimuw.edu.pl}{michal.pilipczuk@mimuw.edu.pl}}
}
\maketitle

\begin{abstract}
 Graph transformations definable in logic can be described using the notion of {\em{transductions}}. By understanding transductions as a basic embedding mechanism, which captures the possibility of encoding one graph in another graph by means of logical formulas, we obtain a new perspective on the landscape of graph classes and of their properties. The aim of this survey is to give a comprehensive presentation of this angle on structural graph theory.

 We first give a logic-focused overview of classic graph-theoretic concepts, such as treedepth, shrubdepth, treewidth, cliquewidth, twin-width, bounded expansion, and nowhere denseness. Then, we present recent developments related to notions defined purely through transductions, such as monadic stability, monadic dependence, and classes of structurally sparse graphs.
\end{abstract}

\begin{textblock}{20}(-2.1,7.2)
 \includegraphics[width=60px]{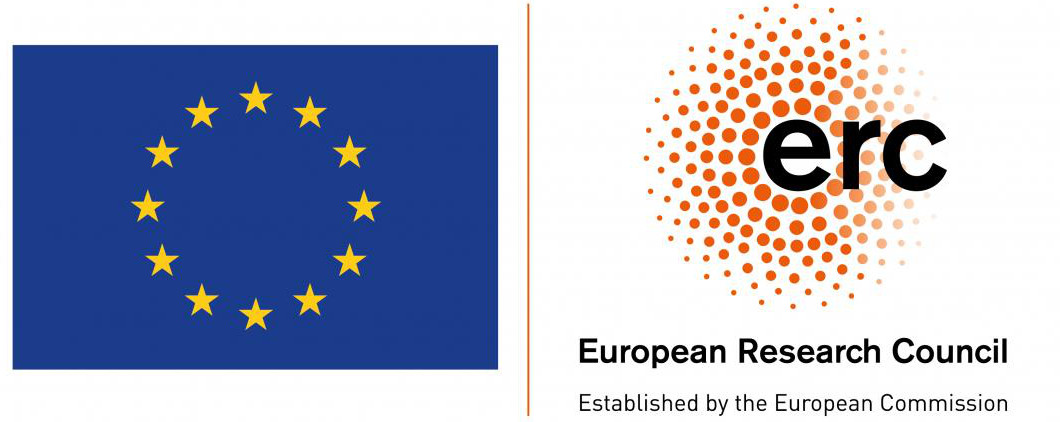}
\end{textblock}

\thispagestyle{empty}

\newpage

\tableofcontents
\thispagestyle{empty}

\newpage

\setcounter{page}{1}

\section{Introduction}\label{sec:introduction}

Arguably, the central theme of structural graph theory are duality theorems. Typically, they tie together concepts of structure with concepts of non-structure, by stating that every graph is either fundamentally structured, or contains an unstructured part.

\hyphenation{hie-rar-chi-cal}
\hyphenation{par-ti-cu-lar-ly}

Concepts of structure are often expressed through various forms of graph decompositions, which explain how to break a graph at hand into simpler pieces. Among those, decompositions in the form of {\em{trees}} are very common and useful, since they are capable of exposing a hidden hierarchical structure. However, weaker forms of decompositions can be applied to graphs that are fundamentally not tree-like; examples include {\em{covers}} or {\em{colorings}}. The general principle of using decompositions is of the local-to-global nature: understanding the simpler pieces into which the studied graph is decomposed allows us to get a grasp on its global structure. This local-to-global methodology is particularly applicable in the field of algorithm design, where the understanding takes the form of an efficient algorithm solving a computational problem. Namely, having understood the problem on the simpler pieces, we can piece together a solution for the whole graph, for instance using dynamic programming.

Concepts of non-structure typically take the form of concrete, complicated combinatorial objects that are embedded in the graph. Such {\em{obstructions}} are difficult to break using the adopted notion of decomposition, and therefore they witness the non-existence of decompositions with low values of relevant parameters. In the algorithmic context, obstructions can be used for constructing lower bound reductions.

With these intuitions in mind, a typical duality theorem in structural graph theory looks as follows:
\begin{center}
 {\em{If a graph $G$ does not contain any obstruction $\cal O$, then it admits a decomposition $\cal D$.}}
\end{center}
There are dozens of duality statements of this form present in the literature, differing by the considered kinds of decompositions and obstructions. The common first step in designing such a duality theorem is to fix a notion of {\em{embedding}}: in what way an obstruction can be contained in the studied graph? Having understood the considered notion of embedding and the type of obstructions that we speak about, it often follows what kind of decompositions should be dual to them. However, pinpointing this correspondence in precise terms is the truly interesting essence of structural graph theory.

Admittedly, the picture painted in the four paragraphs above may seem very vague and abstract. So let us make it more concrete by taking the theory of Graph Minors as an example. Here, the underlying notion of embedding is the {\em{minor order}}. Recall that a graph $H$ is a minor of a graph $G$ if one can map vertices of $H$ to disjoint connected subgraphs of $G$ so that the adjacency relation is preserved in the following sense: if vertices $u$ and $v$ are adjacent in $H$, then there is also an edge connecting the subgraphs of $G$ corresponding to $u$ and $v$. Thus, this is an embedding notion that is fundamentally of topological nature: one may imagine that every vertex of $H$ can be ``stretched'' to a connected subgraph of $G$.

It is therefore no surprise that adopting the minor order as the notion of embedding leads to a theory full of duality theorems involving topology. The most well-known is the Kuratowski-Wagner Theorem~\cite{Kuratowski30,Wagner37} that characterizes planar graphs as exactly those graphs that exclude $K_{3,3}$ and $K_5$ as minors. Here,  planarity --- the existence of an embedding in the plane without crossings --- could be liberally understood as a form of decomposition, while finding $K_{3,3}$ or $K_5$ as a minor is an obstruction to planarity. The theory of Graph Minors was developed by Robertson and Seymour in a monumental series of papers, culminating in the proof of {\em{Wagner's Conjecture}}~\cite{GM20}: there is no infinite anti-chain in the minor order on finite graphs. Along the way, Robertson and Seymour proved a number of fundamental duality statements. For instance, the {\em{Structure Theorem}}~\cite{GM16} is a generalization of Kuratowski-Wagner Theorem that explains the structure in graphs excluding a fixed graph $H$ as a minor. Namely, without going into technical details, every such graph $G$ can be decomposed into {\em{bags}} in a tree-like manner so that (i) the intersection of any two neighboring bags is of bounded size, and (ii) except for a bounded number of aberrations, every bag consists of a subgraph that can be embedded in a fixed surface. Another duality theorem in this spirit is the {\em{Grid Theorem}}~\cite{GM5}, which says that the decomposition above takes a very simple form when $H$ is a planar graph: simply, every bag is of constant size, implying that the whole graph has bounded {\em{treewidth}}.

The adopted notion of embedding reflects the topological viewpoint on graphs, roughly regarding them  as topological spaces consisting of points joined by segments. Therefore, the obtained theory is naturally applicable to algorithmic problems of topological character. For instance, Robertson and Seymour~\cite{GM13} gave an $\Oh_k(n^3)$-time\footnote{For a parameter $p$, the $\Oh_p(\cdot)$ notation hides factors that may depend on $p$. Also, whenever a graph is clear from the context, by $n$ we denote the number of its vertices.} algorithm for the {\em{Disjoint Paths}} problem: given a graph $G$ with specified {\em{terminal vertices}} $s_1,\ldots,s_k,t_1,\ldots,t_k$, decide whether there are vertex-disjoint paths $P_1,\ldots,P_k$ where $P_i$ connects $s_i$ with $t_i$. This is just one of very many examples of efficient algorithms that can be designed using the combinatorial advances of the Graph Minors project. Presenting an overview of those advances and applications would be worth a survey of its own. However, this is not the topic of this survey.

\medskip

Instead, let us point out that the topological viewpoint on graphs, exemplified by adopting the minor order as the basic embedding notion, makes sense in some real-life settings, while in many others it does~not. For instance, if the input graph models a transportation network, then stretching an edge into a path can be just described as introducing waypoints on a road connection, and thus is very sensible. On the other hand, if the graph is a social network, then the operation of replacing the friendship relation between two people by a path of friendships is rather absurd.

In this latter example, we rather view a graph as a {\em{database}}: it consists of a universe of objects that can be pairwise in relation or not. The first association with databases that comes to mind are {\em{query languages}}: formalisms for describing problems of selection of (tuples of) objects satisfying a certain property. What would be then a query language suitable for graphs? Theoretical computer science has already answered this question a long time ago, by considering various kinds of {\em{logic}} on graphs, and on other combinatorial structures. Namely, logic provides a robust, flexible, and formal language for expressing ~properties.

In this survey, we present an overview of an on-going multi-faceted project of understanding the landscape of structural graph theory from the point of view of logic. More precisely, we adopt {\em{First-Order transductions}} --- a notion of graph transformations definable in First-Order logic --- as the basic notion of embedding. Having done this, it is natural to investigate various decomposition and duality statements: What does it mean that some obstruction cannot be logically embedded in a graph? What kind of decompositions can be expected assuming lack of such logically-defined obstructions? What can we say about the structure of a graph assuming that it can be logically embedded into another, well-structured graph? Can these tools be used to design efficient algorithms for computational problems connected to logic? It turns out that asking these questions leads to a fascinating and still largely uncharted area of structural graph theory. This area brings together multiple branches of mathematics and computer science, like classic graph theory, algorithmics, discrete geometry, learning theory, and model theory (both finite and~infinite).

The goal of this survey is {\em{not}} to present a detailed description of all the corners of the area, {\em{nor}} to delve into technical discussions or involved proofs. We rather aim to touch many different topics, in order to highlight analogies, recurring themes, and concepts of a general nature. Therefore, while many works were unfortunately left out from the discussion for the sake of brevity and of simplicity of the narrative, we hope that this survey can serve as an invitation to various more specific subjects that the reader may explore on their own.

\paragraph{Overview of the content.} In \cref{sec:prelims}, we recall all the basic notions and definitions, and in particular we discuss the notion of transductions. Throughout this survey we mostly focus on First-Order logic~$\FO$, but Monadic Second-Order logic $\MSO$ and its variants will also be considered.

In \cref{sec:classic}, we present concepts from more classic areas of structural graph theory, and we discuss them from the adopted viewpoint of logic. The naming ``classic'' might be a bit misleading here, for we also describe rather recent developments, for instance the theory of {\em{Sparsity}} of Ne\v{s}et\v{r}il and Ossona de Mendez and the graph parameter {\em{twin-width}}.

Next, in \cref{sec:new} we discuss concepts of more logical nature, defined using transductions. We start with the heart of the theory: the notions
of {\em{monadic dependence}} and {\em{monadic stability}}. We describe how introducing these two notions organizes the whole landscape, present the recent advances on decomposition tools for monadically stable and monadically dependent graph classes, and mention algorithmic applications of these tools, particularly in the context of the model-checking problem for First-Order logic. Second, we discuss {\em{structurally sparse classes}} --- classes that can be transduced from classes of sparse graphs --- and give an overview of known structural results.

Finally, in \cref{sec:outlook} we outline three research directions that we personally find particularly interesting.

\cref{fig:main} depicts all the main concepts that we are going to discuss, and relations between them. The remainder of this survey can be regarded as a meticulous discussion of this figure.

\definecolor{tdcol}{RGB}{0,0,255}
\definecolor{pwcol}{RGB}{0,85,170}
\definecolor{twcol}{RGB}{0,170,85}
\definecolor{mfcol}{RGB}{0,255,0}
\definecolor{wwcol}{RGB}{85,170,0}
\definecolor{becol}{RGB}{170,85,0}
\definecolor{ndcol}{RGB}{255,0,0}

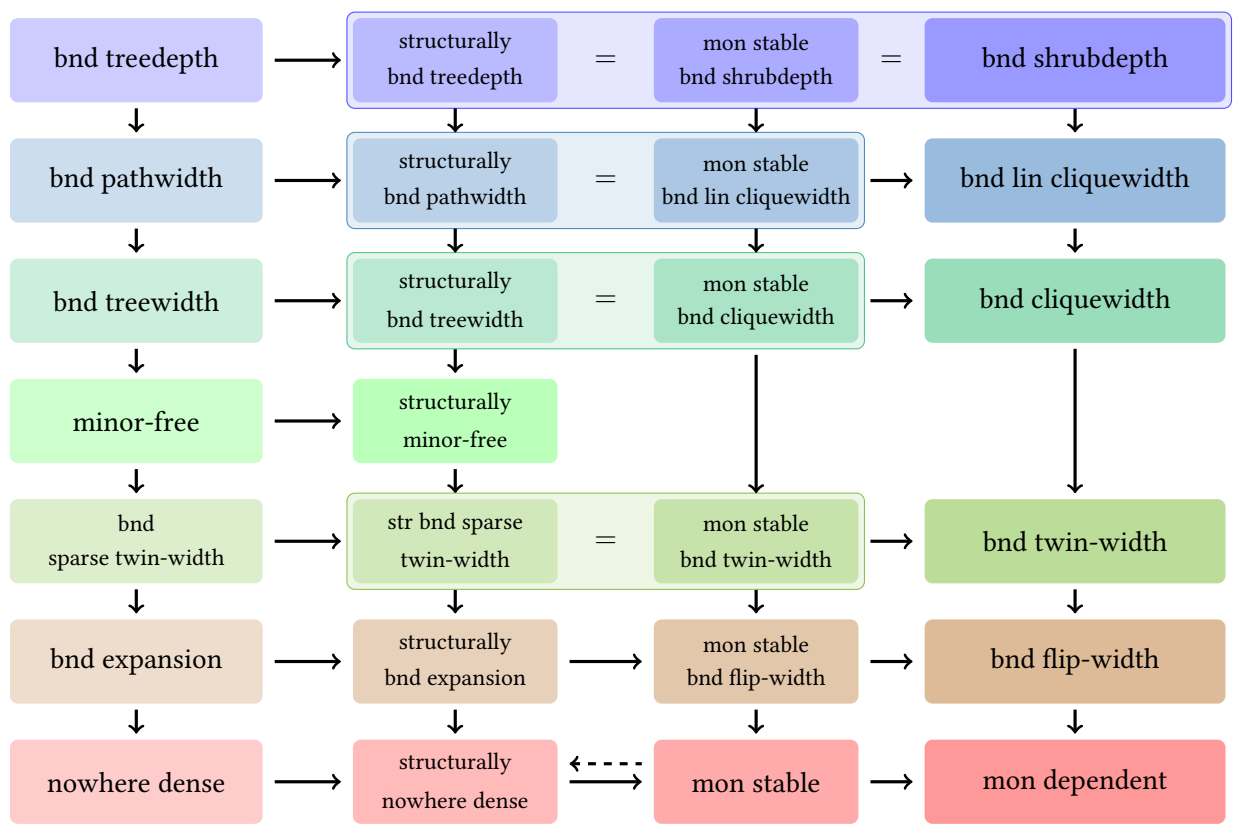
\begin{figure}[htpb]
\centering
\begin{tikzpicture}
\begin{scope}[scale=0.8]

 \begin{scope}[shift={(-7.8,-4)}]
 \draw[very thick,<-] (0,0.8) -- (0,1.2);
 \fill[twcol!20,rounded corners=3] (-2.1,-0.7) rectangle (2.1,0.7);
 \node at (0,0) {bnd treewidth};
 \end{scope}

 \begin{scope}[shift={(-7.8,-6)}]
 \draw[very thick,<-] (0,0.8) -- (0,1.2);
 \fill[mfcol!20,rounded corners=3] (-2.1,-0.7) rectangle (2.1,0.7);
 \node at (0,0) {minor-free};
 \end{scope}

 \begin{scope}[shift={(-7.8,0)}]
 \fill[tdcol!20,rounded corners=3] (-2.1,-0.7) rectangle (2.1,0.7);
 \node at (0,0) {bnd treedepth};
 \end{scope}

 \begin{scope}[shift={(-7.8,-2)}]
 \draw[very thick,<-] (0,0.8) -- (0,1.2);
 \fill[pwcol!20,rounded corners=3] (-2.1,-0.7) rectangle (2.1,0.7);
 \node at (0,0) {bnd pathwidth};
 \end{scope}

 \begin{scope}[shift={(-7.8,-10)}]
 \draw[very thick,<-] (0,0.8) -- (0,1.2);
 \fill[becol!20,rounded corners=3] (-2.1,-0.7) rectangle (2.1,0.7);
 \node at (0,0) {bnd expansion};
 \end{scope}

 \begin{scope}[shift={(-7.8,-12)}]
 \draw[very thick,<-] (0,0.8) -- (0,1.2);
 \fill[ndcol!20,rounded corners=3] (-2.1,-0.7) rectangle (2.1,0.7);
 \node at (0,0) {nowhere dense};
 \end{scope}
%

 \begin{scope}[shift={(2.5,-12)}]
 \draw[very thick,<-] (-1.9,0) -- (-3.1,0);
 \fill[ndcol!33,rounded corners=3] (-1.7,-0.7) rectangle (1.7,0.7);
 \node at (0,0) {mon stable};
 \end{scope}


 \begin{scope}[shift={(7.8,-12)}]
 \draw[very thick,<-] (-2.7,0) -- (-3.4,0);
 \fill[ndcol!40,rounded corners=3] (-2.5,-0.7) rectangle (2.5,0.7);
 \node at (0,0) {mon dependent};
 \end{scope}

 \begin{scope}[shift={(7.8,0)}]
 \fill[tdcol!40,rounded corners=3] (-2.5,-0.7) rectangle (2.5,0.7);
 \node at (0,0) {bnd shrubdepth};
 \end{scope}

 \begin{scope}[shift={(7.8,-2)}]
 \draw[very thick,<-] (0,0.8) -- (0,1.2);
 \draw[very thick,<-] (-2.7,0) -- (-3.4,0);
 \fill[pwcol!40,rounded corners=3] (-2.5,-0.7) rectangle (2.5,0.7);
 \node at (0,0) {bnd lin cliquewidth};
 \end{scope}

 \begin{scope}[shift={(7.8,-4)}]
 \draw[very thick,<-] (0,0.8) -- (0,1.2);
 \draw[very thick,<-] (-2.7,0) -- (-3.4,0);
 \fill[twcol!40,rounded corners=3] (-2.5,-0.7) rectangle (2.5,0.7);
 \node at (0,0) {bnd cliquewidth};
 \end{scope}

 \begin{scope}[shift={(-7.8,-8)}]
 \draw[very thick,<-] (0,0.8) -- (0,1.2);
 \fill[wwcol!20,rounded corners=3] (-2.1,-0.7) rectangle (2.1,0.7);
 \node at (0,0.3) {\footnotesize{bnd}};
 \node at (0,-0.3) {\footnotesize{sparse twin-width}};
 \end{scope}

 \begin{scope}[shift={(7.8,-8)}]
 \draw[very thick,<-] (0,0.8) -- (0,2+1.2);
 \draw[very thick,<-] (-2.7,0) -- (-3.4,0);
 \fill[wwcol!40,rounded corners=3] (-2.5,-0.7) rectangle (2.5,0.7);
 \node at (0,0) {bnd twin-width};
 \end{scope}

 \begin{scope}[shift={(7.8,-10)}]
 \draw[very thick,<-] (0,0.8) -- (0,1.2);
 \draw[very thick,<-] (-2.7,0) -- (-3.4,0);
 \draw[very thick,<-] (0,-1.2) -- (0,-0.8);
 \fill[becol!40,rounded corners=3] (-2.5,-0.7) rectangle (2.5,0.7);
 \node at (0,0) {bnd flip-width};
 \end{scope}
%
%
%
%

 \begin{scope}[shift={(2.5,-10)}]
 \draw[very thick,<-] (0,0.8) -- (0,1.2);
 \draw[very thick,<-] (0,-1.2) -- (0,-0.8);
 \draw[very thick,<-] (-1.9,0) -- (-3.1,0);
 \fill[becol!33,rounded corners=3] (-1.7,-0.7) rectangle (1.7,0.7);
 \node at (0,0.3) {\footnotesize{mon stable}};
 \node at (0,-0.3) {\footnotesize{bnd flip-width}};
 \end{scope}

%
%

 \begin{scope}[shift={(-2.5,-6)}]
 \draw[very thick,<-] (0,0.8) -- (0,1.2);
 \draw[very thick,<-] (-1.9,0) -- (-3,0);
 \fill[mfcol!27,rounded corners=3] (-1.7,-0.7) rectangle (1.7,0.7);
 \node at (0,0.3) {\footnotesize{structurally}};
 \node at (0,-0.3) {\footnotesize{minor-free}};
 \end{scope}

 \begin{scope}[shift={(-2.5,-8)}]
 \draw[very thick,<-] (0,0.8) -- (0,1.2);
 \draw[very thick,<-] (-1.9,0) -- (-3,0);
 \fill[wwcol!27,rounded corners=3] (-1.7,-0.7) rectangle (1.7,0.7);
 \node at (0,0.3) {\footnotesize{str bnd sparse}};
 \node at (0,-0.3) {\footnotesize{twin-width}};
 \end{scope}

 \begin{scope}[shift={(-2.5,-10)}]
 \draw[very thick,<-] (0,0.8) -- (0,1.2);
 \draw[very thick,<-] (-1.9,0) -- (-3,0);
 \fill[becol!27,rounded corners=3] (-1.7,-0.7) rectangle (1.7,0.7);
 \node at (0,0.3) {\footnotesize{structurally}};
 \node at (0,-0.3) {\footnotesize{bnd expansion}};
 \end{scope}

 \begin{scope}[shift={(-2.5,-12)}]
 \draw[very thick,<-] (0,0.8) -- (0,1.2);
 \draw[very thick,<-] (-1.9,0) -- (-3,0);
 \fill[ndcol!27,rounded corners=3] (-1.7,-0.7) rectangle (1.7,0.7);
 \node at (0,0.3) {\footnotesize{structurally}};
 \node at (0,-0.3) {\footnotesize{nowhere dense}};
 \end{scope}

 \fill[tdcol!10,rounded corners=3] (-4.3,-0.8) rectangle (10.4,+0.8);
 \draw[tdcol!70,rounded corners=3] (-4.3,-0.8) rectangle (10.4,+0.8);

 \begin{scope}[shift={(-2.5,0)}]
 \draw[very thick,<-] (-1.9,0) -- (-3,0);
 \fill[tdcol!27,rounded corners=3] (-1.7,-0.7) rectangle (1.7,0.7);
 \node at (0,0.3) {\footnotesize{structurally}};
 \node at (0,-0.3) {\footnotesize{bnd treedepth}};
 \end{scope}

 \begin{scope}[shift={(2.5,0)}]
 \draw (-2.5,0) node {$=$};
 \fill[tdcol!33,rounded corners=3] (-1.7,-0.7) rectangle (1.7,0.7);
 \node at (0,0.3) {\footnotesize{mon stable}};
 \node at (0,-0.3) {\footnotesize{bnd shrubdepth}};
 \end{scope}

 \begin{scope}[shift={(7.8,0)}]
 \draw (-3.05,0) node {$=$};
 \fill[tdcol!40,rounded corners=3] (-2.5,-0.7) rectangle (2.5,0.7);
 \node at (0,0) {bnd shrubdepth};
 \end{scope}

 \fill[pwcol!10,rounded corners=3] (-4.3,-2-0.8) rectangle (4.3,-2+0.8);
 \draw[pwcol!70,rounded corners=3] (-4.3,-2-0.8) rectangle (4.3,-2+0.8);

 \begin{scope}[shift={(-2.5,-2)}]
 \draw[very thick,<-] (0,0.8) -- (0,1.2);
  \draw[very thick,<-] (-1.9,0) -- (-3,0);
 \fill[pwcol!27,rounded corners=3] (-1.7,-0.7) rectangle (1.7,0.7);
 \node at (0,0.3) {\footnotesize{structurally}};
 \node at (0,-0.3) {\footnotesize{bnd pathwidth}};
 \end{scope}

 \begin{scope}[shift={(2.5,-2)}]
 \draw[very thick,<-] (0,0.8) -- (0,1.2);
 \draw (-2.5,0) node {$=$};
 \fill[pwcol!33,rounded corners=3] (-1.7,-0.7) rectangle (1.7,0.7);
 \node at (0,0.3) {\footnotesize{mon stable}};
 \node at (0,-0.3) {\footnotesize{bnd lin cliquewidth}};
 \end{scope}

 \fill[twcol!10,rounded corners=3] (-4.3,-4-0.8) rectangle (4.3,-4+0.8);
 \draw[twcol!70,rounded corners=3] (-4.3,-4-0.8) rectangle (4.3,-4+0.8);

 \begin{scope}[shift={(-2.5,-4)}]
 \draw[very thick,<-] (0,0.8) -- (0,1.2);
 \draw[very thick,<-] (-1.9,0) -- (-3,0);
 \fill[twcol!27,rounded corners=3] (-1.7,-0.7) rectangle (1.7,0.7);
 \node at (0,0.3) {\footnotesize{structurally}};
 \node at (0,-0.3) {\footnotesize{bnd treewidth}};
 \end{scope}

 \begin{scope}[shift={(2.5,-4)}]
 \draw[very thick,<-] (0,0.8) -- (0,1.2);
 \draw (-2.5,0) node {$=$};
 \fill[twcol!33,rounded corners=3] (-1.7,-0.7) rectangle (1.7,0.7);
 \node at (0,0.3) {\footnotesize{mon stable}};
 \node at (0,-0.3) {\footnotesize{bnd cliquewidth}};
 \end{scope}

 \fill[wwcol!10,rounded corners=3] (-4.3,-8-0.8) rectangle (4.3,-8+0.8);
 \draw[wwcol!70,rounded corners=3] (-4.3,-8-0.8) rectangle (4.3,-8+0.8);

 \begin{scope}[shift={(-2.5,-8)}]
 \fill[wwcol!27,rounded corners=3] (-1.7,-0.7) rectangle (1.7,0.7);
 \node at (0,0.3) {\footnotesize{str bnd sparse}};
 \node at (0,-0.3) {\footnotesize{twin-width}};
 \end{scope}

 \begin{scope}[shift={(2.5,-8)}]
 \draw[very thick,<-] (0,0.8) -- (0,3.1);
 \draw (-2.5,0) node {$=$};
 \fill[wwcol!33,rounded corners=3] (-1.7,-0.7) rectangle (1.7,0.7);
 \node at (0,0.3) {\footnotesize{mon stable}};
 \node at (0,-0.3) {\footnotesize{bnd twin-width}};
 \end{scope}

 \draw[very thick, dashed,<-] (-0.6,-11.7) -- (0.6,-11.7);


\end{scope}
\end{tikzpicture}

\caption{Major properties of graph classes discussed in this survey. All properties in the left-most column are weakly sparse, all properties in the remaining three columns are $\FO$ ideals. Arrows represent implications between the properties.}\label{fig:main}
\end{figure}

\paragraph*{Acknowledgements.} The author would like to thank all his coauthors with whom he shared projects on topics related to this survey. Special thanks are due to Szymon Toru\'nczyk, for years of common work on Sparsity and on Density, and for thousands of hours of discussions.

\section{Preliminaries}\label{sec:prelims}

\subsection{Graph theory}

All graphs considered in this survey are finite, undirected, and simple (with no loops or parallel edges), unless otherwise stated.
We use standard graph-theoretic notation. In particular, for a graph $G$, by $V(G)$ and $E(G)$ we  denote the vertex set and the edge set of $G$, respectively. The {\em{distance}} (length of a shortest path) between vertices $u$ and $v$ in $G$ is denoted by $\dist^G(u,v)$, and the ball of radius $r$ around $u$ in $G$ is $\Ball^G_r[u]\coloneqq \{v\in V(G)~|~\dist^G(u,v)\leq r\}$. The superscript can be omitted if it is clear from the context.

\paragraph*{Classes and parameters.}
A {\em{graph class}} is just a set of graphs, typically infinite. Examples include {\em{planar graphs}} --- graphs embeddable in the plane without crossings --- or {\em{subcubic graphs}} --- graphs with maximum degree at most $3$. A graph class $\Cc$ is {\em{monotone}} if it is closed under taking subgraphs: if $G\in \Cc$ and $H$ is a subgraph of $G$, then also $H\in \Cc$. Similarly, $\Cc$ is {\em{hereditary}} if it is closed under taking induced~subgraphs.

As the reader will find out soon, within the considered theory it is graph classes, not individual graphs, that play the role of the most basic objects of study. Therefore, we will be also speaking about {\em{properties}} of graph classes, which are just sets of graph classes. Here, an example could be the property of having {\em{bounded degree}}: a graph class $\Cc$ has bounded degree if and only if there is a universal constant $c\in \N$ such that every graph in $\Cc$ has maximum degree at most $c$. Thus, the class of subcubic graphs has bounded degree, and so has also the class of graphs of maximum degree at most $10$, but not the class of planar~graphs.

More generally, we will often use the following construction to define properties of graph classes. A {\em{graph parameter}} is just a function $\pi\colon \Graphs\to \N$, where by $\Graphs$ we denote the class of all graphs. Then the property of having {\em{bounded $\pi$}} is defined as follows: $\Cc$ has bounded $\pi$ if there exists a constant $c$ such that $\pi(G)\leq c$ for every $G\in \Cc$. We will also denote
$$\pi(\Cc)\coloneqq\sup_{G\in \Cc}\pi(G),$$
and then $\Cc$ has bounded $\pi$ if and only if $\pi(\Cc)$ is finite.

An important instantiation of the principle described above is the definition of weakly sparse classes. We say that a graph class $\Cc$ is {\em{weakly sparse}} if there exists $t\in \N$ such that no graph in $\Cc$ contains the biclique $K_{t,t}$ as a subgraph. Equivalently, if $\omega^{\#}(G)$ denotes the largest $t$ such that $G$ contains $K_{t,t}$ as a subgraph, then $\Cc$ is weakly sparse if and only if $\Cc$ has bounded $\omega^{\#}$.

\paragraph*{Minors.}
At several points we will also refer to the minor order on graphs, which is the standard notion of an embedding for graphs that has a topological character. We say that a graph $H$ is a {\em{minor}} of a graph $G$ if there is a mapping $\eta$, called the {\em{minor model}}, that maps every vertex $u$ of $H$ to a connected subgraph $\eta(u)$ of $G$ so that subgraphs $\{\eta(u)\colon u\in V(H)\}$ are pairwise vertex-disjoint and the adjacency relation is preserved in the following sense: whenever $u$ and $v$ are adjacent in $H$, in $G$ there must be an edge with one endpoint in $\eta(u)$ and the other in $\eta(v)$. The subgraph $\eta(u)$ will be often called the {\em{branch set}} of $u$. An equivalent characterization is the following: $H$ is a minor of $G$ iff $H$ can be obtained from $G$ by a repeated use of vertex deletion, edge deletion, and edge contraction (the operation of taking two adjacent vertices $u$ and $v$, and replacing them with a new vertex, adjacent to all the former neighbors of $u$ or $v$).

A {\em{subdivision}} of a graph $G$ is a graph obtained from $G$ by replacing every edge with a path (of any length). We say that $G$ contains $H$ as a {\em{topological minor}} if $G$ contains a subdivision of $H$ as a subgraph. It can be easily seen that if $H$ is a topological minor of $G$, then $H$ is also a minor of $G$, but the converse implication does not hold: for instance, subcubic graphs contain all graphs as minors, but no subcubic graph contains $K_5$ as a topological minor. A {\em{$d$-subdivision}} of $G$ is the graph obtained from $G$ by replacing every edge of $G$ with a path of length $d+1$ (thus, having $d$ internal vertices).

\subsection{Logic on graphs}

\paragraph*{Relational structures.} In order to introduce logic on graphs, it will be convenient to ground the discussion in the more general framework of relational structures. A relational structure is nothing else than a set, called the {\em{universe}}, equipped with one or more {\em{relations}} --- sets of tuples of elements. To specify what relations are present in a structure, we use the notion of a signature: a {\em{signature}} $\Sigma$ is a set of {\em{predicates}} (relation names), where every predicate $R\in \Sigma$ comes with a prescribed integer $\arity(R)\in \N$. The meaning is that $R$ may speak about tuples of $\arity(R)$ elements. For instance, if $\arity(R)=2$ then $R$ is a {\em{binary predicate}} and signifies a relation between pairs of elements. And if $\arity(R)=1$, then $R$ is a {\em{unary predicate}} that applies to single elements, and thus $R$ just selects a subset of elements of the universe.
Given a signature $\Sigma$, a {\em{$\Sigma$-structure}} $\Af$ consists of a universe $U(\Af)$ and, for each predicate $R\in \Sigma$, its {\em{interpretation}} $R^{\Af}\subseteq U(\Af)^{\arity(R)}$, which is just a set of $\arity(R)$-tuples of elements. Note that the tuples are considered ordered: for instance, if $R$ is a binary predicate, then it may happen that $(u,v)\in R^{\Af}$ but $(v,u)\notin R^{\Af}$.

The formalism of relational structures can seem abstract at first glance, so let us substantiate it with an example: the encoding of graphs. To view a graph $G$ as a relational structure, we simply take the universe to be the vertex set, and equip it with one binary relation $\adj^G(\cdot,\cdot)$ that selects those pairs of vertices that are adjacent. Thus, the signature consists of one binary predicate $\adj$. Note that a priori, we could allow $\adj^G(\cdot,\cdot)$ to be not necessarily symmetric or irreflexive; this would be suitable for encoding directed graphs, even possibly with loops. Therefore, when speaking about undirected simple graphs, we will always assume that they are encoded as above with the relation $\adj^G(\cdot,\cdot)$ being symmetric and irreflexive. This will be called the {\em{adjacency encoding}} of a graph.

The point is that the formalism of relational structures allows us to robustly equip graphs with additional features. For example, we will commonly use the concept of {\em{colored graphs}}, which are just graphs equipped with several highlighted vertex subsets, called {\em{colors}}. To model colored graphs as relational structures, we can simply extend the signature with several unary  predicates, one for each color, and use those predicates to mark the relevant subsets. Note that maybe slightly counter-intuitively, the colors are not necessarily disjoint nor saturate the vertex set: one vertex can be of multiple colors at the same time, or even of no color at all.

Finally, with every structure $\Af$ we can naturally associate a graph, called the {\em{Gaifman graph}} of $\Af$. The vertex set of the Gaifman graph of $\Af$ is the universe $U(\Af)$, and two distinct elements $u,v\in U(\Af)$ are adjacent if and only if they appear together in some tuple in some relation in $\Af$. In other words, for every tuple of every relation present in $\Af$, we put a clique on the elements present in this tuple.

\paragraph*{First-Order logic.} We may now introduce the First-Order logic $\FO$ on relational structures; this logic applied to graphs is the main point of focus in this survey. Fix a signature $\Sigma$. We now describe the logic $\FO[\Sigma]$ that speaks about $\Sigma$-structures. In this logic, we have variables for single elements of the universe, and we can (i) test predicates on those variables, (ii) make new formulas from simpler formulas using negation, conjunction, and disjunction, and (iii) introduce new variables using existential and universal quantification. More formally, $\FO[\Sigma]$ consists of formulas defined inductively as follows:
\begin{itemize}[nosep]
 \item For every predicate $R\in \Sigma$, say of arity $k$, there is an {\em{atomic formula}} $R(x_1,\ldots,x_k)$, where $x_1,\ldots,x_k$ is any $k$-tuple of variables. Thus, if $\Sigma$ is the signature of graphs, we can test adjacency of two vertices. There is also an atomic formula $x=y$ that tests whether $x$ and $y$ are evaluated to the same element.
 \item If $\alpha$ and $\beta$ are formulas, then we can construct formulas $\neg \alpha$, $\alpha\wedge \beta$, and $\alpha \vee \beta$.
 \item If $\alpha$ is a formula and $y$ is a variable, then we can construct formulas $\exists_y\, \alpha$ and $\forall_y\, \alpha$.
\end{itemize}
Note that thus, a formula may use some variables that are not bound by any quantifier. We call those {\em{free variables}} and use the notation $\varphi(\tup x)$ to signify that $\tup x$ is the tuple of free variables of $\varphi$. Thus, $\varphi$ speaks about properties of $k$-tuples of elements of the universe, where $k$ is the length of $\tup x$. A formula without free variables is a {\em{sentence}}. Thus, a sentence just expresses some property of a $\Sigma$-structure.

While the {\em{syntax}} of $\FO[\Sigma]$ is explained above, the {\em{semantics}} --- how formulas are evaluated in structures --- is as expected. That is, if $\Af$ is a structure, then:
\begin{itemize}[nosep]
 \item for an atomic formula $\phi(x,y)=(x=y)$ and elements $u,v\in U(\Af)$, we have that $\phi(u,v)$ holds in~$\Af$ iff $u=v$;
 \item for an atomic formula $R(\tup x)$ and a tuple $\tup u\in U(\Af)^{\tup x}$, we have that $R(\tup u)$ holds in $\Af$ iff $\tup u\in R^\Af$;
 \item for a formula $\phi(\tup x)=\alpha(\tup x)\wedge \beta(\tup x)$ and a tuple $\tup u\in U(\Af)^{\tup x}$, we have that $\phi(\tup u)$ holds in $\Af$ iff both $\alpha(\tup u)$ and $\beta(\tup u)$ hold, and similarly for negation and disjunction; and
 \item for a formula $\phi(\tup x)=\exists_y\, \alpha(\tup x,y)$ and a tuple $\tup u\in U(\Af)^{\tup x}$, we have that $\phi(\tup u)$ holds in $\Af$ iff there is an element $v$ such that $\alpha(\tup u,v)$ holds in $\Af$, and similarly for universal quantification.
\end{itemize}
Here, we use the notation $\tup u\in U(\Af)^{\tup x}$: formally, we consider $\tup x$ to be a set of variables, and then $\tup u$ is an {\em{evaluation}} of those variables in $U(\Af)$, which is just a function from $\tup x$ to the elements of the universe of $\Af$. If $\varphi(\tup x)$ is a formula and $\tup u\in U(\Af)^{\tup x}$ is an evaluation of its free variables, then we write $\Af\models \varphi(\tup u)$ to express that $\varphi(\tup u)$ holds in $\Af$.

Let us make a few examples to make those abstract definitions more familiar for readers with less experience with logic. In all examples, we assume the signature of (the adjacency encoding of) graphs, consisting of one binary predicate $\adj$, except for the last example, where we also use colors. More generally, throughout this article, whenever we speak about $\FO$ or any other logic without specifying the signature, we always mean the signature of the adjacency encodings of graphs.

Here is a formula with three free variables that expresses that three vertices form a triangle:
\begin{align*}
\mathsf{triangle}(x,y,z) & = \adj(x,y)\wedge \adj(y,z)\wedge \adj(z,x).
\end{align*}
And here is a sentence that states that a graph is triangle-free:
\begin{align*}
\mathsf{triangleFree} & = \neg \exists_x\, \exists_y\, \exists_z\, \mathsf{triangle}(x,y,z).
\end{align*}
Next, here are formulas that express that the distance between a pair of vertices is at most $1$ and at most~$3$, respectively:
\begin{align*}
\mathsf{dist}_{\leq 1}(x,y) & = (x=y)\vee \adj(x,y);\\
\mathsf{dist}_{\leq 3}(x,y) & = \exists_s\, \exists_t\, \mathsf{dist}_{\leq 1}(x,s)\wedge \mathsf{dist}_{\leq 1}(s,t)\wedge \mathsf{dist}_{\leq 1}(t,y).
\end{align*}
Finally, here is a sentence that works on graphs with two colors --- red and blue, highlighted by unary predicates $\mathsf{Red}$ and $\mathsf{Blue}$, respectively --- which states that the red vertices distance-$3$ dominate the blue vertices. (That is, every blue vertex is at distance at most $3$ from some red vertex.)
\begin{align*}
 \mathsf{redDominatesBlue} & = \forall_y\, \mathsf{Blue}(y)\Rightarrow \left[\exists_x\, \mathsf{Red}(x)\wedge \mathsf{dist}_{\leq 3}(x,y)\right].
\end{align*}
Note that we used implication, which can be expressed using disjunction and negation.

\paragraph*{Monadic Second-Order logic.} We now turn our attention to a more expressive logic: Monadic-Second-Order logic $\MSO$. The idea is that we extend First-Order logic with the possibility of quantifying over {\em{subsets}} of elements, which effectively means introducing new unary predicates. More precisely, if $\Sigma$ is a signature, then $\MSO[\Sigma]$ extends $\FO[\Sigma]$ by the following constructs:
\begin{itemize}[nosep]
 \item there are also variables for sets of elements, called {\em{monadic}} and denoted by convention by capital letters $X,Y,Z,\ldots$;
 \item there are atomic formulas of the form $x\in X$, where $x$ is an element variable and $X$ is a monadic variable, that check membership; and
 \item we allow quantification over monadic variables as well: if $\alpha$ is a formula, then so are $\exists_X\, \alpha$ and $\forall_X\, \alpha$.
\end{itemize}
The semantics of these constructs are as expected.

$\MSO$ over the signature of the adjacency encoding of graphs is usually called $\MSO_1$ in the literature, and in effect this boils down to extending $\FO$ on graphs by the possibility of quantifying over subsets of vertices. The reader may have also come across the more expressive logic $\MSO_2$, where we additionally allow quantification over subsets of edges. This might be a bit confusing at first glance: there is only one $\MSO$ over relational structures, but two different $\MSO$s on graphs? The answer is actually quite simple: these are not two different $\MSO$s, but two different ways of encoding a graph as a relational structure. Namely, besides the adjacency encoding, we can also consider the {\em{incidence encoding}}: the universe consists of all the vertices {\em{and}} all the edges of the graph, there are two unary relations selecting the vertices and the edges, respectively, and one binary incidence relation $\mathsf{inc}(\cdot,\cdot)$ with the following meaning: for a vertex $u$ and an edge $e$, $\mathsf{inc}(u,e)$ holds iff $u$ is an endpoint of $e$. Now, considering $\MSO$ on such relational structures boils down to allowing quantification both over subsets of vertices and over subsets of edges in~graphs.

Let us look at some examples. Below is an $\MSO_1$ sentence that verifies whether a graph is $3$-colorable. It uses an auxiliary subformula that checks whether a subset of vertices is independent (that is, induces an edgeless~subgraph).
\begin{align*}
 \mathsf{ind}(X) & = \forall_x\, \forall_y\, (x\in X\wedge y\in X)\Rightarrow \neg \adj(x,y);\\
 \mathsf{3col} & = \exists_X\, \exists_Y\, \exists_Z\, \mathsf{ind}(X)\wedge \mathsf{ind}(Y)\wedge \mathsf{ind}(Z)\wedge \left[\,\forall_x\, x\in X\vee x\in Y\vee x\in Z\,\right].
\end{align*}
Next, we give an $\MSO_2$ sentence that verifies whether a graph is Hamiltonian (contains a Hamiltonian cycle). Here, we use the following characterization: a Hamiltonian cycle is a subset of edges $F$ that connects the whole vertex set and every vertex is incident to exactly two vertices of $F$. In quantification, we use shorthands $x\in V$, $f\in E$, $X\subseteq V$, and $F\subseteq E$ to signify whether the quantification ranges over vertices, edges, vertex subsets, or edge subsets. Also, $x\notin X$ is a shorthand for $\neg (x\in X)$, and $x\neq y$ for $\neg (x=y)$.
\begin{align*}
 \mathsf{deg2}(x,F) & = \exists_{f\in E}\, \exists_{f'\in E}\, \mathsf{inc}(x,f)\wedge \mathsf{inc}(x,f')\wedge (f\neq f') \wedge \left[\forall_{f''\in E}\,\mathsf{inc}(x,f'')\Rightarrow (f''=f\vee f''=f')\right];\\
 \mathsf{conn}(F) & = \forall_{X\subseteq V}\, \left[(\exists_{x\in V}\, x\in X) \wedge (\exists_{x\in V}\, x\notin X)\right]\Rightarrow\\
 & \qquad\qquad \quad\left[\exists_{x\in V}\,\exists_{y\in V}\,\exists_{f\in E}\ x\in X\wedge y\notin X \wedge \mathsf{inc}(x,f) \wedge\mathsf{inc}(y,f)\wedge f\in F\right];\\
 \mathsf{Ham} & = \exists_{F\subseteq E}\, \left[\mathsf{conn}(F)\wedge \forall_{x\in V}\,\mathsf{deg2}(x,F)\right].
\end{align*}
For readers familiar with methods of finite model theory, it is an easy exercise to use Ehrenfeucht-Fra\"isse games to prove that 3-colorability is not expressible in $\FO$ and Hamiltonicity is not expressible in~$\MSO_1$.

We will sometimes also speak about logic $\CMSO$, which extends $\MSO$ with modular counting by allowing atomic formulas of the form $|X|\equiv a\bmod m$, where $X$ is a monadic variable, $m$ is a positive integer, and $a\in \{0,1,\ldots,m-1\}$. Thus, we may now check that a certain set is of odd size, or its size is divisible by $3$, etc. This naturally gives rise to logic $\CMSO_1$ and $\CMSO_2$. As we will see, in some contexts it is $\CMSO$ that appears to be the most natural choice of logic, rather than $\MSO$.

\paragraph*{Model-checking.} Once we have understood what kind of logic on graphs we are going to consider, we can also discuss relevant computational problems. Among those, probably the most important is the model-checking problem, defined as follows.

\begin{definition}
 Let $\cal L$ be a logic on graphs. In the {\em{model-checking}} problem for $\cal L$, we are given a graph $G$ and a sentence $\varphi$ of $\cal L$, and the question is to decide whether $G\models \varphi$.
\end{definition}

Within this survey we will consider $\cal L\in \{\FO,\MSO_1,\MSO_2,\CMSO_1,\CMSO_2\}$, but of course the problem can be stated for other logics as well.

Since the problems of deciding whether the input graph is $3$-colorable or Hamiltonian are $\mathsf{NP}$-hard, the model-checking problem for $\MSO_1$ and $\MSO_2$ is $\mathsf{NP}$-hard already for fixed formulas of constant length. The situation is a bit different for model-checking for $\FO$, as it can be solved by brute-force in time $n^{\Oh(\|\phi\|)}$, where by $\|\phi\|$ we denote the length of the sentence $\phi$. Indeed, it suffices to process $\phi$ using a recursive algorithm, and for every quantifier just branch into $n$ subproblems, in each fixing a different evaluation of the quantified variable. From the point of view of parameterized complexity, this places the model-checking problem for $\FO$ parameterized by the sentence $\phi$ in complexity class $\mathsf{XP}$: it can be solved in polynomial time for every fixed value of the parameter. A natural question is whether the problem is actually {\em{fixed-parameter tractable}} (or equivalently, belongs to the class $\mathsf{FPT}$), that is, whether it can be solved in time\footnote{In the formal definition of fixed-parameter tractability, it is usually assumed that the dependence of the running time on the parameter is bounded by a {\em{computable}} function, which here would mean that the constant hidden in the $\Oh_{\phi}(\cdot)$ notation is computable from $\phi$. For simplicity, within this survey we brush such computability issues under the rug, as ultimately they play a secondary role.} $\Oh_{\phi}(n^c)$ for some universal constant $c$, independent of $\varphi$. It turns out that the model-checking problem for $\FO$ is complete for the class $\mathsf{AW}[\star]$, so it being solvable in fixed-parameter time would yield $\mathsf{FPT}=\mathsf{AW}[\star]$ and trigger a collapse of a large part of the hierarchy of parameterized complexity classes (comparable to the collapse $\mathsf{P}=\mathsf{PSPACE}$ in the classic complexity theory). See \cite[Section~8.6]{FlumG06} for a broader introduction to the topic.

The discussion above shows that the model-checking problem for all the logics considered is very hard on general graphs, at least from the perspective of parameterized algorithms. Therefore, we can consider its restriction to various graph classes --- that is, stipulate the input graph to come from a graph class $\Cc$ --- and ask how this affects the complexity. This leads us to the following central question.

\begin{quote}
 Characterize graph classes $\Cc$ for which the model-checking problem for $\FO$ restricted to $\Cc$ is fixed-parameter tractable. Also, give analogous characterizations for $\MSO_1$ and $\MSO_2$.
\end{quote}

It appears that for $\MSO_1$ and $\MSO_2$, the question above is for the most part understood. While, the complete answer for $\FO$ is still lacking, the recent investigations have lead to an exciting theory, which is still under development. Therefore, answering the question above should not really be regarded as the ultimate goal, but rather as an excuse for developing a new part of structural graph theory, from which a resolution will hopefully follow.

\subsection{Transductions}

Finally, we discuss {\em{transductions}}: the main embedding notion that will be of interest in this survey. We focus on the case of $\FO$ first, and then discuss how the notion can be adapted to other considered logics.

First, let us introduce a simpler notion of {\em{simple interpretations}}. Suppose $\phi(x,y)$ is a formula of $\FO$ on graphs with two free variables. Then for a graph $G$, we can apply $\phi$ to the whole $G$ and obtain a new graph $\phi(G)$ defined as follows:
\begin{itemize}[nosep]
 \item the vertex set of $\phi(G)$ is the same as that of $G$; and
 \item two distinct vertices $u,v\in V(G)$ are adjacent in $\phi(G)$ if and only if $G\models \phi(u,v)$ or $G\models \phi(v,u)$. (The disjunction is only to make the graph undirected.)
\end{itemize}
We say that $\phi(G)$ is {\em{interpreted}} in $G$ using $\phi$. Thus, a simple interpretation just changes the edge set, by placing the new edges where the formula $\phi$ tells so. We can easily generalize this notion to relational structures as follows:
\begin{itemize}[nosep]
 \item a simple interpretation from $\Sigma$-structures to $\Gamma$-structures consists of one formula $\varphi_R(\tup x)\in \FO[\Sigma]$ for each predicate $R\in \Gamma$, where $\tup x$ is a tuple of variables of length $\arity(R)$; and
 \item the tuples selected by $\varphi_R(\tup x)$ in the input $\Sigma$-structure form the relation $R$ in the output $\Gamma$-structure.
\end{itemize}

Let us see some examples. If $\phi(x,y)=\neg \adj(x,y)$, then $\phi(G)$ is the {\em{complement}} of $G$: the graph obtained from $G$ by swapping all the edges with non-edges. Further, if $\phi(x,y)=\mathsf{dist}_{\leq 2}(x,y)$, the formula checking whether the distance between $x$ and $y$ is at most $2$, then $\phi(G)$ is the {\em{square}} of $G$: the graph obtained from $G$ by making all pairs of vertices at distance at most $2$ adjacent.

The intuition is that the interpretation $\phi$ allows us to encode the graph $\phi(G)$ within $G$, which means that, in a sense, the ``logical space'' of $\phi(G)$ embeds into the ``logical space'' of $G$. This intuition is supported by the following simple lemma: $\FO$-definitions of properties of $\phi(G)$ can be translated back to $\FO$-definitions working on $G$.

\begin{lemma}[Backwards Translation Lemma]\label{lem:backwards}
 Let $\varphi(x,y)\in \FO$ be a formula and $\psi\in \FO$ a sentence. Then there is a sentence $\psi[\phi]\in \FO$ such that for every graph $G$,
 $$G\models \psi[\phi]\quad\textrm{if and only if}\quad  \phi(G)\models \psi.$$
\end{lemma}
\begin{proof}
 To obtain $\psi[\phi]$, it suffices to modify $\psi$ by replacing every occurrence of the predicate $\adj(x,y)$ with the formula $\phi(x,y)\vee \phi(y,x)$.
\end{proof}

\newcommand{\Tf}{{\mathsf T}}
\newcommand{\tleq}{\sqsubseteq}
\newcommand{\tgeq}{\sqsupseteq}

We now proceed to the notion that is of main interest to us: transductions.
A {\em{transduction}} $\Tf$ consists of a set of colors $C$ and a formula $\phi(x,y)$ of $\FO$ on $C$-colored graphs. (That is, the signature consists of the binary predicate $\adj$ and a unary predicate for every color of $C$.) Applying $\Tf$ to a graph $G$ yields the set $\Tf(G)$ consisting of all graphs that can be obtained from $G$ as follows:
\begin{itemize}[nosep]
 \item {\em{coloring:}} add the colors of $C$ to $G$ in an arbitrary way, thus obtaining a $C$-colored graph $G^+$;
 \item {\em{interpretation:}} apply a simple interpretation using formula $\varphi$ to $G^+$, thus obtaining an (uncolored) graph $\varphi(G^+)$; and
 \item {\em{restriction:}} output an arbitrary induced subgraph of $\varphi(G^+)$.
\end{itemize}
Thus, $\Tf$ can be regarded as a nondeterministic mechanism that takes $G$ on input, and gives any of the graphs of $\Tf(G)$ on output. It should not really be surprising that the notion is nondeterministic in nature. One graph may have multiple subgraphs, or induced subgraphs, or minors. So also it may yield multiple different outputs under a transduction, which is a mechanism of transforming one graph into another.

Now, transductions can be applied to classes of graphs, yielding new classes. More precisely, for a graph class $\Cc$ and a transduction $\Tf$, we define
$$\Tf(\Cc)\coloneqq \bigcup_{G\in \Cc} \Tf(G).$$
With this in mind, we can introduce the quasi-order defined by transductions.

\begin{definition}
 We say that a graph class $\Dd$ is {\em{transducible}} from $\Cc$ if there exists a transduction $\Tf$ such that
 $$\Dd\subseteq \Tf(\Cc).$$
 We also say that $\Cc$ {\em{transduces}} $\Dd$, and we denote this by $\Dd\tleq_\FO \Cc$.
\end{definition}

The right way to read this definition is the following: every graph from $\Dd$ can be encoded in a colored graph from $\Cc$ using a fixed $\FO$-definable encoding mechanism. Indeed, the Backwards Translation Lemma (\cref{lem:backwards}) allows us to translate $\FO$-definitions of properties of graphs from $\Dd$ to colored graphs from $\Cc$.
Thus, the transductions are not really an embedding notion for single graphs, but rather for {\em{graph classes}}. This fits well the premise that it is a graph class, and not a single graph, that is the main object of study.

Given the notation, it is natural to expect that $\tleq_\FO$ is a quasi-order, that is, a reflexive and transitive relation on classes of graphs. Reflexivity is trivial, while transitivity follows from the compositionality of transductions, which we argue next.

\begin{lemma}\label{lem:trans-composition}
 A composition of two transductions is again a transduction.
\end{lemma}
\begin{proof}
 Let $(C,\phi)$ and $(D,\psi)$ be the two considered transductions; we may assume $C\cap D=\emptyset$. It suffices to take the transduction $(C\cup D\cup \{W\},\psi')$, where $W$ is a new color not belonging to $C\cup D$, and $\psi'$ is the formula obtained from $\psi$ by restricting every quantification to the vertices of color $W$ and replacing every occurrence of the predicate $\adj(x,y)$ with the formula $\phi(x,y)\vee \phi(y,x)$. Here, the meaning of color $W$ is that it selects the set of vertices to which the first transduction decides to trim the vertex set.
\end{proof}

Finally, having a quasi-order allows us to define {\em{ideals}}, that is, downward-closed sets.

\begin{definition}
 A property of graph classes $\Pi$ is called an {\em{$\FO$ ideal}} if it is closed under transductions: for every transduction $\Tf$ and graph class $\Cc\in \Pi$, we also have $\Tf(\Cc)\in \Pi$.
\end{definition}

We will discuss multiple concrete $\FO$ ideals throughout this survey. In fact, all concepts depicted in \cref{fig:main} except for the ones in the left-most column are $\FO$ ideals.

\medskip

Let us make an example to illustrate the abstract notions we have just introduced. The class of {\em{rook graphs}} is defined as follows: a rook graph has vertex set $\{1,\ldots,a\}\times \{1,\ldots,b\}$ for some $a,b\in \N$, and two distinct vertices $(i,j)$ and $(i',j')$ are adjacent iff $i=i'$ or $j=j'$. We prove the~following.

\begin{lemma}\label{lem:rook-not-dependent}
 The class of rook graphs transduces the class of all graphs.
\end{lemma}
\begin{proof}
 We first argue that the class of all bipartite graphs can be transduced from the class of rook graphs. Take any bipartite graph $H$, say with sides $A$ and $B$ with $|A|=a$ and $|B|=b$. Consider the rook graph $G$ on vertex set $\{1,\ldots,a+1\}\times \{1,\ldots,b+1\}$. We transduce $H$ from $G$ as follows (see \cref{fig:rook}):
 \begin{itemize}[nosep]
  \item Highlight $A'=\{1,\ldots,a\}\times \{b+1\}$ and $B'=\{a+1\}\times \{1,\ldots,b\}$ using two colors, $A'$ and $B'$, respectively.
  \item Arbitrarily enumerate $A$ as $\{u_1,\ldots,u_a\}$ and $B$ as $\{v_1,\ldots,v_b\}$, and highlight the adjacency matrix of $H$ within $\{1,\ldots,a\}\times \{1,\ldots,b\}$ using a color $F$. That is, include a vertex $(i,j)$ in $F$ if and only if $u_iv_j$ is an edge in $H$.
  \item Apply interpretation with a formula $\phi(x,y)=A'(x)\wedge B'(y)\wedge \exists_z\,[\adj(x,z)\wedge \adj(y,z)\wedge F(z)]$. That is, put a new edge between vertices $x$ and $y$ iff $x$ is in $A'$, $y$ is in $B'$, and $x$ and $y$ have a common neighbor in $F$.
  \item Output the subgraph induced by $A'\cup B'$. It is isomorphic to $H$.
 \end{itemize}
 Clearly, the above mechanism is captured by a transduction that uses a set of three colors $C=\{A',B',F\}$ and formula $\phi$ for interpretation. The construction above shows that the image of this transduction on the class of rook graphs contains all bipartite graphs.

 \begin{figure}
  \centering
  \begin{tikzpicture}
   \node at (-4,0) {\includegraphics[scale=0.4]{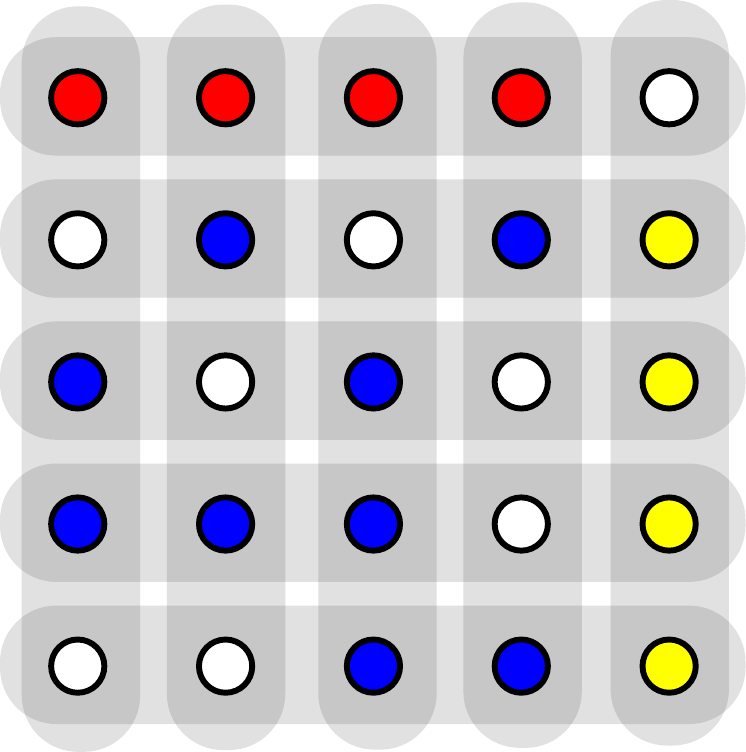}};
   \node at ( 4,0) {\includegraphics[scale=0.4]{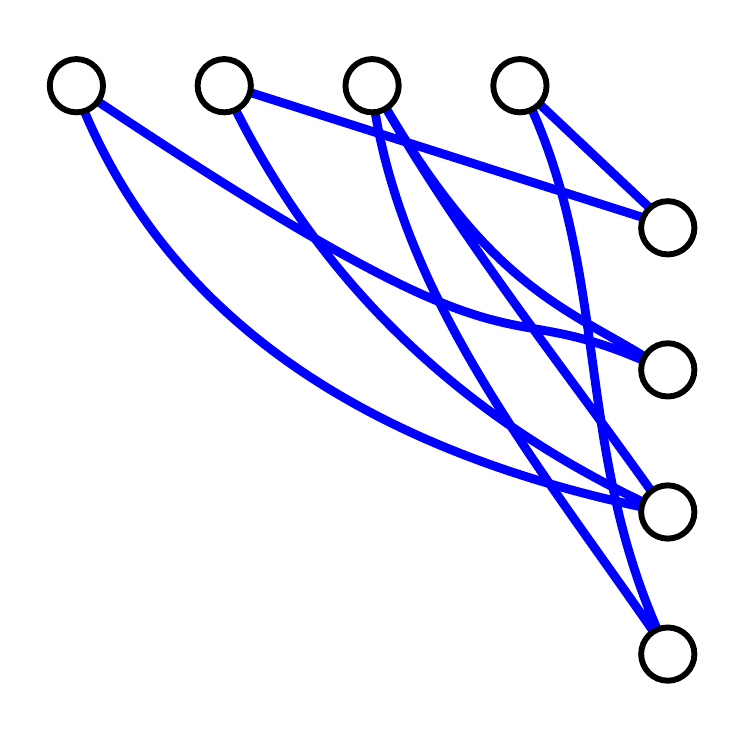}};
  \end{tikzpicture}
  \caption{Transduction from the proof of \cref{lem:rook-not-dependent}. Left: a rook graph with colors $A'$, $B'$, and $F$ depicted in red, yellow, and blue, respectively. Note that every row and every column is a clique, which is depicted using gray ovals. Right: the transduced bipartite graph.}\label{fig:rook}
 \end{figure}

 Now, the class of all graphs can be easily transduced from the class of bipartite graphs as follows: every graph $G$ can be transduced from its $1$-subdivision (which is bipartite) by taking the square and removing all the subdivision vertices. By \cref{lem:trans-composition}, the composition of the two transductions described above is again a transduction, and we are done.
\end{proof}

While the discussion above focuses on transducing graphs from graphs, we could easily extend the definitions to capture also transductions from $\Sigma$-structures to $\Gamma$-structures, for any signatures $\Sigma$ and $\Gamma$. Then the first coloring step would add some unary predicates $C$ to the structure in an arbitrary way, the second interpretation step would apply an interpretation with input signature $\Sigma\cup C$ to output signature~$\Gamma$, and the last restriction step would output any induced substructure. Also, we could replace $\FO$ with any other reasonable logic~$\cal L$, for instance $\MSO$ or $\CMSO$, thus obtaining the notion of $\cal L$-transductions that use formulas of $\cal L$ for interpretation. (We follow the convention that when we discuss transductions without specifying the logic, we mean $\FO$-transductions.) We could also talk about $\MSO_1$- or $\MSO_2$-transductions for graph classes, but in this context, for the sake of clarity, we will explicitly speak about $\MSO$-transductions that input or output the adjacency encodings or the incidence encodings of graphs. When the encoding is not specified, we always mean the adjacency encoding.

We conclude this section by making a few side remarks about transductions.
\begin{itemize}[nosep]
 \item The choice of transductions as the basic notion of embedding should not be regarded as the only correct one. For instance, the notion of logical intepretations is commonly used in finite model theory. They are deterministic and, in particular, do not include the coloring step. One could imagine constructing a different theory by adopting a different concept of logical encodability of structures in structures. However, it seems that the notion of transductions described above leads to a particularly elegant structure theory.
 \item Related to the above, readers with background in finite model theory could have been surprised that the last restriction step involves just taking an induced subgraph. In the standard definition of an interpretation, there would be another formula $\phi_{\mathsf{dom}}(x)$ with one free variable, and the restriction step would trim the universe to those vertices that satisfy $\phi_{\mathsf{dom}}$. In our setting we do not really need this additional formula thanks to the coloring step: essentially, one can always guess the set of elements to which the universe will be restricted using an additional color.
 \item Again, in finite model theory it is common to consider interpretations that interpret elements of the output structure in {\em{tuples}} of elements of the input structure, allowing a polynomial blow-up of the size of the universe. Interpretations used in our notion of transductions are one-dimensional, in the sense that elements of the output structure are interpreted in single elements of the input structure. We make this restriction, because allowing both two-dimensional interpretations and colorings at the same time would allow constructing a Cartesian product of a set followed by guessing the adjacency matrix of any graph in this product, similarly as in the proof \cref{lem:rook-not-dependent}. This would make the class of all graphs transducible from edgeless graphs, making the notion trivial. A different sensible choice would be to allow multi-dimensional interpretations and disallow colorings. We are not aware of work going in this direction on the grounds of structural graph theory.
 \item It is indeed a limitation of our definition of transductions that they can only remove vertices, and are not capable of creating new ones. Therefore, sometimes it is assumed that transductions can also perform {\em{copying}}, as an additional step before coloring. This works as follows: at the very beginning, we replace the input graph $G$ with the disjoint union of $c$ copies of $G$, for some constant $c$, and copies of the same original vertex are bound by a new binary predicate $\mathsf{copy}(\cdot,\cdot)$. This copied structure is then passed to the coloring step and further. Thus, transductions with copying can increase the size of the universe, but only by a constant multiplicative factor. In essence, a transduction with copying is almost the same as a transduction without copying that is provided on input the Cartesian product of the input graph $G$ with the complete graph $K_c$. If we study transducibility from classes satisfying some property $\Pi$, and $\Pi$ is closed under taking Cartesian products with constant-size complete graphs (that is, if $\Cc\in \Pi$, then also $\Cc\square K_c\coloneqq \{G\square K_c\colon G\in \Cc\}\in \Pi$ for every $c\in \N$), then the two notions of transducibility coincide. This will be the case in all the relevant statements presented in this survey, hence for simplicity we stick to transductions without copying.
\end{itemize}

\section{Classic concepts}\label{sec:classic}

In this section we present more standard notions of structure in graphs, which in large part originate from classic structural graph theory. However, we focus our discussion on model-theoretic properties. We describe the concepts starting from the most restrictive, which in \cref{fig:main} means starting from the top.

\newcommand{\td}{\mathsf{td}}
\newcommand{\SCd}{\mathsf{SCd}}
\newcommand{\lp}{\mathsf{lp}}

\subsection{Star-like graphs: treedepth and shrubdepth}

We start with parameters {\em{treedepth}} and {\em{shrubdepth}}, which, roughly speaking, measure star-likeness of a graph, or its resemblance to a bounded-depth tree. Treedepth is the more classic notion that is suited for the treatment of sparse graphs. The name {\em{treedepth}} was first used by Ne\v{s}et\v{r}il and Ossona de Mendez in~\cite{NesetrilM06}, the same notion actually appeared under different disguises well before that; see for instance the discussion and references in~\cite[Section 6.1]{sparsity} and~\cite{ReidlRVS14}. Shrubdepth is more general and captures star-likeness in dense graphs. It was introduced by Ganian, Hlin\v{e}n\'y, Ne\v{s}et\v{r}il, Obdr\v{z}\'alek, and Ossona de Mendez in~\cite{GanianHNOM19}\footnote{In fact, shrubdepth was first defined in an earlier conference paper by Ganian, Hlin\v{e}n\'y, Ne\v{s}et\v{r}il, Obdr\v{z}\'alek, Ossona de Mendez, and Ramadurai~\cite{GanianHNOMR12}, on which the journal article~\cite{GanianHNOM19} is largely based.}.

\subsubsection{Treedepth}\label{sec:treedepth}

The decomposition notion underlying treedepth is called an {\em{elimination forest}}, though other names, like simply {\em{treedepth decomposition}}, can be also found in the literature. In the following, a {\em{rooted forest}} is a forest where in every connected component we distinguish one root vertex; this naturally imposes an ancestor/descendant relation on the vertices. The {\em{depth}} of a rooted forest is the maximum number of vertices on a root-to-leaf path.

\begin{definition}
 An {\em{elimination forest}} of a graph $G$ is a rooted forest on the same vertex set as $G$, i.e. $V(F)=V(G)$, such that for every edge $uv$ of $G$, $u$ and $v$ are in the ancestor/descendant relation in $F$. The {\em{treedepth}} of $G$, denoted $\td(G)$, is the minimum depth of an elimination forest of $G$.
\end{definition}

See \cref{fig:td-shb}, left panel, for an example.

 \begin{figure}
  \centering
  \begin{tikzpicture}
   \node at (-4.5,0.45) {\includegraphics[scale=0.4]{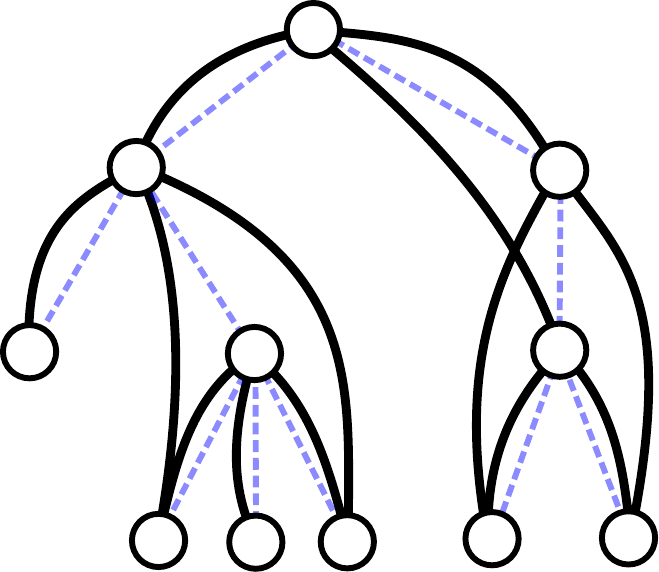}};
   \node at ( 4,0) {\includegraphics[scale=0.4]{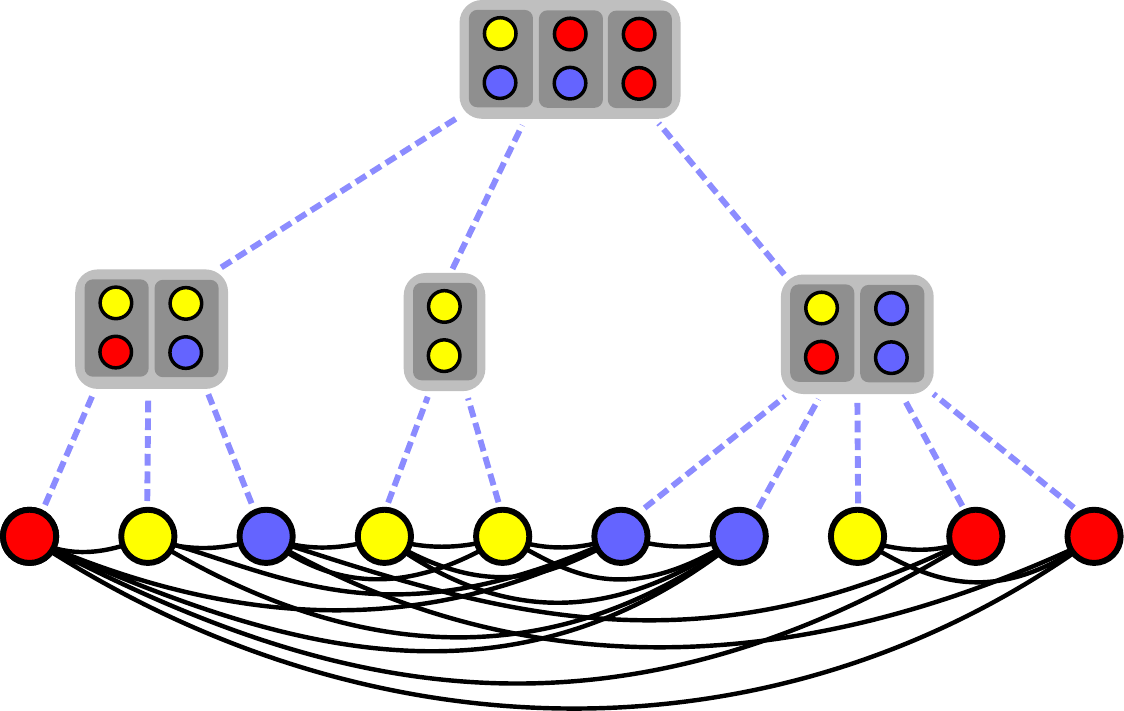}};
  \end{tikzpicture}
  \caption{Left: A graph (in black) together with its elimination tree of depth $4$ (light blue). Right: a graph (in black) together with its tree-model of depth $3$ that uses $3$ labels. The leaves coincide with the vertices of the graph, their labels are depicted with colors (yellow, red, or blue). The internal nodes are presented as boxes containing those pairs of labels that are mapped by the corresponding function $M_t$ to $1$.}\label{fig:td-shb}
 \end{figure}

Observe that in case $G$ is connected, every elimination forest of $G$ must in fact be a tree. Indeed, otherwise there could be no edges of $G$ between the different connected components of $F$. Therefore, in this case we can speak about {\em{elimination trees}}.

This observation leads to the following alternative, recursive definition of treedepth.

\begin{lemma}\label{lem:td-recursive}
 For a graph $G$, we have
 $$\td(G)=\begin{cases}1, &\textrm{if }G\textrm{ has one vertex;}\\ \max\, \{\,\td(C)\colon C\textrm{ is a component of }G\,\}, & \textrm{if }G\textrm{ is disconnected; and}\\ 1+\min_{u\in V(G)}\td(G-u),& \textrm{otherwise.}\end{cases}$$
\end{lemma}
\begin{proof}
That $\td(G)=1$ when $G$ has one vertex is obvious. To see that the treedepth of a disconnected graph $G$ is the maximum among the treedepth of its connected components $C$, note that on one hand, an elimination forest of $G$ can be obtained by taking the disjoint union of optimum-depth elimination forests of the components~$C$, and on the other hand, it is easy to see that the treedepth of a subgraph of $G$ cannot be larger than the treedepth of $G$.

 Finally, suppose that $G$ is connected and has more than one vertex. Then if $F$ is an optimum-depth elimination tree of $G$ and $u$ is the root of $G$, then $F-u$ is an elimination forest of $G-u$ and the depth of $F-u$ is one smaller than the depth of $F$; this proves that $1+\min_{u\in V(G)} \td(G-u)\leq \td(G)$. On the other hand, if $u$ is any vertex of $G$ and $F'$ is an elimination forest of $G-u$, then an elimination tree of $G$ of depth one larger than that of $F'$ can be obtained by adding $u$ to $F'$ as the new root and making all the former roots into children of $u$. This proves that $1+\min_{u\in V(G)}\td(G-u)\geq \td(G)$.
\end{proof}

The alternation of $\min$ and $\max$ operators in \cref{lem:td-recursive} suggests a definition of treedepth through a game, which we will call the {\em{radius-$\infty$ Splitter Game}}. The game is played on a graph $G$, called the {\em{arena}}. There are two players: {\em{Splitter}} and {\em{Connector}}. The game proceeds in rounds, where each round is as~follows:
\begin{itemize}[nosep]
 \item First, Connector picks any connected component $C$ of the arena, and the game is restricted to this component. That is, the arena becomes $C$.
 \item Then, Splitter removes any vertex from the arena.
\end{itemize}
The game finishes with Splitter's win when the arena becomes empty. The goal of Splitter is to win the game as quickly as possible, and the goal of Connector is to prevent this for as long as possible.

A simple inductive argument based on \cref{lem:td-recursive} proves the following.

\begin{lemma}\label{lem:td-splitter}
 For a graph $G$, $\td(G)$ is the minimum number of rounds needed for Splitter to win the radius-$\infty$ Splitter game on $G$.
\end{lemma}

The ``radius-$\infty$'' prefix indicates that in this variant of the Splitter Game, Connector selects a connected component, or equivalently, a subgraph induced by vertices reachable from some selected vertex by a path of arbitrary length. As the reader might suspect, we will later consider also the radius-$r$ Splitter Game, where the moves of the Connector are more constrained: in every round, she picks a ball of radius $r$, and the arena gets restricted to the subgraph induced by this ball. It turns out that this variant of Splitter Game provides an extremely useful notion of a decomposition for general classes of sparse graphs (called {\em{nowhere dense}} classes, see \cref{sec:sparsity}). Hence, the game-theoretic view on treedepth exemplified by \cref{lem:td-splitter} is actually very consequential.

We now turn our attention to obstructions dual to low-depth elimination forests, which turn out to be very simple: they are just paths. More precisely, if by $\lp(G)$ we denote the number of vertices in the longest path that is contained in $G$ as a subgraph, then we have the following.

\begin{lemma}\label{lem:lp-td}
 For every graph $G$, we have $\log_2 \lp(G)\leq \td(G)\leq \lp(G)$.
\end{lemma}
\begin{proof}
 The right inequality follows from the observation that every depth-first search forest of $G$ is also an elimination forest, and such a forest cannot have depth larger than $d$ in the absence of paths on more than $d$ vertices. For the left inequality, it suffices to observe that the treedepth of a path on $n$ vertices is at least $\log_2 (n+1)$. This can be proved by giving a strategy for Connector in the radius-$\infty$ Splitter Game on a path: once Splitter breaks the current path into two subpaths by deleting one vertex, always select the longer among the two subpaths.
\end{proof}

The proof of \cref{lem:lp-td} provides a very simple approximation algorithm for treedepth: just run depth-first search and output the obtained forest of the search; its depth is surely bounded by $2^{\td(G)}$. More generally, the treedepth of a graph $G$ can be computed exactly in fixed-parameter time $2^{\Oh(\td(G)^2)}\cdot n$~\cite{ReidlRVS14}. Obtaining a constant-factor approximation algorithm running in time $2^{o(\td(G)^2)}\cdot n^c$, for any constant $c\in \N$, remains a notorious open problem.

\subsubsection{Shrubdepth}\label{sec:shrubdepth}

We now move to shrubdepth and recall the definitions proposed by Ganian et al.~\cite{GanianHNOM19}.
The notion of decomposition is provided by {\em{tree-models}} explained below. Readers familiar with {\em{cliquewidth}} (see \cref{sec:cliquewidth}) should immediately recognize similarities with decompositions considered there, only here we stipulate tree-models to have bounded depth.

\begin{definition}\label{def:tree-model}
 A {\em{tree-model}} of a graph $G$ consists of a finite set of labels $\Lambda$, a labelling $\lambda\colon V(G)\to \Lambda$, a rooted tree $T$ whose leaf set is equal to the vertex set of $G$, and, for every non-leaf node $t$ of $T$, a symmetric function $M_t\colon \Lambda\times \Lambda\to \{0,1\}$ (that is, $M_t(\alpha,\beta)=M_t(\beta,\alpha)$ for all $\alpha,\beta\in \Lambda$). We require that for any two distinct vertices $u$ and $v$ of $G$, we have
 $$u\textrm{ and }v\textrm{ are adjacent}\qquad\textrm{if and only if}\qquad M_t(\lambda(u),\lambda(v))=1,$$
 where $t$ is the lowest common ancestor of $u$ and $v$ in $T$.
\end{definition}

See \cref{fig:td-shb}, right panel, for an example.

We remark that in~\cite{GanianHNOM19}, tree-models are defined slightly differently: in essence, it is required that all functions $M_t$ on each level of the tree are the same. It is not hard to prove that the two definitions are equivalent in the sense that one variant of a tree-model can be transformed into the other so that the number of labels increases to a function of the original number, but the depth is preserved exactly. This will have no influence on our further discussion, so we will stick to the formulation presented in \cref{def:tree-model}, as we find it more natural and closer to cliquewidth.

With tree-models understood, we can define shrubdepth.

\begin{definition}\label{def:shrubdepth}
 The {\em{shrubdepth}} of a graph class $\Cc$ is the least $d\in \N$ such that the following holds: there is $m\in \N$ such that every graph in $\Cc$ admits a tree-model of depth at most $d$ and using a label set of size at most $m$. If no such integer exists, then we say that $\Cc$ has {\em{unbounded shrubdepth}}, and otherwise it has {\em{bounded shrubdepth}}.
\end{definition}

A slightly confusing aspect of this definition is that shrubdepth is defined only for graph classes, and not for individual graphs. The reason for this choice will become clear in a moment. Nevertheless, it is instructive to think of tree-models as inherently biparametric objects: the two parameters governing the complexity of a tree-model are the depth and the number of labels used. In the definition of shrubdepth, the depth becomes the primary parameter of concern, while the number of labels serves a secondary role.

\paragraph*{SC-depth and Flipper Game.}
It is possible to define graph parameters that are applicable to single graphs, while being equivalent to shrubdepth in the following sense: they are bounded on the same classes of graphs. One such parameter is {\em{SC-depth}}, where {\em{SC}} stands for {\em{subset complementation}}. SC-depth was also introduced by Ganian et al. in~\cite{GanianHNOM19}.

 \begin{figure}
  \centering
  \begin{tikzpicture}
   \node at (-4,0) {\includegraphics[scale=0.4]{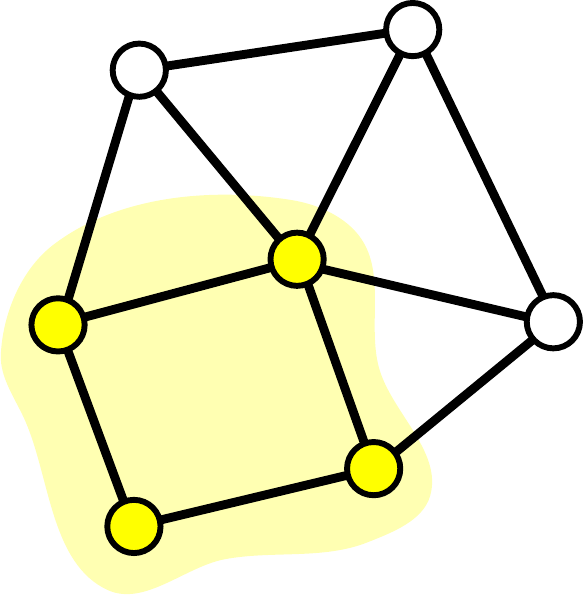}};
   \draw[very thick,black!50,dashed,->] (-1,0) -- (1,0);
   \node at ( 4,0) {\includegraphics[scale=0.4]{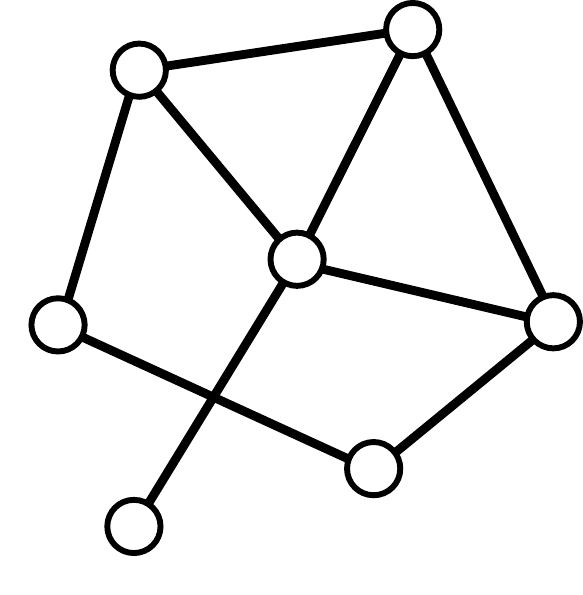}};
  \end{tikzpicture}
  \caption{Applying a flip on a set $A$, depicted in yellow.}\label{fig:flip}
 \end{figure}

First, we describe the operation of applying a {\em{flip}} to a graph (this operation is also called {\em{subset complementation}} or {\em{perturbation}}). Given a graph $G$ and a set of vertices $A$, applying a flip on $A$ in $G$ results in the graph $G\oplus A$ obtained from $G$ by inverting the adjacency relation within~$A$. That is, the vertex sets of $G$ and $G\oplus A$ are the same, and for every pair of distinct vertices $u$ and~$v$ we have the following:
\begin{itemize}[nosep]
 \item if $u,v\in A$, then $u$ and $v$ are adjacent in $G\oplus A$ iff they are {\em{not}} adjacent in $G$; and
 \item otherwise, $u$ and $v$ are adjacent in $G\oplus A$ iff they are adjacent in $G$.
\end{itemize}
With this definition, the SC-depth of a graph is defined through a recurrence similar to that of \cref{lem:td-recursive}.

\begin{definition}\label{def:SC-depth}
 For a graph $G$, the {\em{SC-depth}} of $G$, denoted by $\SCd(G)$, is defined as follows:
 $$\SCd(G)=\begin{cases}0, &\textrm{if }G\textrm{ has one vertex;}\\ \max\, \{\,\SCd(C)\colon C\textrm{ is a component of }G\,\}, & \textrm{if }G\textrm{ is disconnected; and}\\ 1+\min_{A\subseteq V(G)}\SCd(G\oplus A),& \textrm{otherwise.}\end{cases}$$
\end{definition}

Then we have the following.

\begin{theorem}[\cite{GanianHNOM19}]
 A graph class has bounded shrubdepth if and only if it has bounded SC-depth.
\end{theorem}

Observe that similarly to the characterization of \cref{lem:td-recursive}, \cref{def:SC-depth} can be understood through a game, which in this context we call the {\em{radius-$\infty$ Flipper Game}}. Now, the players are called {\em{Flipper}} and {\em{Keeper}}, and the game proceeds in round as follows:
\begin{itemize}[nosep]
 \item First, Keeper picks a connected component and the arena gets restricted to this component. (So Keeper's moves are exactly the same as Connector's in the Splitter Game.)
 \item Next, Flipper chooses a vertex subset $A$ and applies a flip on $A$.
\end{itemize}
The game finishes once the arena is restricted to a single vertex. The goals of the players are the same. The following statement is then obtained in exact analogy to \cref{lem:td-splitter}.

\begin{lemma}\label{lem:SC-flipper}
 For a graph $G$, $\SCd(G)$ is the minimum number of rounds needed for Flipper to win the radius-$\infty$ Flipper Game on $G$.
\end{lemma}

Let us pause for a moment, because the reasoning presented above is a simple instantiation of a very important general principle, which seems to permeate the whole theory. In essence, the definitions of SC-depth and of the Flipper Game have been obtained by taking the definitions of treedepth and of the Splitter Game, and replacing the concept of {\em{deleting a vertex}} with the concept of {\em{applying a flip}}. We will see this idea applied repeatedly later on, for more complicated definitions: the flip operations seems to be the right analogue of vertex deletion in the context of dense graphs, and by following this analogy we can design structural concepts for dense graphs mirroring those working on sparse graphs.

We remark that sometimes one assumes a slightly more general notion of a flip, where we are given two vertex subsets $A$ and $B$ in a graph $G$, not necessarily disjoint, and the operation consists of inverting the adjacency relation in $A\times B$. This change is non-consequential for theory: these stronger flips can model the standard flips described above by taking $A=B$, and any stronger flip can be emulated using three standard flips: on $A$, on $B$, and on $A\cup B$.

Also, note that making a flip can be modelled by a transduction: there is a transduction that given $G$, outputs all graphs that can be obtained from $G$ by applying a flip. More generally, for every fixed $k\in \N$, there is a transduction $\Tf_k$ such that $\Tf_k(G)$ consists of all {\em{$k$-flips}} of $G$: graphs that can be obtained from $G$ by applying at most $k$ flips. This may explain why the flip operation is so relevant in the constructed structural theory based on transductions.

\paragraph*{Model-theoretic aspects.} Let us proceed with the discussion of the properties of shrubdepth. The following model-theoretic characterization was in fact the main motivation of Ganian et al. for the introduction of this notion: from the point of view of logic, classes of bounded shrubdepth are those encodable in bounded-depth trees. (Here, we see rooted trees as binary structures over the signature consising of the ancestor relation.)

\begin{theorem}[\cite{GanianHNOM19}]\label{thm:shrub-logic}
 Let $\Cc$ be a graph class. Then $\Cc$ has shrubdepth at most $d$ if and only if $\Cc$ can be transduced from the class of trees of depth at most $d$. In particular, $\Cc$ has bounded shrubdepth if and only if $\Cc$ can be transduced from a class of trees of bounded depth.
\end{theorem}

We note that in fact, in~\cite{GanianHNOM19} Ganian et al. proved \cref{thm:shrub-logic} for $\MSO$ and $\CMSO$ transductions, which is a stronger result. Indeed, the right-to-left implication for $\mathsf{(C)MSO}$ transductions entails this implication for $\FO$ transductions as well, while the left-to-right implication is anyway easy, as decoding the graph from its tree-model of bounded depth and using a bounded number of labels can be easily done using an $\FO$ transduction. This is connected to the collapse of $\MSO$ to $\FO$ on classes of bounded shrubdepth, see \cref{thm:shrub-FO-MSO}.

The equivalence provided by \cref{thm:shrub-logic} together with compositionality of transductions (\cref{lem:trans-composition}) immediately implies the following.

\begin{theorem}\label{cor:shrub-ideal}
 For every $d\in \N$, classes of shrubdepth at most $d$ form an $\FO$ ideal. In particular, classes of bounded shrubdepth are an $\FO$ ideal as well.
\end{theorem}

\cref{cor:shrub-ideal} explains the choices made in the definition of shrubdepth. Namely, behind shrubdepth there is a hierarchy of strictly increasing $\FO$ ideals, and the depth parameter precisely pinpoints where are the borders between those ideals. Thus, highlighting the depth of a tree-model as the primary parameter is completely natural from the point of view of model-theoretic aspects of shrubdepth, even though it makes the combinatorial definition a bit more complicated.

Further, as proved by Gajarsk\'y and Hlin\v{e}n\'y~\cite{GajarskyH15}, on classes of bounded shrubdepth the expressive powers of $\FO$ and of $\MSO$ coincide, in the sense described below. This is essentially a consequence of this collapse occurring in the setting of bounded-depth trees.

\begin{theorem}[\cite{GajarskyH15}]\label{thm:shrub-FO-MSO}
 Let $\Cc$ be a class of bounded shrubdepth. Then for every sentence $\phi\in \MSO_1$, there exists a sentence $\phi\in \FO$ that is equivalent to $\phi$ on graphs from $\Cc$:
 $$G\models \phi\quad\textrm{if and only if}\quad G\models\phi',\qquad\textrm{for every }G\in \Cc.$$
\end{theorem}

See also a different take on \cref{thm:shrub-FO-MSO} by Chen and Flum~\cite{ChenF20}, and an earlier work of Elberfeld, Grohe, and Tantau on an analogous statement for classes of bounded treedepth~\cite{ElberfeldGT16}. As expected, there the collapse applies also to $\MSO_2$, not only $\MSO_1$. In essence, \cref{thm:shrub-FO-MSO} explains why in the context of classes of bounded shrubdepth, it is equivalent to speak about $\FO$- and $\MSO$-transductions.

\paragraph*{Relation to treedepth.} As should be expected, shrubdepth generalizes treedepth, and restricting attention to sparse graphs projects classes of bounded shrubdepth to classes of bounded treedepth.

\begin{theorem}\label{thm:td-shb}
 Every class of bounded treedepth also has bounded shrubdepth. Conversely, every weakly sparse class of bounded shrubdepth in fact has bounded treedepth.
\end{theorem}
\begin{proof}
 The first statement is proved explicitly in~\cite{GanianHNOM19}. The second statement follows from the result of Galvin, Rival, and Sands~\cite{GalvinRS82} that weakly sparse classes containing arbitrarily long paths as subgraphs also contain arbitrarily long paths as induced subgraphs (see also~\cite{AtminasLR12,DuronER24,HunterMST24} for effective proofs), combined with the facts that taking the hereditary closure of a class does not increase the shrubdepth and that the class of paths has unbounded shrubdepth~\cite{GanianHNOM19}.
\end{proof}

The following model-theoretic characterization of treedepth is a consequence of combining \cref{thm:td-shb} with the characterization of \cref{thm:shrub-logic}.

\begin{theorem}\label{thm:td-model}
 A class of graphs $\Cc$ has bounded treedepth iff the incidence encodings of graphs from $\Cc$ can be transduced from a class of trees of bounded depth.
\end{theorem}
\begin{proof}
 Let $\Cc'$ be the class of Gaifman graphs of the incidence encodings of graphs from $\Cc$. Note that $\Cc'$ consists just of $1$-subdivisions of graphs from $\Cc$.

 Suppose first that $\Cc$ has bounded treedepth. Then so does the class $\Cc'$ too, for applying a $1$-subdivision to a graph can increase its treedepth by at most $1$. By \cref{thm:td-shb}, $\Cc'$ has bounded shrubdepth, so it can be transduced from a class of trees of bounded depth, due to \cref{thm:shrub-logic}. Making a graph from $\Cc'$ into an incidence encoding of a graph from $\Cc$ requires just highlighting the vertex set and the edge set using a pair of colors, which is clearly a transduction.

 Suppose now that the incidence encodings of graphs from $\Cc$ can be transduced from a class of trees of bounded depth. By dropping the predicates selecting vertices and edges, this means that $\Cc'$ also can be transduced from a class of trees of bounded depth. So $\Cc'$ has bounded shrubdepth by \cref{thm:shrub-logic}. But since $\Cc'$ is weakly sparse as well (it excludes $K_{3,3}$ as a subgraph), by \cref{thm:td-shb} we conclude that $\Cc'$ in fact has bounded treedepth. As $\Cc'$ consists of $1$-subdivisions of graphs from $\Cc$, it is now easy to argue, for example by noting that treedepth is a minor-monotone parameter, that $\Cc$ has bounded treedepth as~well.
\end{proof}

\paragraph*{Obstructions.} As we have seen in \cref{lem:lp-td}, treedepth admits very simple obstructions: paths as subgraphs. It would be tempting to conjecture that similar obstructions for shrubdepth should be paths as induced subgraphs, but this is not the case. The following class of {\em{half-graphs}} contains no graphs with an induced path on $5$ vertices, and yet its shrubdepth is unbounded.

\begin{definition}\label{def:half-graph}
 A {\em{half-graph}} of order $n$ is a bipartite graph with sides $\{a_1,\ldots,a_n\}$ and $\{b_1,\ldots,b_n\}$ such that for all $i,j\in \{1,\ldots,n\}$,
 $$a_i\textrm{ and }b_j\textrm{ are adjacent}\quad \textrm{if and only if}\quad i\leq j.$$
\end{definition}

\begin{figure}
 \centering
 \begin{tikzpicture}

    \foreach \i in {1,2,3,4,5} {
        \node[vertex] (a\i) at (\i, 0.8) {};
        \node[above] at (a\i) {$a_\i$};
        \node[vertex] (b\i) at (\i,-0.8) {};
        \node[below] at (b\i) {$b_\i$};
    }
    \foreach \i/\j in {1/1,1/2,1/3,1/4,1/5,2/2,2/3,2/4,2/5,3/3,3/4,3/5,4/4,4/5,5/5}{
        \draw[thick] (a\i) -- (b\j);
    }

 \end{tikzpicture}
 \caption{A half-graph of order $5$.}
\end{figure}
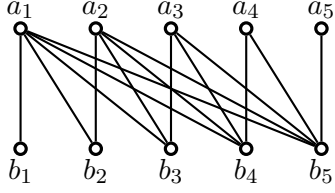

\begin{lemma}\label{lem:half-graphs-shb}
 The class of half-graphs has unbounded shrubdepth.
\end{lemma}
\begin{proof}
 It is easy to transduce the class of paths from the class of half-graphs. As proved in~\cite{GanianHNOM19}, the class of paths has unbounded shrubdepth. Since classes of bounded shrubdepth are an $\FO$ ideal (\cref{cor:shrub-ideal}), it follows that the class of half-graphs must have unbounded shrubdepth as well.
\end{proof}

The proof of \cref{lem:half-graphs-shb} suggests that transducibility of long paths might be the real problem, and not necessarily containing them as induced subgraphs. The following duality result of Ossona de Mendez, Pilipczuk, and Siebertz~\cite{OssonaPS21} confirms this suspicion, and provides a characterization of shrubdepth through logical obstructions.

\begin{theorem}[\cite{OssonaPS21}]\label{thm:shrub-obstr}
 A class of graphs $\Cc$ has bounded shrubdepth if and only if the class of all paths is not transducible from $\Cc$.
\end{theorem}

However, there is actually interesting combinatorics behind induced subgraph obstructions for shrubdepth. More precisely, Ganian et al.~\cite{GanianHNOM19} used the classic result of Ding~\cite{Ding92} to show that on every class of bounded shrubdepth, the induced subgraph order is a well quasi-order. (Recall that this means that there are no infinite antichains or descending chains.) From this we immediately get the following statement.

\begin{theorem}[\cite{GanianHNOM19}]\label{thm:shb-finite}
 For every hereditary class $\Cc$ of bounded shrubdepth, there is a finite set of graphs $\Ff$ such that for every graph $G$,
 $$G\in \Cc\quad\textrm{if and only if}\quad G\textrm{ contains no member of }\Ff\textrm{ as an induced subgraph.}$$
\end{theorem}
\begin{proof}
 It suffices to takes $\Ff$ to be the set of minimal, with respect to the induced subgraph order, graphs that do not belong to $\Cc$. Then $\Ff$ is finite, because $\Ff$ is an antichain contained in the class $\Cc'$ defined as follows: a graph $G$ belongs to $\Cc'$ iff one can remove one vertex from $G$ to obtain a graph from $\Cc$. Note here that $\Cc'$ also has bounded shrubdepth.
\end{proof}

From this we immediately conclude that hereditary classes of bounded shrubdepth are $\FO$-definable.

\begin{corollary}
 For every hereditary class $\Cc$ of bounded shrubdepth there exists a sentence $\phi_\Cc\in \FO$ such that for every graph $G$,
 $$G\in \Cc\quad\textrm{if and only if}\quad G\models \phi_\Cc.$$
\end{corollary}
\begin{proof}
 The sentence $\phi_\Cc$ just expresses that the graph does not contain any member of $\Ff$ as an induced subgraph, where $\Ff$ is the finite set of obstructions provided by \cref{thm:shb-finite}. To check whether a graph $G$ contains some $H\in \Ff$ as an induced subgraph, just quantify the vertices of $H$ existentially, and verify their distinctness and suitable (non-)adjacencies.
\end{proof}

The shape of forbidden induced obstructions (members of $\Ff$) is not really well-understood, even for ``canonical'' classes of bounded shrubdepth such as graphs admitting tree-models of depth $d$ and with $m$ labels, for fixed $d,m\in \N$. The exception is the case of the class of graphs of treedepth at most $d$: it is known that all minimal obstructions for this class have size at most $d^{\Oh(d)}$~\cite{ChenCDFHNPPSWZ21}, but there exist obstructions with $2^d$ vertices~\cite{DvorakGT12}. Closing this asymptotic gap remains an interesting open problem.


\newcommand{\tw}{\mathsf{tw}}
\newcommand{\pw}{\mathsf{pw}}
\newcommand{\rw}{\mathsf{rw}}
\newcommand{\mw}{\mathsf{mw}}
\newcommand{\lmw}{\mathsf{lmw}}
\newcommand{\bag}{\mathsf{bag}}
\newcommand{\Leaves}{\mathsf{Leaves}}

\subsection{Path- and tree-like graphs: pathwidth, treewidth, and (linear) cliquewidth}\label{sec:treewidth}

We now proceed to describing path- and tree-like graphs, which in the sparse case are exemplified by the boundedness of parameters pathwidth and treewidth, and in the dense case are similarly defined through linear cliquewidth and cliquewidth, respectively. These notions have been extensively studied over the last four decades and there exists a large body of literature describing their algorithmic, combinatorial, and model-theoretic aspects. Therefore, we choose to restrict our description to a bare minimum that allows the reader to correctly place classes of bounded pathwidth, treewidth, and (linear) cliquewidth within the larger landscape.

\subsubsection{Treewidth and pathwidth}

Treewidth is probably the most well-known graph parameter measuring the structure of a graph. In its current form, it was introduced by Robertson and Seymour in~\cite{GM2}. The underlying notion of decomposition is called a {\em{tree decomposition}}.

\begin{definition}
 Let $G$ be a graph. A {\em{tree decomposition}} of $G$ consists of a tree $T$ and a {\em{bag function}} $\bag\colon V(T)\to 2^{V(G)}$ that maps every node of $T$ to a subset of vertices of $G$. We require the following:
 \begin{itemize}[nosep]
  \item For each vertex $u$ of $G$, the set of nodes $x$ of $T$ with $u\in \bag(x)$ is non-empty and connected in~$T$.
  \item For each edge $uv$ of $G$, there exists a node $x$ of $T$ such that $u,v\in \bag(x)$.
 \end{itemize}
 The {\em{width}} of $(T,\bag)$ is the maximum bag size minus $1$, that is, $\max_{x\in V(T)} |\bag(x)|-1$. The {\em{treewidth}} of~$G$, denoted $\tw(G)$, is the minimum width of a tree decomposition of $G$.
\end{definition}

See \cref{fig:Td} for an illustration.

\begin{figure}
 \centering
 \begin{tikzpicture}
\node at ( -4,0) {\includegraphics[scale=0.4]{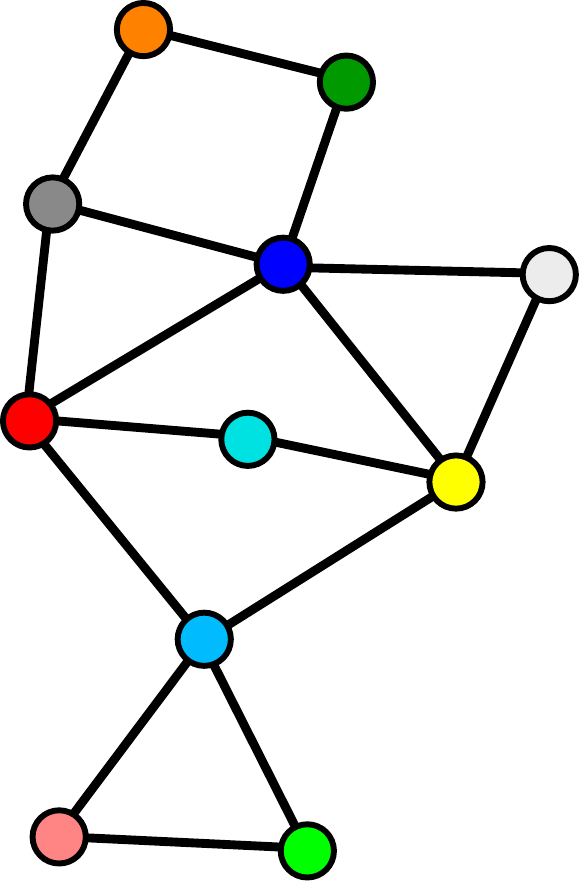}};
\node at ( 4,0) {\includegraphics[scale=0.4]{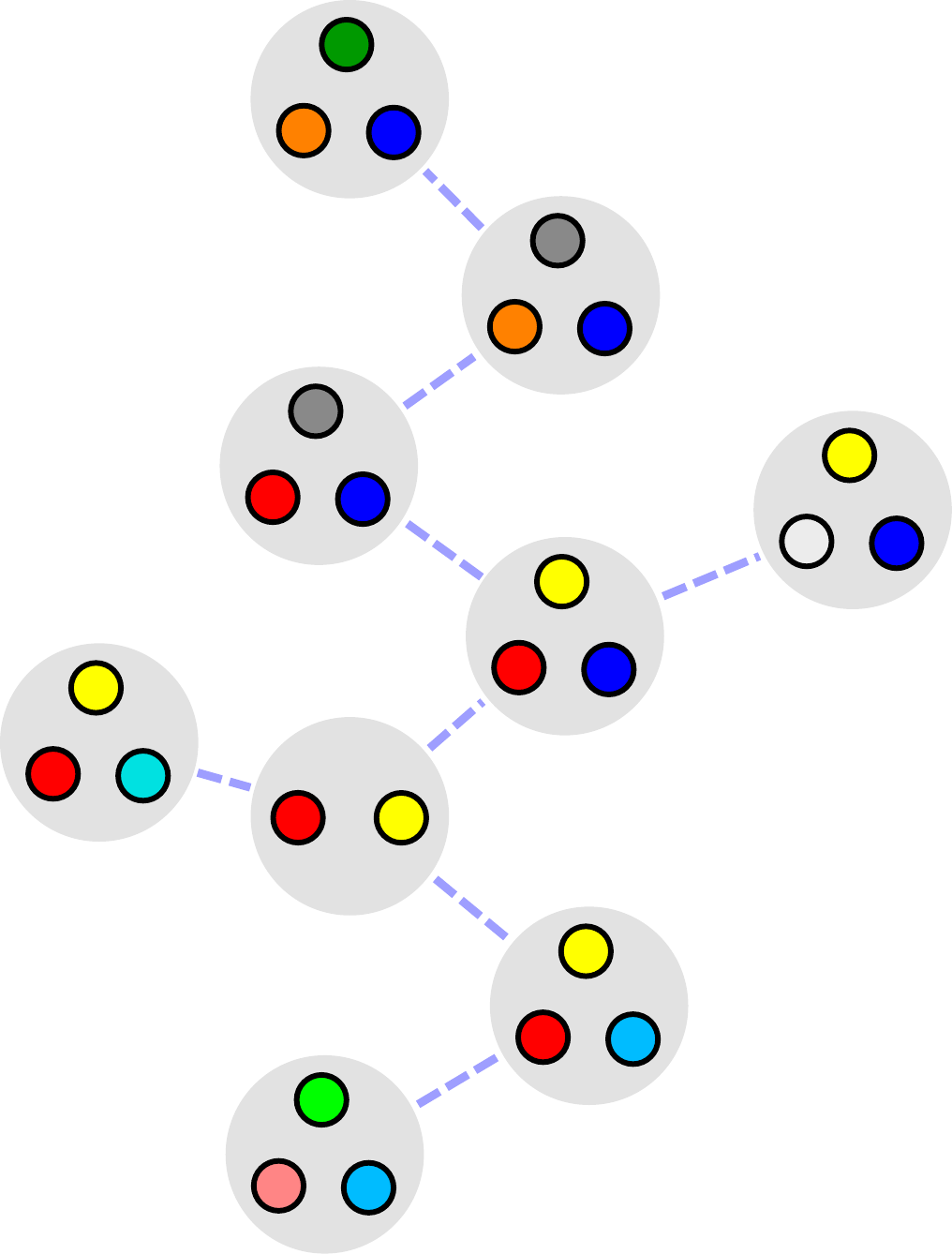}};
 \end{tikzpicture}
 \caption{A graph and its tree decomposition of width $2$.}\label{fig:Td}
\end{figure}

Over the years, tree decompositions turned out to be just the right abstraction for understanding what it means for a graph to be decomposed into pieces using vertex separators.
One convenient way to look at the definition is as follows. For every vertex $u$ of $G$, we ``stretch'' $u$ to a non-empty subtree $T_u$ of $T$ so that the adjacency relation is respected in the following sense: whenever $u$ and $v$ are adjacent, $T_u$ and $T_v$ have a non-empty intersection. Then the bag of a node $x\in V(T)$ consists of all vertices $u$ for which the subtree $T_u$ contains $x$. Thus, minimizing the width of a tree decomposition means minimizing the maximum number of subtrees $T_u$ that meet at a single node of $T$.

The notion of a {\em{path decomposition}} and of {\em{pathwidth}} is defined in exactly the same way, except we additionally require the tree $T$ underlying the decomposition to be a path. The pathwidth of a graph $G$ is denoted~$\pw(G)$. Pathwidth was introduced in~\cite{GM1}.

Clearly, we have $\tw(G)\leq \pw(G)$ for every graph $G$, because every path decomposition is also a tree decomposition. On the other hand, pathwidth is bounded by treedepth.

\begin{lemma}
 For every graph $G$, we have $\pw(G)\leq \td(G)-1$.
\end{lemma}
\begin{proof}
 Let $F$ be an elimination forest of $G$ of depth $\td(G)$. Let $u_1,\ldots,u_n$ be the ordering of vertices of $G$ according to the pre-order in $F$, where we assume that for every vertex, its children in $F$ are arbitrarily ordered (and the roots are arbitrarily ordered as well). For every vertex $u_i$, create a node $x_i$ with $\bag(x_i)$ consisting of $u_i$ and all its ancestors in $F$. It is now straightforward to verify that arranging the nodes $x_1,\ldots,x_n$ into a path in this order yields a path decomposition of $G$ of width $\td(G)-1$.
\end{proof}

Hence, every graph class of bounded treedepth also has bounded pathwidth, and every graph class of bounded pathwidth has also bounded treewidth.

The initial interest in treewidth and pathwidth comes from graph theory, particularly the Graph Minors project. The fundamental understanding of these parameters is delivered by two duality theorems below, which intuitively explain that
\begin{itemize}[nosep]
 \item each graph is either path-like (has bounded pathwidth) or branching (has a large tree minor); and
 \item each graph is either tree-like (has bounded treewidth) or two-dimensional (has a large grid~minor).
\end{itemize}
In the following statements, we say that a graph class $\Cc$ {\em{excludes}} some graph $H$ {\em{as a minor}} if every member of $\Cc$ excludes $H$ as a minor. Similarly for topological minors.

\begin{theorem}[\cite{GM1}]\label{thm:pw-duality}
 A graph class $\Cc$ has bounded pathwidth if and only if $\Cc$ excludes some tree as a minor.
\end{theorem}

\begin{theorem}[\cite{GM5}]\label{thm:tw-duality}
 A graph class $\Cc$ has bounded treewidth if and only if $\Cc$ excludes some grid as a minor.
\end{theorem}

It is easy to see that excluding some tree as a minor is equivalent to excluding some subcubic tree (tree of maximum degree at most $3$) as a topological minor, hence in \cref{thm:pw-duality} we could be equivalently postulate that $\Cc$ excludes some subcubic tree as a topological minor. Similarly, excluding some grid as a minor is equivalent to excluding some wall (see \cref{fig:wall}) as a topological minor, hence in \cref{thm:tw-duality} we could equivalently speak about excluding walls as topological minors. Also, grids are planar graphs and it is not hard to prove that every planar graph is a minor of some grid. Therefore, in \cref{thm:tw-duality} we could equivalently write that $\Cc$ excludes some planar graph as a minor. Note that this actually gives a remarkable characterization of planar graphs: a graph is planar if and only if its exclusion as a minor implies a bound on the treewidth.

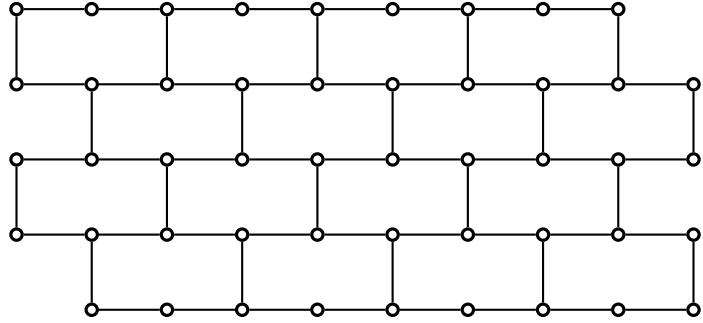
\begin{figure}
 \centering
 \begin{tikzpicture}

    \foreach \j in {2,3,4} {
        \foreach \i in {1,2,3,4,5,6,7,8,9,10} {
            \node[vertex] (a\i\j) at (\i,\j) {};
        }
    }
    \foreach \j in {1} {
        \foreach \i in {2,3,4,5,6,7,8,9,10} {
            \node[vertex] (a\i\j) at (\i,\j) {};
        }
    }
    \foreach \j in {5} {
        \foreach \i in {1,2,3,4,5,6,7,8,9} {
            \node[vertex] (a\i\j) at (\i,\j) {};
        }
    }

    \foreach \i in {2,4,6,8,10} {
        \foreach \j/\k in {1/2,3/4} {
            \draw[thick]  (a\i\j) -- (a\i\k);
        }
    }
    \foreach \i in {1,3,5,7,9} {
        \foreach \j/\k in {2/3,4/5} {
            \draw[thick]  (a\i\j) -- (a\i\k);
        }
    }
    \foreach \i/\k in {1/2,2/3,3/4,4/5,5/6,6/7,7/8,8/9,9/10} {
        \foreach \j in {2,3,4} {
            \draw[thick] (a\i\j) -- (a\k\j);
        }
    }
    \foreach \i/\k in {2/3,3/4,4/5,5/6,6/7,7/8,8/9,9/10} {
        \foreach \j in {1} {
            \draw[thick] (a\i\j) -- (a\k\j);
        }
    }
    \foreach \i/\k in {1/2,2/3,3/4,4/5,5/6,6/7,7/8,8/9} {
        \foreach \j in {5} {
            \draw[thick] (a\i\j) -- (a\k\j);
        }
    }

 \end{tikzpicture}
 \caption{A $4\times 4$ wall.}\label{fig:wall}
\end{figure}

When it comes to logical aspects, probably the most important property of treewidth is that its boundedness allows efficient model-checking of $\MSO_2$. This fact is known as {\em{Courcelle's Theorem}}.

\begin{theorem}[\cite{Courcelle90}]\label{thm:courcelle-tw}
 Let $\Cc$ be a graph class of bounded treewidth. Then there is an algorithm that given a graph $G\in \Cc$ and a sentence $\phi\in \MSO_2$, verifies whether $G\models \phi$ in time $\Oh_{\Cc,\phi}(n)$.
\end{theorem}

We remark that while \cref{thm:courcelle-tw} is stated above for graphs only, the result actually works for any class $\Cc$ of relational structures, provided the Gaifman graphs of structures from $\Cc$ form a graph class of bounded treewidth. Also, this is only a simple formulation that applies to decision problems. There are multiple other variants of Courcelle's Theorem, notably the optimization variant of Arnborg, Lagergren, and Seese~\cite{ArnborgLS91}. Also, it has been investigated by Kreutzer and Tazari~\cite{KreutzerT10} that under relevant assumptions from parameterized complexity theory, the boundedness of treewidth is a necessary condition for statements like \cref{thm:courcelle-tw} to hold; see~\cite{KreutzerT10} for a precise technical meaning of this claim.

We postpone the discussion of the proof of \cref{thm:courcelle-tw} to the next section, where it will be explained together with an analogous statement for cliquewidth.

Let us point out one important aspect of \cref{thm:courcelle-tw}. In the original formulation, Courcelle assumed that the input graph $G$ is given together with a tree decomposition of $G$ of width $k$, where $k$ is the constant bound on the treewidth of graphs from $\Cc$. However, if the input consists solely of the graph $G$, how to compute such a tree decomposition efficiently? This is a highly non-trivial algorithmic problem that has been thoroughly studied over the last 40 years. Let us mention three notable results in this line of research, which on the technical level represent three different approaches to the problem:
\begin{itemize}[nosep]
 \item Robertson and Seymour~\cite{GM13} gave a $4$-approximation algorithm running in time $\Oh(8^k\cdot k^2\cdot n^2)$ (see~\cite[Section~7.6]{platypus} for an implementation with this complexity). More precisely, the algorithm either certifies that $\tw(G)>k$, or finds a tree decomposition of $G$ of width at most $4k+3$.
 \item Bodlaender~\cite{Bodlaender96} gave an exact fixed-parameter algorithm for treewidth with running time $2^{\Oh(k^3)}\cdot n$. That is, the algorithm either concludes that $\tw(G)>k$ or constructs a tree decomposition of width at most $k$.
 \item Recently, Korhonen~\cite{Korhonen21} gave a $2$-approximation algorithm with running time $2^{\Oh(k)}\cdot n$.
\end{itemize}
In particular, by using either the algorithm of Bodlaender or that of Korhonen, in the context of \cref{thm:courcelle-tw} we may assume the input graph is supplied with  a tree decomposition of bounded width. We invite the reader to the introductory sections of the recent works of Korhonen~\cite{Korhonen21} and of Korhonen and Lokshtanov~\cite{KorhonenL23} for more information on the literature on computing tree decompositions.

\subsubsection{Cliquewidth and linear cliquewidth}\label{sec:cliquewidth}

Cliquewidth and linear cliquewidth are parameters analogous to treewidth and pathwidth, but suited for measuring tree- and path-likeness of dense graphs. The term {\em{cliquewidth}} and the contemporary definition were coined by Courcelle, Makowsky, Olariu, and Rotics in~\cite{courcelle2000linear,CourcelleO00}, however the notion has deeper roots in the work on (hyper)graph grammars done in the 90s, particularly in the work of Courcelle, Engelfriet, and Rozenberg~\cite{CourcelleER93}. In fact, a definition very similar and functionally equivalent to cliquewidth, called {\em{NLC-width}}, was introduced earlier by Wanke~\cite{Wanke94}. Among the notions closely related to cliquewidth, we find the concept of {\em{modular-width}} proposed by Rao~\cite{Rao08} the most compelling. The reason is that it is combinatorially the simplest and nicely extracts the intuition while avoiding the need of introducing notation related to graph grammars or labelled graphs. Therefore, we adopt this definition as our notion of ``cliquewidth''.

First, we need a simple notion of {\em{diversity}} of a vertex subset. Let $G$ be a graph and $A$ be a subset of vertices. We define the equivalence relation $\sim_A$ on $A$ as follows: for $u,v\in A$, we have
 $$u\sim_A v\quad\textrm{if and only if}\quad N(u)\setminus A = N(v) \setminus A.$$
In other words, $u\sim_A v$ if and only if $u$ and $v$ have the same neighbors outside of $A$. The {\em{diversity}} of $A$ is the number of equivalence classes of $\sim_A$. Intuitively, if the diversity of $A$ is at most $k$, then among the vertices of $A$ there are only $k$ different ``behaviors'' with respect to the outside.

We may now proceed to the definitions of modular-width\footnote{Another parameter with the name {\em{modular-width}} was introduced independently by Gajarsk\'y, Lampis, and Ordyniak~\cite{GajarskyLO13}, but their variant is not functionally equivalent to cliquewidth. More precisely, the boundedness of the modular-width of Gajarsk\'y et al. implies the boundedness of cliquewidth, but not vice versa.} and of linear modular-width.

\begin{definition}
 Let $G$ be a graph. A {\em{laminar decomposition}}\footnote{The term {\em{laminar decomposition}} does not exist in the literature and was invented for the purpose of this survey.} of $G$ is just a rooted binary tree $T$ whose leaf set is equal to the vertex set of $G$. For a node $x$ of $T$, by $\Leaves^T(x)$ we denote the set of leaves of $T$ that are descendants of $x$. (If $x$ itself is a leaf, then $\Leaves^T(x)=\{x\}$.) The {\em{diversity}} of $T$ is the maximum diversity of $\Leaves^T(x)$ among the nodes $x$ of $T$. Similarly, if $\sigma=(u_1,\ldots,u_n)$ is an ordering of the vertices of $G$, then the {\em{diversity}} of $\sigma$ is the maximum diversity of a prefix $\{u_1,\ldots,u_i\}$ of $\sigma$.

 The {\em{modular-width}} of $G$, denoted $\mw(G)$, is the minimum diversity of a laminar decomposition of $G$. The {\em{linear modular-width}} of $G$, denoted $\lmw(G)$, is the minimum diversity of an ordering of vertices of $G$.
\end{definition}

It is not hard to see that in the definition of linear modular-width, instead of speaking about vertex orderings, we could equivalently consider laminar decompositions that are {\em{caterpillars}}: paths with single leaves attached to every node. Thus, the linear modular-width of a graph is never smaller than its modular-width. It is known that modular-width differs at most by a multiplicative factor of $2$ from cliquewidth and NLC-width~\cite[Theorem~7]{Rao08}, and a similar fact also applies to linear modular-width, linear cliquewidth, and linear NLC-width. Therefore, in this article we refrain from recalling the original definition of (linear) cliquewidth, and we rely on (linear) modular-width instead.  However, since ``cliquewidth'' is by now a very widespread term, in the context of graph classes we will still speak about classes of bounded/unbounded (linear) cliquewidth, noting that this is the same as classes of bounded/unbounded (linear) modular-width.

\paragraph*{Comparison with other parameters.}
The relation between pathwidth and treewidth on one side, and linear cliquewidth and cliquewidth on the other side, is as expected: the dense notions generalize the sparse notions, and restricting attention to weakly sparse classes projects the dense notions to the sparse notions. In \cref{thm:td-shb}, we have seen the same behavior for treedepth and shrubdepth.

\begin{theorem}[\cite{CourcelleO00}]\label{thm:pw-lcw}
 Every class of bounded pathwidth also has bounded linear cliquewidth. Conversely, every weakly sparse class of bounded linear cliquewidth in fact has bounded pathwidth.
\end{theorem}

\begin{theorem}[\cite{CourcelleO00}]\label{thm:tw-cw}
 Every class of bounded treewidth also has bounded cliquewidth. Conversely, every weakly sparse class of bounded cliquewidth in fact has bounded treewidth.
\end{theorem}

We note that \cref{thm:pw-lcw} is not mentioned in~\cite{CourcelleO00}, but can be proved using a similar argumentation as \cref{thm:tw-cw}, which is proved in~\cite{CourcelleO00}.

Next, it should not be a surprise that classes of bounded linear cliquewidth generalize classes of bounded shrubdepth.

\begin{theorem}[\cite{GanianHNOM19}]
 Every class of bounded shrubdepth also has bounded linear cliquewidth.
\end{theorem}
\begin{proof}[Proof sketch]
 Let $\Cc$ be the class in question and suppose $d,m\in \N$ are such that every graph in $\Cc$ has a tree-model of depth at most $d$ and using at most $m$ labels. Consider any $G\in \Cc$ and let $(\Lambda,\lambda,T,(M_t))$ be such a tree-model. Let $\sigma=(u_1,\ldots,u_n)$ be an ordering of the vertex set of $G$ according to a pre-order in~$T$, where in~$T$ we assume that the children of every node are ordered arbitrarily. (Recall here that the leaf set of $T$ coincides with the vertex set of $G$.) It is easy to verify that the diversity of $\sigma$ is bounded by~$dm$.
\end{proof}

Finally, we comment on the question of computing laminar decompositions of approximately optimum width. For this task, it is convenient to consider a yet another parameter {\em{rankwidth}}, introduced by Oum and Seymour~\cite{OumS06}. We refrain here from introducing rankwidth formally; the only property that we need is that it is functionally equivalent to modular-width as follows: for every graph $G$,
$$\rw(G)\leq \mw(G)\leq 2^{\rw(G)},$$
where $\rw(G)$ denotes the rankwidth of $G$. Usually, the inequality as above is stated for cliquewidth instead of modular-width, but the proof also works for modular-width and is even simpler. Also, in the proof one even shows that exactly the same decomposition for one width notion works also for the other width notion. Therefore, to approximately compute cliquewidth or modular-width, it is enough to approximate rankwidth. For this, Oum and Seymour~\cite{OumS06} gave a $2^{\Oh(k)}\cdot n^9\log n$-time algorithm, which either certifies that $\rw(G)>k$, or finds a rank decomposition of width at most $3k+1$. A line of improvements of the running time followed, culminating in the recent work of Korhonen and Sokołowski~\cite{Korhonen024}, who gave an exact algorithm with running time $\Oh_k(n^{1+o(1)})+\Oh(m)$, where $m$ is the edge count of the input graph. We invite the reader to the introductory section of~\cite{Korhonen024} for an overview of the relevant literature, but from our perspective, the bottom line is the~following.

\begin{theorem}[\cite{Korhonen024}]\label{thm:cw-apx}
 Let $\Cc$ be a class of graphs of bounded cliquewidth. Then given $G\in \Cc$, one can compute a laminar decomposition of $G$ of diversity bounded by a constant in time $\Oh_{\Cc}(n^{1+o(1)})+\Oh(m)$.
\end{theorem}

\paragraph*{Model-theoretic aspects.} First, cliquewidth enjoys the following analogue of \cref{thm:courcelle-tw} for $\MSO_1$, proved by Courcelle, Makowsky, and Rotics~\cite{courcelle2000linear}.

\begin{theorem}[\cite{courcelle2000linear}]\label{thm:courcelle-cw}
 Let $\Cc$ be a graph class of bounded cliquewidth. Then there is an algorithm that given a graph $G\in \Cc$ and a sentence $\phi\in \MSO_1$, verifies whether $G\models \phi$ in time $\Oh_{\Cc,\phi}(n^{1+o(1)})+\Oh(m)$.
\end{theorem}

Again, the original proof presented in~\cite{courcelle2000linear} assumes that a suitable decomposition is provided on input; then the algorithm runs in time linear in the size of this decomposition. The statement above is derived by combining this with the algorithm of Korhonen and Sokołowski (\cref{thm:cw-apx}). Similarly as in the case of \cref{thm:courcelle-tw}, the proof of \cref{thm:courcelle-cw} is robust and allows various variants, such as an optimization variant; see \cite{courcelle2000linear} for details.

The underlying reason behind \cref{thm:courcelle-tw,thm:courcelle-cw} is the following model-theoretic characterization of classes of bounded treewidth or cliquewidth as those $\MSO$-transducible from the class of~trees, and of classes of bounded pathwidth or linear cliquewidth as those $\MSO$-transducible from the class of~paths.

\begin{theorem}\label{thm:cw-model}
 A class of graphs $\Cc$ has bounded cliquewidth if and only if $\Cc$ can be $\MSO$-transduced from the class of trees.
\end{theorem}

\begin{theorem}\label{thm:lcw-model}
 A class of graphs $\Cc$ has bounded linear cliquewidth if and only if $\Cc$ can be $\MSO$-transduced from the class of paths.
\end{theorem}

\begin{theorem}\label{thm:tw-model}
 A class of graphs $\Cc$ has bounded treewidth if and only if the incidence encodings of graphs $\Cc$ can be $\MSO$-transduced from the class of trees.
\end{theorem}

\begin{theorem}\label{thm:pw-model}
 A class of graphs $\Cc$ has bounded pathwidth if and only if the incidence encodings of graphs $\Cc$ can be $\MSO$-transduced from the class of paths.
\end{theorem}

\cref{thm:cw-model} is explicitly proved in the book of Courcelle and Engelfriet as a part of~\cite[Corollary~7.38]{CEbook}. The same argumentation can be also applied to argue \cref{thm:lcw-model}, see also~\cite[Theorem~7.47]{CEbook}. \cref{thm:tw-model,thm:pw-model} can be then derived from \cref{thm:cw-model} using the same argument as we did in the proof of \cref{thm:td-model}, except that we use \cref{thm:tw-cw,thm:pw-lcw} instead of \cref{thm:td-shb}.

From \cref{thm:cw-model,thm:lcw-model,lem:trans-composition} we immediately get the following.

\begin{theorem}\label{thm:cw-lcw-ideals}
 Classes of bounded cliquewidth as well as classes of bounded linear cliquewidth are both $\MSO$ and $\FO$ ideals.
\end{theorem}

With \cref{thm:cw-model,thm:tw-model} explained, we can go back to \cref{thm:courcelle-tw,thm:courcelle-cw} and sketch how they can be proved. Consider the cliquewidth case for concreteness, that is, \cref{thm:courcelle-cw}. Given a graph $G$ belonging to a class of bounded cliquewidth $\Cc$, we can compute a laminar decomposition $T$ of $G$ of constant diversity, using \cref{thm:cw-apx}. What the proof of \cref{thm:cw-model} tells us is that $G$ can be in fact $\MSO$-transduced from~$T$. This means that $G$ can be interpreted, using an $\MSO$ interpretation, in some coloring $T^+$ of $T$. Recall that our goal is to verify whether a given sentence $\phi\in \MSO_1$ is satisfied in $G$. We can use now the Backwards Translation Lemma (\cref{lem:backwards}, or rather its analogue for $\MSO$) to construct a sentence $\psi$ over the signature of $T^+$ such that $G\models \phi$ if and only if $T^+\models \psi$. Thus, we effectively reduced model-checking $\MSO_1$ on classes of bounded cliquewidth to model-checking $\MSO$ on classes of colored trees. The latter problem can be solved in linear time using classic tree automata constructions. Namely, we translate $\psi$ into a tree automaton $\cal A_\psi$ that is equivalent to $\psi$ in the sense that $\cal A_\psi$ accepts exactly those colored trees in which $\psi$ is satisfied. Then it suffices to run $\cal A_\psi$ on $T^+$. The reasoning for \cref{thm:courcelle-tw} is analogous.

\paragraph*{Obstructions.} It should be mentioned that there is an embedding notion, called {\em{vertex-minors}}, that in the context of cliquewidth (or rather more suitably, rankwidth) serves the role analogous to that of the minor order. In particular, Geelen, Kwon, McCarty, and Wollan~\cite{GeelenKMW23} proved an analogue of \cref{thm:tw-duality}: a class of graphs has bounded cliquewidth if and only if it excludes a (suitably defined) grid as a vertex-minor. The theory of vertex-minors is a fascinating topic with significant parallels to the theory of graph minors, see a recent survey of Kim and Oum~\cite{KimO24}. However, in this survey we focus on logical aspects, which at this point seem to be somewhat orthogonal to the ``philosophy'' of vertex-minors (we comment more on this in \cref{sec:outlook}). Hence, we choose not to explore here the connections with the theory of vertex-minors.

Regarding logical obstructions, the following fundamental duality result for cliquewidth was proved by Courcelle and Oum~\cite{CourcelleO07} using a Grid Theorem for binary matroids due to Geelen, Gerards, and Whittle~\cite{GeelenGW07}.

\begin{theorem}[\cite{CourcelleO07}]\label{thm:cw-dependent}
 A class of graphs $\Cc$ has bounded cliquewidth if and only if the class of all graphs cannot be $\CMSO$-transduced from $\Cc$.
\end{theorem}

The left-to-right implication of \cref{thm:cw-dependent} follows from the fact that classes of bounded cliquewidth are even closed under $\CMSO$-transductions. For the difficult right-to-left implication, the main work lies is showing that if a class $\Cc$ has unbounded cliquewidth, then the class of all grids is $\CMSO$-transducible from $\Cc$. Indeed, then one can easily see that the class of all graphs is $\MSO$-transducible from the class of grids by an argument similar to the one we used in \cref{lem:rook-not-dependent}.

Note that in \cref{thm:cw-dependent}, the transduction uses the logic $\CMSO$ that is more expressive than $\MSO$ in that it allows modular counting. This is used crucially in the proof: the reasoning relies on connections with vertex-minors, and to express relevant properties it seems necessary to be able to check the parity of the cardinalities of sets (thus, in fact, only counting modulo $2$ is needed). Whether \cref{thm:cw-dependent} holds also for the weaker notion of $\MSO$-transductions remains a curious open problem.

An analogue of \cref{thm:cw-dependent} for linear cliquewidth would be the following statement, which to the best of our knowledge is at this point still open.

\begin{conjecture}\label{con:lcw-model}
 A class of graphs $\Cc$ has bounded linear cliquewidth if and only if the class of trees cannot be $\CMSO$-transduced from $\Cc$.
\end{conjecture}

Finally, an analogue of \cref{thm:cw-dependent} also holds for treewidth. As expected, one should simply replace the adjacency encoding with the incidence encoding.

\begin{theorem}\label{thm:tw-dependent}
 A class of graphs $\Cc$ has bounded treewidth if and only if the class of all graphs cannot be $\MSO$-transduced from the incidence encodings of the graphs from $\Cc$.
\end{theorem}
\begin{proof}[Proof sketch]
 On one hand, if $\Cc$ has bounded treewidth, then so does the class of $1$-subdivisions of graphs from $\Cc$, which are Gaifman graphs of the incidence encodings of graphs from $\Cc$. Hence, every class $\MSO$-transducible from those incidence encodings has bounded cliquewidth (\cref{thm:cw-lcw-ideals}). However, the class of grids has unbounded cliquewidth, because it is weakly sparse and has unbounded treewidth (\cref{thm:tw-cw}).

 On the other hand, if $\Cc$ has unbounded treewidth, then by \cref{thm:tw-duality}, graphs from $\Cc$ contain arbitrarily large grids as minors. It is then not hard to transduce the class of all grids from $\Cc$. Namely, given the incidence encoding of a graph $G$ that contains a minor model of a $k\times k$ grid, the transduction highlights the edges relevant for this minor model using two colors --- one for the spanning trees of branch sets and one for the edges connecting the branch sets of vertices adjacent in the grid --- and proceeds to encode the adjacency relation of the grid using an $\MSO_2$ formula. As we mentioned before, an argument similar to that used in the proof of \cref{lem:rook-not-dependent} shows that the class of all graphs can be $\MSO$-transduced from the class of grids. Composing the two transductions sketched above gives an $\MSO$-transduction that applied to incidence encodings of graphs from $\Cc$, yields all graphs.
\end{proof}

\subsection{Minor-free classes}\label{sec:minors}

We now proceed to the first concept of graphs that are not necessarily tree-like: minor-free graphs. This is a highly explored topic within structural graph theory, thanks to the monumental Graph Minors project developed by Robertson and Seymour, and the even larger body of works that were later built on top of~it. In this survey we only touch upon the tip of the iceberg and discuss the role of minor-free classes within the larger logically-motivated perspective.

Call a graph class $\Cc$ {\em{minor-free}} if there exists a graph $H$ such that every graph in $\Cc$ excludes $H$ as a minor. For instance, planar graphs form a minor-free class, for they exclude both $K_5$ and $K_{3,3}$ as minors~\cite{Kuratowski30,Wagner37}. Observe that since every graph is a subgraph of some complete graph $K_t$, a graph class is minor-free if and only if it excludes $K_t$ as a minor, for some $t\in \N$. Therefore, if we define the {\em{Hadwiger number}} of a graph $G$ as the largest $t$ such that $G$ contains $K_t$ as a minor, then $\Cc$ is minor-free if and only if $\Cc$ has bounded Hadwiger number.

Arguably, the most important fact about minor-free graphs is the Graph Minors Theorem, proved by Robertson and Seymour~\cite{GM20}, which states that the minor order is a well quasi-ordering on finite graphs. Recall that this means that in the minor order there are no infinite descending chains (which is trivial) and no infinite antichains (which is highly non-trivial). From this it directly follows that minor-closed classes (that is, classes closed under taking minors) can be characterized by finitely many forbidden minors. These are often called {\em{minor obstructions}}.

\begin{corollary}\label{cor:forb-minors}
 For every minor-closed class of graphs $\Cc$, there exists a finite set of graphs $\Ff$ such that for every graph $G$,
 $$G\in \Cc\quad\textrm{if and only if}\quad G\textrm{ does not contain any member of }\Ff\textrm{ as a minor.}$$
\end{corollary}
\begin{proof}
 It suffices to take as $\Ff$ the minor-minimal graphs that do not belong to $\Cc$. They form an antichain in the minor order, hence from the Graph Minors Theorem it follows that $\Ff$ is finite.
\end{proof}

Note that \cref{cor:forb-minors} does not provide any method for deriving the obstruction set $\Ff$ from the class~$\Cc$. For instance, if $\Cc$ is the class of planar graphs then $\Ff=\{K_5,K_{3,3}\}$, and a complete obstruction set is known for the class of graphs embeddable in the projective plane~\cite{Archdeacon81,GloverHW79}. However, the full list is actually unknown even for the class of graphs embeddable in the torus. There have been, however, studies of bounds on the maximum sizes of minor obstructions for specific minor-closed classes of interest. Examples include the classes of graphs of pathwidth at most $k$ and of treewidth at most $k$~\cite{Lagergren98}, as well as graphs embeddable in a fixed surface~\cite{Seymour93}.

Along the way to prove the Graph Minors Theorem, Robertson and Seymour introduced a large toolbox of results for minor-free classes of graphs. Among those, we would like to highlight the so-called {\em{Structure Theorem}}, which is a fundamental duality result explaining that graphs from minor-free classes can be pieced together in a tree-like manner from graphs that are close to being embeddable in a fixed surface.

\begin{theorem}[Structure Theorem for minor-free graphs, \cite{GM16}]\label{thm:structure-theorem}
 Let $\Cc$ be a minor-free class of graphs. Then there exists a surface $\Sigma$ and an integer $k\in \N$ such that every graph $G\in \Cc$ admits a tree decomposition $(T,\bag)$ with the following properties:
 \begin{itemize}[nosep]
  \item the adhesion of $(T,\bag)$ is at most $k$; and
  \item for every node $x$ of $T$, the torso of $\bag(x)$ in $G$ is $k$-almost embeddable in $\Sigma$.
 \end{itemize}
\end{theorem}

\begin{figure}
 \centering
 \begin{tikzpicture}
\node at (0,0) {\includegraphics[scale=0.4]{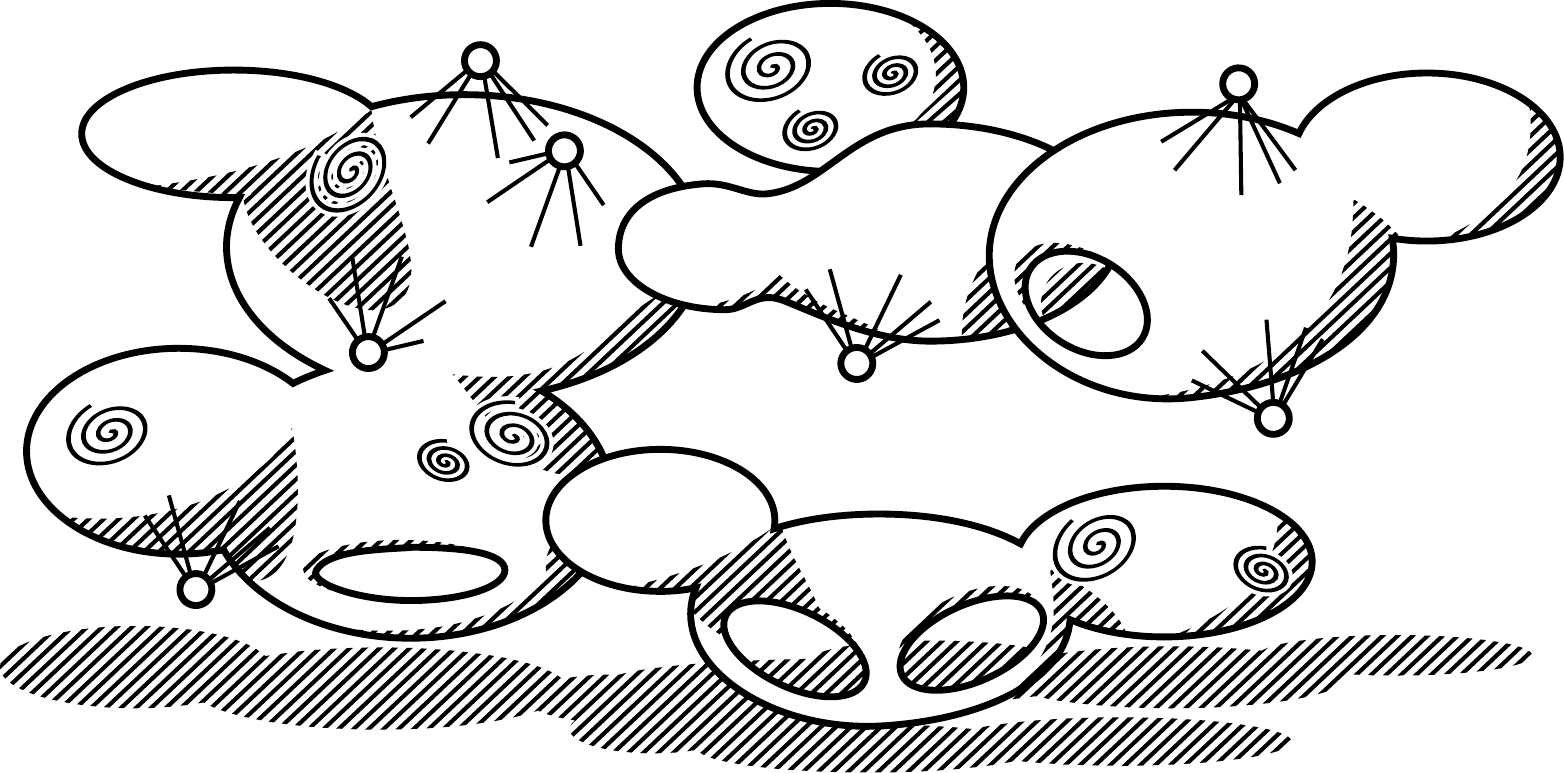}};
 \end{tikzpicture}
 \caption{A somewhat artistic take on the decomposition provided by the Structure Theorem for minor-free classes (\cref{thm:structure-theorem}). Figure by Felix Reidl. The author would like to thank Felix for allowing the usage of this figure in this survey.}\label{fig:minorFree}
\end{figure}

See \cref{fig:minorFree} for an illustration.
Let us explain the undefined terms:
\begin{itemize}[nosep]
 \item The {\em{adhesion}} of a tree decomposition $(T,\bag)$ is the maximum size of the intersection of any pair of adjacent bags, that is, $\max_{xy\in E(T)} |\bag(x)\cap \bag(y)|$. Thus, the tree decomposition provided by \cref{thm:structure-theorem} may have bags of arbitrarily large size, but the ``interface'' between any two adjacent bags is of constant size.
 \item The {\em{torso}} of $\bag(x)$ is the graph obtained from $G[\bag(x)]$ by turning the intersection $\bag(x)\cap \bag(y)$ into a clique, for every neighbor $y$ of $x$ in $T$. Note that by the previous point, these cliques are of size at most $k$ each.
 \item By a {\em{surface}} we mean a compact $2$-dimensional manifold without boundary. The notion of embedding is standard: vertices are mapped to distinct points, and edges are mapped to internally disjoint curves connecting corresponding endpoints.
 \item We refrain from formally explaining the concept of $k$-almost embeddability, due to its technicality. In a nutshell, a graph is $k$-almost embeddable in a surface $\Sigma$ if it is embeddable in $\Sigma$ except that (a) at most $k$ vertices, called {\em{apices}}, do not need to be embedded, and (b) in the embedding there might be some crossings, but they are confined to at most $k$ disks called {\em{vortices}}, each inducing a subgraph of pathwidth at most $k$.
\end{itemize}
We remark that the existence of a tree decomposition like in \cref{thm:structure-theorem} also witnesses that the graph in question excludes some fixed graph $H$ as a minor (where $H$ depends on $\Sigma$ and $k$), hence \cref{thm:structure-theorem} can be regarded as a duality theorem that characterizes minor-free classes through the existence of a suitable decomposition. We also refer the reader to the work of Diestel, Kawarabayashi, M\"uller, and Wollan~\cite{DiestelKMW12} for a more contemporary exposition of the Structure Theorem. 

The tree decomposition provided by \cref{thm:structure-theorem} is extremely useful in the context of algorithm design. Essentially, it allows lifting algorithmic results from the setting of surface-embedded graphs to the setting of general minor-free classes, provided that (a) one can efficiently deal with the features introduced by the ``almost'' aspect of the embedding, that is, with apices and vortices, and (b) the understanding of almost-embeddable torsos can be pieced together along the tree decomposition, using for instance dynamic programming. There are countless instantiations of this principle in the literature, but let us mention one that is particularly relevant for us: the algorithm for model-checking $\FO$ on minor-free classes, due to Flum and Grohe~\cite{FlumG01}.

\begin{theorem}[\cite{FlumG01}]\label{thm:FO-mc-minor-free}
For every minor-free class of graphs $\Cc$, there exists $c\in \N$ and an algorithm that given a graph $G\in \Cc$ and a sentence $\phi\in \FO$, decides whether $G\models \phi$ in time $\Oh_{\Cc,\phi}(n^c)$.
\end{theorem}

We remark that the original proof of \cref{thm:FO-mc-minor-free} did not provide any explicit bound on the constant $c$. Later developments that we will discuss in \cref{sec:sparsity} generalized the result of \cref{thm:FO-mc-minor-free} while providing a linear fixed-parameter algorithm, that is, one with $c=1$.

The proof of \cref{thm:FO-mc-minor-free}, due to Flum and Grohe~\cite{FlumG01}, relies on assembling tools established by Frick and Grohe~\cite{FrickG01} and by Grohe~\cite{Grohe03} with the decomposition provided by \cref{thm:FO-mc-minor-free}.
In summary, the whole approach can be described as follows:
\begin{itemize}[nosep]
 \item Use Gaifman's Locality Theorem~\cite{gaifman1982local} to reduce deciding whether the given graph $G$ satisfies the given sentence $\phi$ to understanding what $\FO$ sentences of bounded {\em{quantifier rank}} (i.e., maximum number of nested quantifiers) are satisfied in balls of bounded radius in $G$.
 \item Graphs embeddable in a fixed surface $\Sigma$ are known to have {\em{bounded local treewidth}}: the subgraph induced by any ball of radius $r$ has treewidth bounded by $\Oh_{\Sigma}(r)$. Therefore, balls of bounded radius in a surface-embedded graph can be understood using Courcelle's Theorem (\cref{thm:courcelle-tw}). This solves the case of graphs embeddable in a fixed surface.
 \item Lift the argument from embeddable graphs to almost embeddable graphs by a careful treatment of apices and vortices. In fact, one shows that after deletion of the apices, almost embeddable graphs have bounded local treewidth, even in the presence of vortices.
 \item Lift the argument from almost embeddable graphs to all graphs from the considered minor-free class $\Cc$ using dynamic programming on the decomposition provided by \cref{thm:FO-mc-minor-free}.
\end{itemize}

\subsection{Twin-width and sparse twin-width}\label{sec:twinwidth}

The notion of twin-width is relatively new, the parameter was proposed in 2020 by Bonnet, Kim, Thomass\'e, and Watrigant~\cite{BonnetKTW22}. However, the idea behind it is so simple, so fundamental, and fits the big picture so well that it is indeed surprising that it was not discovered well before. In fact, an important inspiration for Bonnet et al. was the work of Guillemot and Marx~\cite{GuillemotM14} on pattern-free permutations, which to a large extent already introduced basic decomposition concepts that eventually became twin-width.

The main point of twin-width is that it is a parameter that is suited for understanding the structure in dense graphs, while not being restricted to tree-like graphs like cliquewidth. Before its introduction, there was no robust parameter with this combination of features. Twin-width turned out to be particularly robust from the logical perspective, as classes of bounded twin-width form an $\FO$-ideal (\cref{thm:tww-ideal}). Being naturally tailored to dense graphs, twin-width also has a sparse counterpart, {\em{sparse twin-width}}, which appears to nicely fit into the hierarchy of properties of classes of sparse graphs.

\newcommand{\tww}{\mathsf{tww}}
\newcommand{\err}{\mathsf{err}}
\newcommand{\quo}{\mathsf{quo}}

\subsubsection{Twin-width}

The decomposition concept behind twin-width is called a {\em{contraction sequence}}. Intuitively, a contraction sequence is a procedure that iteratively ``folds'' the graph  while keeping a small amount of ``error'' or ``ambiguity'' along the way; this amount of error is the {\em{width}} of the contraction sequence. We need a few definitions to make this formal.

Let $G$ be a graph. Given two disjoint vertex subsets $A$ and $B$, we say that the pair $(A,B)$ is
\begin{itemize}[nosep]
 \item {\em{complete}} if every vertex of $A$ is adjacent to every vertex of $B$;
 \item {\em{anti-complete}} if every vertex of $A$ is non-adjacent to every vertex of $B$; and
 \item {\em{impure}} if it is neither complete nor anti-complete.
\end{itemize}
Next, a {\em{partition}} of $G$ is just a partition of the vertex set of $G$, that is, a family of non-empty, pairwise disjoint subsets of $V(G)$ that sums up to the whole $V(G)$. Elements of a partition will be called {\em{parts}}. For a partition $\Pp$, we define the {\em{error graph}} $\err(G,\Pp)$ as follows:
\begin{itemize}[nosep]
 \item the vertices of $\err(G,\Pp)$ are the parts of the partition $\Pp$; and
 \item two distinct parts $A,B\in \Pp$ are adjacent in $\err(G,\Pp)$ if they form an impure pair.
\end{itemize}
The intuition is that for pure (complete or anti-complete) pairs of parts, the relation between them is homogeneous and can be described by one bit: either there are all possible edges between the parts, or no edges at all. Thus, adjacency in the error graph signifies ambiguity of the adjacency relation between vertices of the parts.

With these basic definitions in place, we can introduce contraction sequences and twin-width.

\begin{definition}
A {\em{contraction sequence}} of a graph $G$ is a sequence of partitions $\Pp_n,\ldots,\Pp_1$ of $G$ such that:
\begin{itemize}[nosep]
 \item $\Pp_n$ is the discrete partition where every vertex is in its own singleton part;
 \item $\Pp_1$ is the partition in which there is only one part comprising of all the vertices; and
 \item for each $i\in\{n-1,n-2,\ldots,1\}$, $\Pp_i$ is obtained from $\Pp_{i+1}$ by taking some pair of parts and merging them into one.
\end{itemize}
The {\em{width}} of the contraction sequence $\Pp_n,\Pp_{n-1},\ldots,\Pp_1$ is the minimum integer $d$ such that every error graph $\err(G,\Pp_i)$, for $i\in \{n,\ldots,1\}$, has maximum degree at most $d$. The {\em{twin-width}} of $G$, denoted $\tww(G)$, is the minimum width of a contraction sequence of $G$.
\end{definition}

See \cref{fig:contraction} for an example of a contraction sequence. Note that in the definition, we chose to index the partitions in the reverse way: from $\Pp_n$ to $\Pp_1$. This is not only to have the nice property that $|\Pp_i|=i$, but also to indicate that in the proofs it is often useful to think about contraction sequences being reversed, as procedures that ``unfold'' the whole graph from a single vertex.

 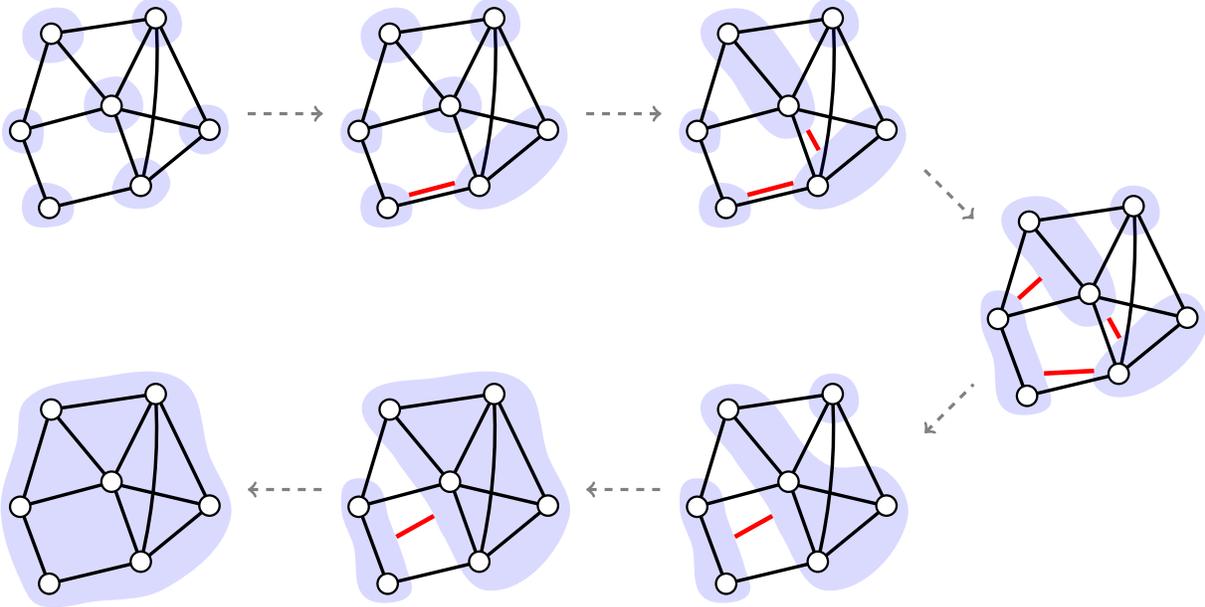
\begin{figure}
  \centering
  \begin{tikzpicture}
   \node at (-6.5,2.5) {\includegraphics[scale=0.3]{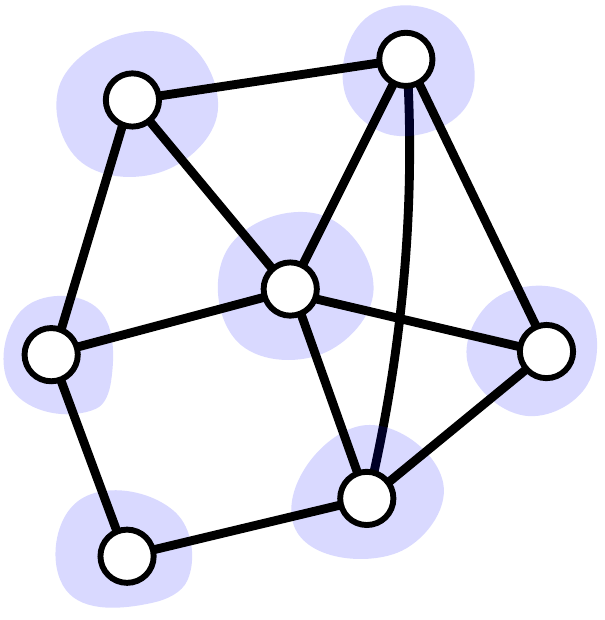}};
   \node at (-6.5,-2.5) {\includegraphics[scale=0.3]{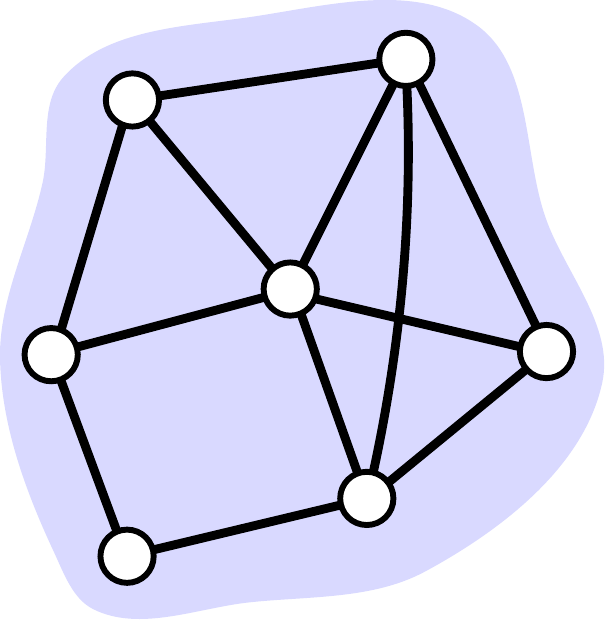}};
   \node at (-2,2.5) {\includegraphics[scale=0.3]{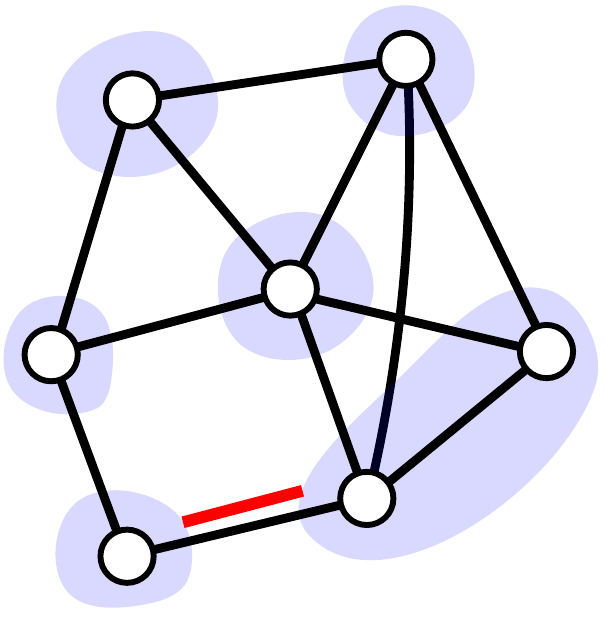}};
   \node at (-2,-2.5) {\includegraphics[scale=0.3]{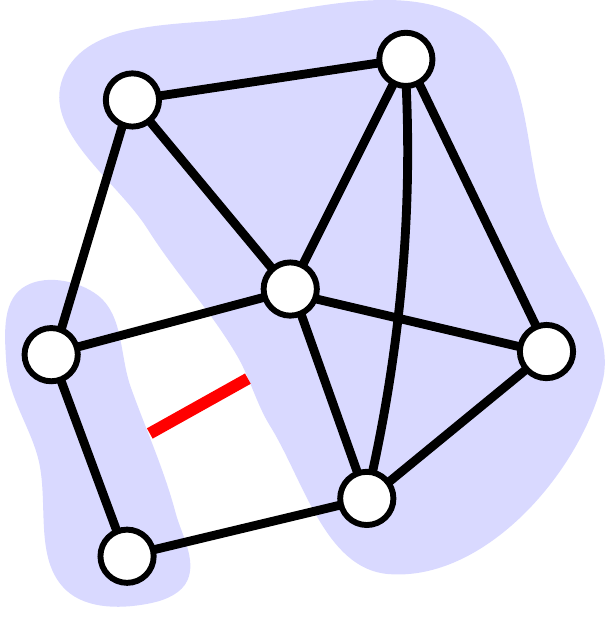}};
   \node at ( 2.5,2.5) {\includegraphics[scale=0.3]{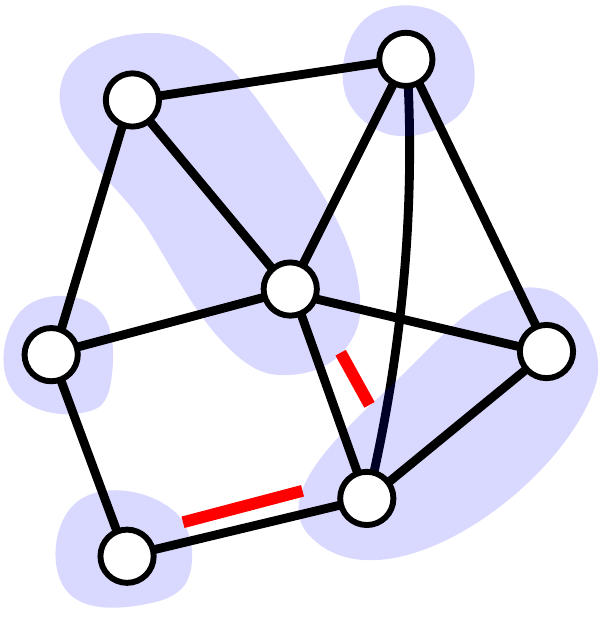}};
   \node at ( 2.5,-2.5) {\includegraphics[scale=0.3]{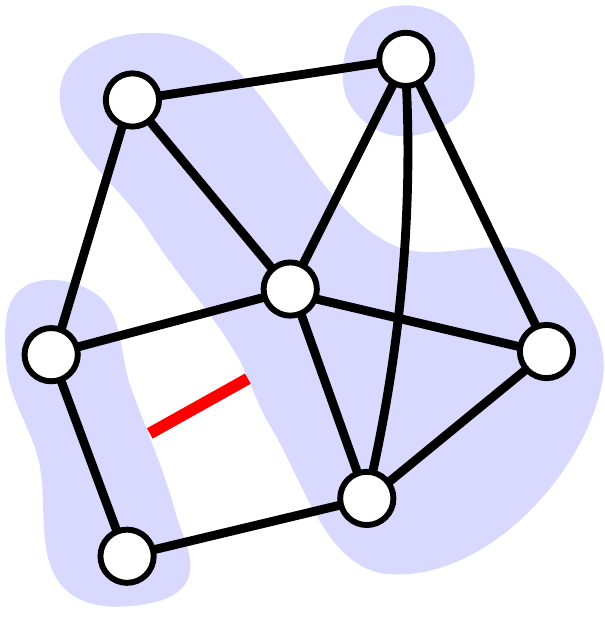}};
   \node at ( 6.5,0) {\includegraphics[scale=0.3]{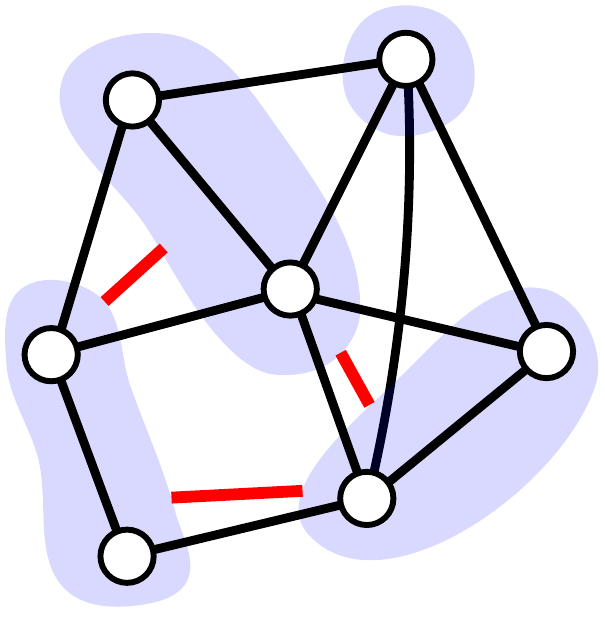}};

\draw[very thick,black!50,dashed,->] (-4.75,2.5) -- (-3.75,2.5);

\draw[very thick,black!50,dashed,<-] (-4.75,-2.5) -- (-3.75,-2.5);

\draw[very thick,black!50,dashed,->] (-4.75+4.5,2.5) -- (-3.75+4.5,2.5);

\draw[very thick,black!50,dashed,<-] (-4.75+4.5,-2.5) -- (-3.75+4.5,-2.5);

\draw[very thick,black!50,dashed,->] (4.25,1.75) -- (4.9,1.1);
\draw[very thick,black!50,dashed,<-] (4.25,-1.75) -- (4.9,-1.1);

  \end{tikzpicture}
  \caption{A contraction sequence of width $2$ of a graph. The parts of the consecutive partitions are depicted in blue, the edges of the error graph are depicted in red.}\label{fig:contraction}
 \end{figure}

\paragraph*{Relation to other notions.} On one hand, it is not hard to see that the boundedness of cliquewidth implies the boundedness of twin-width.

\begin{theorem}\label{thm:tww-mw}
 For every graph $G$, we have
 $$\tww(G)\leq 2\cdot \mw(G)-1.$$
 Consequently, every graph class of bounded cliquewidth also has bounded twin-width.
\end{theorem}
\begin{proof}[Proof sketch]
 Let $T$ be a laminar decomposition of $G$ of diversity $k\coloneqq \mw(G)$. Let a {\em{front}} in $T$ be a set of nodes $F$ such that every root-to-leaf path in $G$ contains exactly one element of $F$. For a front $F$, we can define a partition $\Pp_F$ of $G$ as follows: the parts of $\Pp_F$ are the equivalence classes of all the equivalence relations $\sim_{\Leaves^T(x)}$, for $x\in F$. (Recall here the notation from the definition of diversity of a vertex subset.) It is easy to see that if $F$ is a front, then in the error graph $\err(G,\Pp_F)$ all connected components are of size at most $k$, for equivalence classes of different relations $\sim_{\Leaves^T(x)}$ always form pure pairs. Therefore, $\err(G,\Pp_F)$ in particular has maximum degree at most $k-1$.

 We construct a contraction sequence of $G$ in a bottom-up manner. At each point we keep a front $F$ and the current partition in the contraction sequence is $\Pp_F$. Initially, $F$ is the set of all the leaves of $T$ (this corresponds to the discrete partition), and at the end $F$ consists of only the root of $T$ (this corresponds to the partition with only one part). In each step of the construction, we find a node $x$ that does not belong to $F$ but both its children $y$ and $z$ belong to $F$, and we replace $y$ and $z$ with $x$, thus obtaining a new front $F'$; this can be always done until $F$ consists only of the root. It is not hard to see that in the contraction sequence, this operation can be emulated by taking the at most $2k$ equivalence classes of $\sim_{\Leaves^T(y)}$ and $\sim_{\Leaves^T(z)}$, and merging some of them to eventually obtain the equivalence classes of $\sim_{\Leaves^T(x)}$. This can be done by merging classes in pairs in any order. Since there are at most $2k$ classes involved in this merging, the maximum degrees of the error graphs along the way never exceed $2k-1$.
\end{proof}

The proof of \cref{thm:tww-mw} actually shows that for graphs of bounded modular width (or equivalently, cliquewidth), we can construct a contraction sequence where even every connected component of every error graph is bounded in size. This has been observed by Bonnet, Kim, Reynald, and Thomass\'e in~\cite{BonnetKRT22}, and in fact this is an exact characterization of classes of bounded cliquewidth. Similarly, the boundedness of linear cliquewidth is equivalent to the possibility of constructing a contraction sequence where in every error graph, the total number of edges is always bounded~\cite{BonnetKRT22}.

On the other hand, it turns out that all minor-free classes actually have bounded twin-width.

\begin{theorem}[\cite{BonnetKTW22}]\label{thm:minor-free-tww}
 Every minor-free class of graphs has bounded twin-width.
\end{theorem}

Thus, the theory of twin-width also applies to classes of graphs that are not necessarily tree-like, for instance to planar graphs. In fact, it is known that every planar graph has twin-width at most $8$~\cite{HlinenyJ23}, and there are planar graphs of twin-width at least $7$~\cite{KralL23}.

However, it turns out that not all classes of sparse graphs have bounded twin-width.

\begin{theorem}[\cite{BonnetGKTW21}]\label{thm:subcubic-tww}
 The class of subcubic graphs has unbounded twin-width.
\end{theorem}

The known proof of \cref{thm:subcubic-tww}, proposed by Bonnet, Geniet, Kim, Thomass\'e, and Watrigant~\cite{BonnetGKTW21}, is indirect and relies on asymptotic counting. Namely, in~\cite{BonnetGKTW21} Bonnet et al. proved that if a class of graphs $\Cc$ has bounded twin-width, then the number of distinct graphs with vertex set $\{1,\ldots,n\}$ in $\Cc$ is bounded by $2^{\Oh(n)}\cdot n!$. However, the number of distinct subcubic graphs with vertex set $\{1,\ldots,n\}$ is $\Omega((n!)^{\frac{3}{2}})$. Constructing an explicit family of subcubic graphs of unbounded twin-width remains a curious open problem. This question can be regarded as an excuse for developing tools for proving lower bounds on the twin-width of~graphs.

\paragraph*{Ordered graphs and matrices.} A key element of the theory of twin-width is the connection with adjacency matrices. More precisely, it turns out that a graph has bounded twin-width if and only if its vertex set can be ordered so that in this ordering, the adjacency matrix avoids certain complicated patterns. We need a few more definitions to express this connection formally.

A {\em{vertex ordering}} of a graph $G$ is just a linear order $\leq$ on the vertex set of $G$. Then, an {\em{ordered graph}} is a pair $(G,\leq)$ where $G$ is a graph and $\leq$ is its vertex ordering. Call a set of vertices $A$ {\em{convex}} in $\leq$ if $u\leq v\leq w$ and $u,w\in A$ entails also $v\in A$. Then a {\em{division}} of $(G,\leq)$ is a partition of $V(G)$ into parts that are convex in $\leq$.

For an ordered graph $(G,\leq)$, we can consider its {\em{adjacency matrix}}, which is just a $\{0,1\}$-matrix $M$ indexed by vertices of $G$ in the order $\leq$: in the cell $M[u,v]$ we put $1$ if vertices $u$ and $v$ are adjacent, and $0$ otherwise. For two convex sets $A$ and $B$, the {\em{zone}} induced by $A$ and $B$ is the submatrix $M[A,B]$ consisting of cells that lie in rows corresponding to $A$ and columns corresponding to $B$. We call a zone {\em{mixed}} if it contains at least two unequal rows and at least two unequal columns. Equivalently, a zone is non-mixed if either all its rows are equal or all its columns are equal.

\newcommand{\mx}{\mathsf{mix}}
\newcommand{\grid}{\mathsf{grid}}

With these definitions in place, we can describe what adjacency matrices we consider complicated: those that contain large {\em{mixed minors}}.

\begin{definition}
 Let $(G,\leq)$ be an ordered graph and $M$ be its adjacency matrix. A {\em{mixed minor}} of order $k$ in $M$ is a pair of divisions $(\cal R,\cal C)$ of $(G,\leq)$ such that
 \begin{itemize}[nosep]
  \item $|\cal R|=|\cal C|=k$, and
  \item for every pair of parts $A\in \cal R$ and $B\in \cal C$, the zone $M[A,B]$ is mixed.
 \end{itemize}
 We define the {\em{mixed minor number}}, $\mx(G,\leq)$, to be the largest order of a mixed minor in $M$.
\end{definition}

In the definition above, it is instructive to think of $\cal R$ and of $\cal C$ as  divisions of the row set of $M$ and of the column set of $M$, respectively. Under this interpretation, a mixed minor in $M$ is a partition of $M$ into a $k\times k$ grid of zones, each of them being mixed.

The following result provides the fundamental connection between twin-width and mixed minors. In essence, it says that the boundedness of twin-width is equivalent to the possibility of ordering the graph so that the adjacency matrix avoids large mixed minors.

\begin{theorem}\label{thm:tww-mixed}
 Let $G$ be a graph.
 \begin{itemize}[nosep]
  \item If $\tww(G)\leq d$, then there exists a vertex ordering $\leq$ of $G$ such that $\mx(G,\leq)\leq \Oh(d)$.
  \item If there exists a vertex ordering $\leq$ of $G$ such that $\mx(G,\leq)\leq k$, then $\tww(G)\leq 2^{2^{\Oh(k)}}$.
 \end{itemize}
\end{theorem}

The proof of the first point of \cref{thm:tww-mixed} is not very hard: one can imagine a contraction sequence witnessing $\tww(G)\leq d$ as a tree $T$ describing the structure of consecutive merges, and then for the vertex ordering $\leq$ one can take the pre-order of $T$ restricted to the leaves ({\em{aka}}, vertices of $G$). The second point is the interesting one. The proof relies on a deep result of Marcus and Tardos~\cite{MarcusT04} which says the following: if an $n\times n$ $\{0,1\}$-matrix contains at least $c(k)\cdot n$ entries $1$, where $c(k)$ is a constant depending only on $k$, then there is a pair of divisions $(\cal R,\cal C)$ as in the definition of a mixed minor, just with every zone containing at least one entry $1$. With this result, in the proof of the second point of \cref{thm:tww-mixed} one applies a greedy procedure that constructs a contraction sequence. Marcus-Tardos Theorem is used to argue that if the procedure gets stuck, then it is because it uncovered a large mixed minor.

In a later work, Bonnet, Giocanti, Ossona de Mendez, Simon, Thomass\'e, and Toru\'nczyk~\cite{BonnetGMSTT24} showed that there is nothing special about the condition ``at least two different rows and at least two different columns'' in the definition of a mixed zone. One could equivalently replace it with condition ``at least $k$ different rows and at least $k$ different columns'', and \cref{thm:tww-mixed} would still hold (subject to changes in the asymptotics of the bounds).

\paragraph*{Model-theoretic aspects.} Finally, we arrive at the key element of the discussion: the properties of twin-width related to logic. First, as proved by Bonnet et al.~\cite{BonnetKTW22} already in the first article on twin-width, $\FO$ model-checking can be done efficiently on graphs provided with contraction sequences of constant width.

\begin{theorem}[\cite{BonnetKTW22}]\label{thm:tww-mc}
 There is an algorithm that given a graph $G$, a contraction sequence of $G$ of width at most~$d$, and a sentence $\phi\in \FO$, decides whether $G\models \phi$ in time $\Oh_{\phi,d}(n)$.
\end{theorem}

Note that \cref{thm:tww-mc} assumes that the input graph is supplied with a contraction sequence witnessing a bound on the twin-width. Unfortunately, currently it is unknown whether the twin-width of a graph can be approximated efficiently. An algorithmic statement that would match \cref{thm:tww-mc} would be the following: there is a universal constant $c\in \N$ and an algorithm that given a graph $G$ of twin-width $d$, computes in time $\Oh_d(n^c)$ a contraction sequence of $G$ of width $\Oh_d(1)$. If such a fixed-parameter approximation algorithm existed, then combining it with \cref{thm:tww-mc} would prove that the $\FO$ model-checking problem is fixed-parameter tractable when parameterized by the twin-width and the size of the input sentence. At this point, it is even unknown whether an $\mathsf{XP}$ approximation algorithm exists for this task (that is, one where $c$ may depend on~$d$).

The algorithm of \cref{thm:tww-mc} applies a method of iterative aggregation of information. We process the contraction sequence in order and at each point, say when considering partition $\Pp_i$, we store information about the behavior of every part $A\in \Pp_i$ with respect to parts that are close in the error graph $\err(G,\Pp_i)$. It turns out that provided this ``information about the behavior'' is defined right, it can be updated upon every consecutive merge in the contraction sequence. At the last, trivial partition we obtain the relevant information about the whole graph, from which the satisfaction of the sentence in question can be extracted. The original proof of \cite{BonnetKTW22} devised a combinatorial notion of {\em{shuffles}} to describe updating the relevant information. In a later work, Gajarsk\'y, Pilipczuk, Przybyszewski, and Toru\'nczyk~\cite{GajarskyPPT22} reinterpreted this approach using a more model-theoretic notion of (suitably defined) $\FO$~types.

The same iterative aggregation approach underlying the proof of \cref{thm:tww-mc} can be used to argue that applying a fixed $\FO$ transduction to a graph of bounded twin-width again yields a graph of bounded twin-width. In other words, we have the following.

\begin{theorem}[\cite{BonnetKTW22}]\label{thm:tww-ideal}
 Classes of bounded twin-width are an $\FO$ ideal.
\end{theorem}

Finally, it appears that the notion of twin-width explains very well the model-theoretic aspects of ordered graphs; this topic was explored by Bonnet, Giocanti, Ossona de Mendez, Simon, Thomass\'e, and Toru\'nczyk in~\cite{BonnetGMSTT24}. Namely, we can understand ordered graphs as relational structures with two binary relations: the adjacency relation and the order. The notion of twin-width can be very easily generalized to binary relational structures (that is, with all relations of arity at most $2$): simply, a pair $A,B$ of disjoint subsets of the universe is pure if and only if the relations holding between every $a\in A$ and every $b\in B$ are the same. Thus, we may speak about the twin-width of different binary structures, such as directed graphs, permutations (treated as pairs of linear orders), and ordered graphs. In the context of ordered graphs, Bonnet et al.~\cite{BonnetGMSTT24} proved the following duality result, which says that the boundedness of twin-width exactly captures not being equivalent to all graphs.

\begin{theorem}[\cite{BonnetGMSTT24}]\label{thm:tww-dependent}
 A class of ordered graphs $\Cc$ has bounded twin-width if and only if the class of all graphs cannot be transduced from $\Cc$.
\end{theorem}

Note here that from the class of all graphs one can transduce the class of all ordered graphs, and in fact the class of all $\Sigma$-structures, for any finite signature $\Sigma$. This means that a class of ordered graphs $\Cc$ is either restricted by the boundedness of twin-width, and can transduce only classes of bounded twin-width, or is ``all powerful'' and can transduce any class of structures.

The work of Bonnet et al.~\cite{BonnetGMSTT24} includes many other interesting results about the twin-width of ordered graphs. These include a fixed-parameter algorithm to approximate the twin-width of an ordered graph, and a combinatorial characterization of classes of ordered graphs of bounded twin-width via exclusion of certain forbidden patterns in their adjacency matrices. We invite the reader to~\cite{BonnetGMSTT24} for more details.

\subsubsection{Sparse twin-width}

Finally, we also discuss the sparse counterpart of twin-width. The easiest way to define it is to use the general principle: restrict attention to weakly sparse classes.

\begin{definition}
 A class of graphs $\Cc$ has {\em{bounded sparse twin-width}} if it has bounded twin-width and is weakly sparse.
\end{definition}

Note that maybe a bit counter-intuitively, we did not define any graph parameter ``sparse twin-width'' whose boundedness might be in question. Though, a parameter that would suit the definition above could be $\max(\tww(G),\omega^{\#}(G))$, where recall that $\omega^{\#}(G)$ is the largest $t\in \N$ such that $G$ contains $K_{t,t}$ as a subgraph. Also, note that every minor-free class $\Cc$ is both weakly sparse and, by \cref{thm:minor-free-tww}, has bounded twin-width, hence $\Cc$ also has bounded sparse twin-width.

Classes of bounded sparse twin-width were introduced and investigated by Bonnet et al. in~\cite{BonnetGKTW21}. They proved that there is an elegant characterization of those classes through adjacency matrices, analogous to that of \cref{thm:tww-mixed}. The only element that needs to be replaced is the definition of a mixed minor. The idea is that instead of considering mixed zones, we will consider simply {\em{non-empty zones}}, that is, zones that contain at least one entry $1$. Then a {\em{grid minor}} of order $k$ in the adjacency matrix $M$ of an ordered graph $(G,\leq)$ is a pair of divisions $(\cal R, \cal C)$ such that $|\cal R|=|\cal C|=k$ and for every $A\in \cal R$ and $B\in \cal C$, the zone $M[A,B]$ is non-empty; and the {\em{grid minor number}} of $(G,\leq)$, denoted $\grid(G,\leq)$, is the largest order of a grid minor in~$M$. With these definitions, the said characterization reads a follows.

\begin{theorem}[\cite{BonnetGKTW21}]\label{thm:sparse-tww}
 A class of graphs $\Cc$ has bounded sparse twin-width if and only if there exists an integer $k\in \N$ such that with every graph $G\in \Cc$ one can associate a vertex ordering $\leq_G$ satisfying $\grid(G,\leq_G)\leq k$.
\end{theorem}



\subsection{Sparsity: bounded expansion and nowhere denseness}\label{sec:sparsity}

We arrive at the most general concepts of well-structured sparse graphs considered in this survey: notions of classes of bounded expansion and of nowhere dense classes. These notions were first introduced and studied by Ne\v{s}et\v{r}il and Ossona de Mendez in a series of articles~\cite{NPoM-be1,NPoM-be2,NPoM-be3,NPoM-nd-uqw,NPoM-nd}. Their vast potential was very quickly recognized: multiple different authors joined in the investigation of the subject, and the area has experienced a tremendous growth throughout the last 15 years. It became collectively known as the field of {\em{Sparsity}}.

We again touch upon only the most important topics in Sparsity, focusing on aspects connected with logic, and particularly with the model-checking problem for $\FO$. However, we still make a swift run through all the vital characterizations of classes of bounded expansion and of nowhere dense classes. The reason is two-fold:
\begin{itemize}[nosep]
 \item These characterizations uncover different facets of the theory, showing that the concepts of bounded expansion and nowhere denseness are of fundamental nature. Each characterization brings along a technique, which can be applied both in the combinatorial and in the algorithmic~context.
 \item In \cref{sec:new}, we will discuss multiple concepts of well-structuredness in dense graphs of model-theoretic origin, most importantly {\em{monadic stability}} and {\em{monadic dependence}}. It turns out that the combinatorial characterizations of these dense notions are often direct lifts of the characterizations of sparse notions from the world of Sparsity. Thus, a good understanding of the tools of Sparsity brings important intuition to the context of well-structured dense graphs.
\end{itemize}
For a broader and deeper overview of the theory of Sparsity, we refer an interested reader to the book of Ne\v{s}et\v{r}il and Ossona de Mendez~\cite{sparsity}, or to more contemporary and more compact lecture notes of Pilipczuk, Pilipczuk, and Siebertz~\cite{sparsityNotes}.

\subsubsection{Definitions and intuitions}

The motivation behind the notions of bounded expansion and of nowhere denseness can be in fact traced back to the study of First-Order logic, $\FO$. Namely, recall that in \cref{thm:FO-mc-minor-free} we have seen that the model-checking problem for $\FO$ can be solved in fixed-parameter time on every minor-free class of graphs. But could minor-freeness be really the delimiting line for the complexity of this problem? $\FO$ is famously {\em{local}}: it cannot really discover long connections between vertices. For instance, it is a standard exercise in the methods of model theory such as compactness or Ehrenfeucht-Fra\"isse games to show that there is no $\FO$ formula $\varphi(x,y)$ that would check whether $x$ and $y$ are in the same connected component of the graph. Even distinguishing whether the distance between $x$ and $y$ is $2^q+1$ or $2^q+2$ cannot be done by an $\FO$ formula of quantifier rank at most $q$ (where the quantifier rank is the maximum number of nested quantifiers). More generally, the powerful Gaifman's Locality Theorem~\cite{gaifman1982local} (whose formal statement we omit due to its technicality) explains that verifying the satisfaction of any $\FO$ sentence $\phi$, say of quantifier rank $q$, on a graph $G$ can be reduced to understanding what $\FO$-definable properties are satisfied in balls of radius $2^{\Oh(q)}$ in $G$. The bottom line is: even if in $G$ there was a minor model of a large clique, but whose branch sets are of huge diameter, an $\FO$ sentence $\varphi$ of bounded quantifier rank would be unable to detect the presence of such a model.

This discussion leads to the following hopeful idea: maybe, the tractability of $\FO$ model-checking would be still maintained even if we excluded the existence of only {\em{local}} minor models? For this, we need to understand what it means for a minor model to be local. This understanding is delivered through the following definition; see \cref{fig:shallow} for an illustration.

\begin{definition}
 We say that a graph $H$ is a {\em{depth-$d$ minor}} of a graph $G$ if there exists a minor model $\eta$ of $H$ in $G$ such that for every vertex $u\in V(H)$, the branch set $\eta(u)$ has radius at most $d$.
\end{definition}

 \begin{figure}
  \centering
  \begin{tikzpicture}
   \node at (0,0) {\includegraphics[scale=0.4]{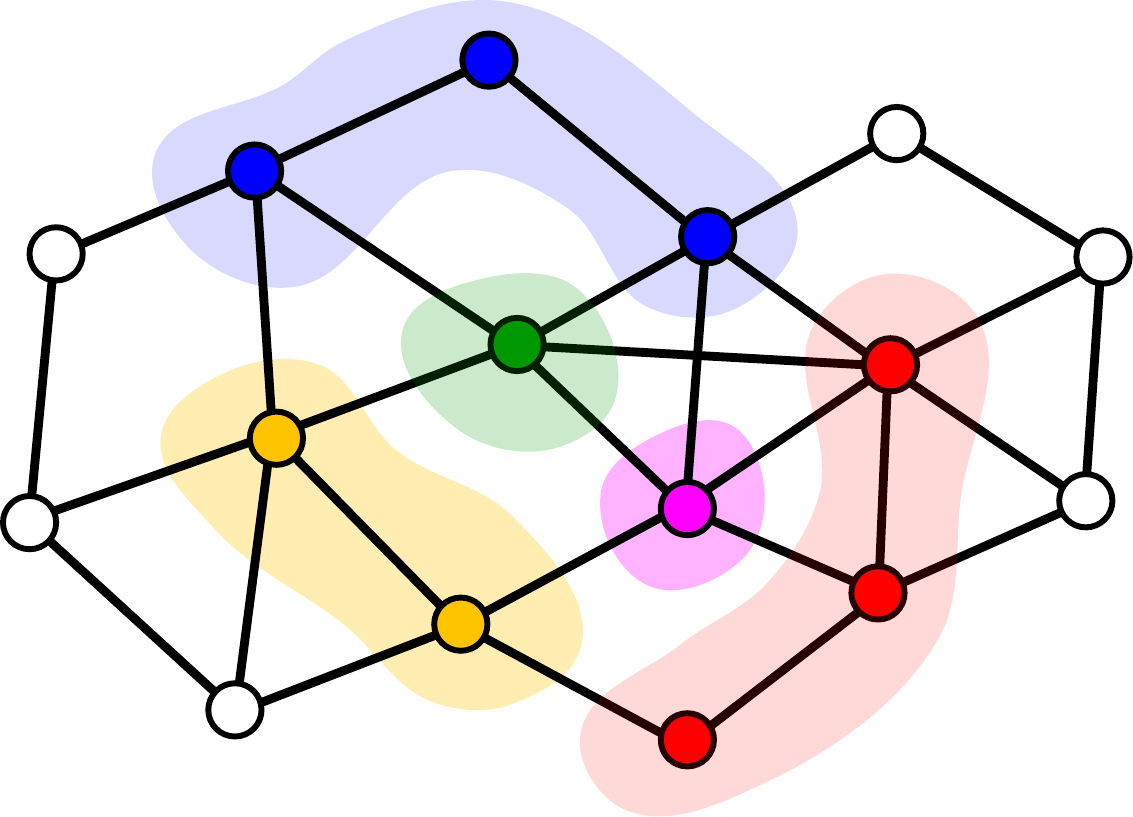}};

  \end{tikzpicture}
  \caption{A depth-$1$ minor model of $K_5$.}\label{fig:shallow}
 \end{figure}

Note that depth-$0$ minors are just subgraphs, while depth-$\infty$ minors, where we only require that the branch sets are connected, are just minors. Thus, the notion of bounded-depth minors serves as an interpolation between the subgraph order and the minor order.

The idea is to consider classes of graphs that exclude complicated bounded-depth minors. There are two natural ways to understand the term ``complicated'': either we exclude dense graphs, or we exclude complete graphs. These two ways lead to the definitions of bounded expansion and of nowhere denseness, respectively. Recall here that for a graph parameter $\pi$ and a graph class $\Cc$, we denote $\pi(\Cc)\coloneqq \sup_{G\in \Cc} \pi(G)$.

\newcommand{\avg}{\mathsf{avg}}

\begin{definition}\label{def:be}
 For a graph $G$, the {\em{average degree}} of $G$ is the quantity $$\avg(G)\coloneqq \frac{\sum_{u\in V(G)}\deg(u)}{|V(G)|}=2\cdot \frac{|E(G)|}{|V(G)|},$$ where $\deg(u)$ denotes the degree of vertex $u$. Then for a graph $G$ and an integer $d\in \N$, we define the {\em{depth-$d$ average degree}} as the maximum average degree among the depth-$d$ minors of $G$:
 $$\nabla_d(G)\coloneqq \max\left\{\,\avg(H)\colon H\textrm{ is a depth-}d\textrm{ minor of }G\,\right\}.$$
 We say that a graph class $\Cc$ has {\em{bounded expansion}} if $\nabla_d(\Cc)$ is finite, for every $d\in \N$.
\end{definition}

\begin{definition}\label{def:nd}
 For a graph $G$ and an integer $d\in \N$, we define the {\em{depth-$d$ clique number}} of $G$ as the maximum order of a complete graph that is contained in $G$ as a depth-$d$ minor:
 $$\omega_d(G)\coloneqq \max\left\{\,t\colon K_t\textrm{ is a depth-}d\textrm{ minor of }G\,\right\}.$$
 We say that a graph class $\Cc$ is {\em{nowhere dense}} if $\omega_d(\Cc)$ is finite, for every $d\in \N$.
\end{definition}

Thus, bounded expansion and nowhere denseness are not defined by the boundedness of one parameter, as was the case for all the properties of graph classes considered in the previous sections, but rather by simultaneous boundedness of a family of parameters: $\nabla_0,\nabla_1,\nabla_2,\ldots$ for bounded expansion, and $\omega_0,\omega_1,\omega_2,\ldots$ for nowhere denseness. It is instructive to ``unpack'' what this actually means:
\begin{itemize}[nosep]
 \item A graph class $\Cc$ has bounded expansion if and only if there is function $c\colon \N\to \N$ such that for every graph $G\in \Cc$ and a depth-$d$ minor $H$ of $G$, we have $\avg(H)\leq c(d)$.
 \item A graph class $\Cc$ is nowhere dense if and only if there is a function $t\colon \N\to \N$ such that for every graph $G\in \Cc$ and $d\in \N$, $G$ does not contain $K_{t(d)+1}$ as a depth-$d$ minor.
\end{itemize}
Thus, our definitions of sparsity are naturally gradated by the depth. Namely, when we look at larger and larger depth $d$, we allow the existence of denser and denser graphs, or larger and larger cliques, as depth-$d$ minors; this is the meaning of functions $c(d)$ and $t(d)$, respectively. However, for every fixed $d\in \N$, there is a constant upper bound on the average degree, respectively the clique number, of graphs that can be derived as depth-$d$ minors of graphs from the class in question.

Clearly, for every graph $G$ and $d\in \N$, we have $\omega_d(G)\leq \nabla_d(G)+1$, because a complete graph on $t$ vertices has average degree $t-1$. Hence, every class of bounded expansion is also nowhere dense. It turns out that these two concepts are actually different: there exist classes that are nowhere dense but have unbounded expansion. Admittedly, they are a bit artificial; below is an example. Recall here that the {\em{girth}} of a graph $G$ is the length of the shortest cycle in $G$.

\begin{theorem}[{\cite[Example 5.1]{sparsity}}]\label{thm:nd-be-example}
 Let $\Cc$ be the class of all graphs $G$ such that the girth of $G$ is larger than the maximum degree of $G$. Then $\Cc$ is nowhere dense, but does not have bounded~expansion.
\end{theorem}

Having understood the relation between bounded expansion and nowhere denseness, let us compare these concepts with notions that we have already seen before.
First, it is known that every graph excluding $K_t$ as a minor has average degree bounded by $\Oh(t\sqrt{\log t})$~\cite{Kostochka84}. Since every minor of a $K_t$-minor-free graph is also $K_t$-minor-free, independent of the depth of the minor model, we infer that $\nabla_d(G)\leq \Oh(t\sqrt{\log t})$ for every $d\in \N$, provided $G$ excludes $K_t$ as a minor. In other words, the parameters $\nabla_d(G)$ are universally bounded by a constant independent of $d$. From this we conclude that every minor-free class of graphs has bounded~expansion.

In fact, a stronger statement is true: even every class of bounded sparse twin-width has bounded expansion. This was first proved by Bonnet et al.~\cite{BonnetGKTW21} using an indirect reasoning involving stability of twin-width under transductions (\cref{thm:tww-ideal}). Later, Dreier, Gajarsk\'y, Jiang, Ossona de Mendez, and Raymond~\cite{DreierGJMR22} gave an elegant direct argument (see also the correction~\cite{DreierGJMR24}).

\begin{theorem}[\cite{BonnetGKTW21,DreierGJMR22,DreierGJMR24}]
 Every class of bounded sparse twin-width has bounded expansion.
\end{theorem}

However, the concept of bounded expansion is strictly more general than that of bounded sparse twin-width. Recall that the class of subcubic graphs does not have bounded sparse twin-width (\cref{thm:subcubic-tww}). Yet, it is not hard to see that for every $\Delta\in \N$, the class of graphs of maximum degree $\Delta$ has bounded expansion. Indeed, every depth-$d$ minor of such a graph has maximum degree at most $\Delta\cdot (\Delta-1)^{d-1}$, hence also average degree bounded by this quantity. More generally, every topological-minor-free class of graphs also has bounded expansion. This follows from the results of Dvo\v{r}\'ak~\cite{dvorak-thesis} relating parameters $\nabla_d$ with their topological-minor variants, combined with the bounds on the average degree in topological-minor-free classes~\cite{BollobasT98,KomlosS94}; see also~\cite[Corollary 4.1]{sparsity}.

Having placed bounded expansion and nowhere denseness within the big picture, we may now proceed to describing the most important combinatorial characterizations of these concepts.

\newcommand{\WReach}{\mathsf{WReach}}

\subsubsection{Characterizations: bounded expansion}\label{sec:char-be}

\paragraph*{Generalized coloring numbers.} Generalized coloring numbers provide a decompositional viewpoint on classes of bounded expansion (and on nowhere dense classes, see \cref{{thm:quant-nd}}), which is extremely useful when working with those concepts on the technical level. They were introduced by Kierstead and Young in~\cite{KiersteadY03}, but their applicability in the context of classes of bounded expansion was observed by Zhu~\cite{Zhu09}.

Usually one considers three variants of generalized coloring numbers: {\em{weak coloring number}}, {\em{strong coloring number}}, and {\em{admissibility}}. All these are functionally equivalent, in the sense of being bounded by a function of each other. For the sake of brevity, we discuss here only the weak coloring number, which appears to be the most useful as a~tool; a broader discussion can be found in~\cite{sparsity,sparsityNotes}.
Roughly speaking, the idea is that every graph from a fixed class of bounded expansion admits a vertex ordering that controls short connections between vertices through constant-size local separators, called {\em{weak reachability sets}}.  The formal definitions are below, see also \cref{fig:WReach} for an illustration.

\begin{definition}
Consider an ordered graph $(G,\leq)$ and a distance parameter $d\in \N$. For two vertices $u\leq v$ of $G$, we say that $u$ is {\em{weakly $d$-reachable}} from $v$ if there is a path $P$ with endpoints $u$ and $v$ and of length at most $d$, such that every vertex $w$ traversed by $P$ satisfies $u\leq w$. (In other words, $P$ is disallowed to travel through vertices smaller in $\leq$ than the lower endpoint $u$.) We define the {\em{weak $d$-reachability set}} of~$v$, denoted $\WReach_d^{G,\leq}[v]$, as the set of all vertices $u\leq v$ that are weakly $d$-reachable from $v$. Then the {\em{weak $d$-coloring number}} of the ordered graph $(G,\leq)$ is the maximum size of a weak $d$-reachability set:
$$\wcol_d(G,\leq)\coloneqq \max_{v\in V(G)} \left|\WReach_d^{G,\leq}[v]\right|.$$
Finally, the {\em{weak $d$-coloring number}} of a graph $G$ is the minimum weak $d$-coloring number that can be obtained by equipping $G$ with a vertex ordering:
$$\wcol_d(G)\coloneqq \min\{\,\wcol_d(G,\leq)\colon \leq\textrm{ is a vertex ordering of }G\,\}.$$
\end{definition}

 \begin{figure}
  \centering
  \begin{tikzpicture}
   \node at (0,0) {\includegraphics[scale=0.4]{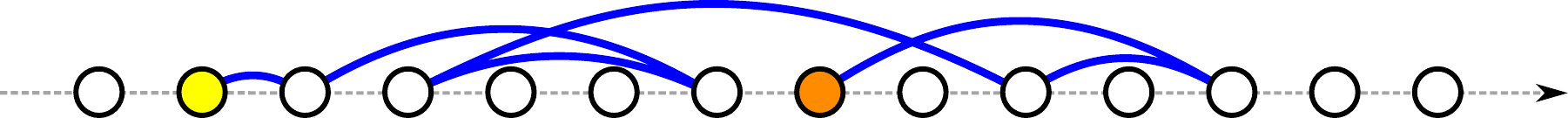}};

  \end{tikzpicture}
  \caption{The blue path witnesses that the yellow vertex is weakly $6$-reachable from the orange vertex. Vertices to the left are smaller in $\leq$.}\label{fig:WReach}
 \end{figure}

To see why a vertex ordering $\leq$ of a graph $G$ with a small weak $d$-coloring number can be used to control short connections $G$, consider the following observation: for any pair of vertices $u,v$, the set $S_{u,v}\coloneqq \WReach^{G,\leq}_d[u]\cap \WReach^{G,\leq}_d[v]$ intersects all paths of length at most $d$ that connect $u$ and $v$. Indeed, for every such path $P$, the $\leq$-smallest vertex of $P$ belongs to $S_{u,v}$. And the size of $S_{u,v}$ is bounded by $\wcol_d(G,\leq)$. Most of the applications of orderings with low weak coloring numbers are based on (often quite involved) variations of this observation.

Weak coloring numbers can be regarded as a generalization of the notion of {\em{degeneracy}} of a graph to higher distances, where the degeneracy of a graph $G$ is equal to $\wcol_1(G)-1$ (i.e., for every vertex we measure the number of neighbors smaller in the ordering). It is known that the degeneracy is tightly connected to the maximum average degree among subgraphs (see \cite[Chapter~1]{sparsityNotes}), hence one might suspect that the weak coloring numbers are tightly connected to the parameters~$\nabla_d$. And indeed, Zhu~\cite{Zhu09} proved that the boundedness of all the parameters $\nabla_d$ is equivalent to the boundedness of all the parameters $\wcol_d$. This yields the following characterization of the notion of bounded expansion.

\begin{theorem}[\cite{Zhu09}]\label{thm:wcol-be}
 A graph class $\Cc$ has bounded expansion if and only if $\wcol_d(\Cc)$ is finite for every~$d\in \N$.
\end{theorem}

We remark that the proof of \cref{thm:wcol-be} can be made effective: there is an $\Oh_{\Cc,d}(n)$-time algorithm that given $G\in \Cc$ and $d\in \N$, outputs a vertex ordering $\leq$ of $G$ such that $\wcol_d(G,\leq)\leq \Oh_{\Cc,d}(1)$~\cite{Dvorak13}.

Let us immediately give one concrete application of \cref{thm:wcol-be} that will be of relevance later:  construction of sparse {\em{neighborhood covers}} in classes of bounded expansion. Intuitively, a neighborhood cover of a graph is a robust system of subgraphs that covers every local neighborhood; the notion comes from the area of distributed algorithms.

\begin{definition}
 Let $G$ be a graph and $d\in \N$. A {\em{distance-$d$ neighborhood cover}} in $G$ of {\em{radius~$r$}} and {\em{overlap~$p$}} is a family $\Ff$ of vertex subsets of $G$ with the following properties:
 \begin{itemize}[nosep]
  \item For every vertex $u\in V(G)$, there exists $F\in \Ff$ such that $\Ball^G_d[u]\subseteq F$.
  \item For every set $F\in \Ff$, the induced subgraph $G[F]$ is connected and has radius at most $r$.
  \item Every vertex $v\in V(G)$ belongs to at most $p$ elements of $\Ff$.
 \end{itemize}
\end{definition}

The parameters measuring the quality of a neighborhood cover are the radius $r$ and the overlap $p$. It turns out that in classes of bounded expansion, we can easily construct distance-$d$ neighborhood covers of radius $2d$ and constant overlap. The following construction is due to Grohe, Kreutzer, Rabinovich, Siebertz, and Stavropoulous~\cite{GroheKRSS18}.

\begin{theorem}[\cite{GroheKRSS18}]\label{thm:nei-cov-be}
 Let $\Cc$ be a class of bounded expansion and $d\in \N$. Then every graph $G\in \Cc$ admits a distance-$d$ neighborhood cover of radius $2d$ and overlap $\Oh_{\Cc,d}(1)$.
\end{theorem}
\begin{proof}[Proof sketch]
 Let $\leq$ be a vertex ordering of $G$ with $\wcol_{2d}(G,\leq)\leq \wcol_{2d}(\Cc)=\Oh_{\Cc,d}(1)$. Then it suffices to construct a neighborhood cover $\Ff$ by taking the inverse weak reachability sets:
 $$\Ff\coloneqq \{\{v\in V(G)~|~u\in \WReach_{2d}^{G,\leq}[v]\}\colon u\in V(G)\}.$$
 It is not hard to verify that $\Ff$ defined in this way is a distance-$d$ neighborhood cover of radius $2d$ and overlap $\wcol_{2d}(G,\leq)$.
\end{proof}

\paragraph*{Low treedepth covers.} We now describe the next characterization, through so-called {\em{low treedepth covers}}. It turns out that graphs from classes of bounded expansion can be robustly covered by graphs of low treedepth so that every constant-size set of vertices is entirely covered by a single element of the cover. This resembles, and is inspired by, the classic Baker's technique from the area of approximation algorithms.

\begin{theorem}[\cite{NPoM-be1}]\label{thm:low-treedepth-covers-be}
 A class of graphs $\Cc$ has bounded expansion if and only if the following condition holds: there is a function $K\colon \N\to \N$ and, for every graph $G\in \Cc$ and $p\in \N$, a family $\Ff$ consisting of at most $K(p)$ vertex subsets of $G$ such that:
 \begin{itemize}[nosep]
  \item for every $A\in \Ff$, the induced subgraph $G[A]$ has treedepth at most $p$; and
  \item for every subset of vertices $X\subseteq V(G)$ with $|X|\leq p$, there exists $A\in \Ff$ such that $X\subseteq A$.
 \end{itemize}
\end{theorem}

A family $\Ff$ satisfying the two conditions described in \cref{thm:low-treedepth-covers-be} is called a {\em{treedepth-$p$ cover}} of~$G$.

We remark that in most of the literature, \cref{thm:low-treedepth-covers-be} is formulated so that it speaks about {\em{low treedepth colorings}}, defined as follows. A treedepth-$p$ coloring of a graph $G$ is a coloring $\lambda$ such that any collection of $i\leq p$ colors induces a graph of treedepth at most $i$; then bounded expansion classes can be characterized by the existence of treedepth-$p$ colorings using at most $M(p)$ colors, for some function $M\colon \N\to \N$. As observed in~\cite{GajarskyKNMPST20}, this definition is essentially equivalent to the notion of covers provided by \cref{thm:low-treedepth-covers-be}: on one hand, a treedepth-$p$ coloring with $M$ colors yields a treedepth-$p$ cover with $\binom{M}{p}$ elements by taking all collections of $p$ colors, and on the other hand, a treedepth-$p$ cover of size $K$ yields\footnote{The argument does not recover the fine bounds on the treedepth of subgraphs induced by $i<p$ colors, but they are usually immaterial in applications.} a treedepth-$p$ coloring with $2^K$ colors by assigning every intersection of the sets of the cover or their complements a different color. We choose to state \cref{thm:low-treedepth-covers-be} in terms of covers, as we feel that this better expresses the character of this notion in the applications.

\cref{thm:low-treedepth-covers-be} was originally proved by Ne\v{s}et\v{r}il and Ossona de Mendez~\cite{NPoM-be1} using {\em{transitive fraternal augmentations}}, a notion similar in spirit to weak coloring numbers. However, an elegant proof that uses weak coloring numbers directly was also found by Zhu~\cite{Zhu09}. His construction is actually very simple. Given a graph $G$ belonging to a bounded expansion class $\Cc$, and $p\in\N$, we first find a vertex ordering $\leq$ of $G$ such that $\wcol_{2^{p-1}}(G,\leq)\leq \wcol_{2^{p-1}}(\Cc)=\Oh_{\Cc,p}(1)$. Then, we apply a greedy coloring procedure that processes the vertices of $G$ in the order $\leq$ and colors them using $\wcol_{2^{p-1}}(G,\leq)$ colors so that every vertex $u$ receives a color different from the other vertices of $\WReach_{2^{p-1}}^{G,\leq}[u]$. It turns out that the coloring of vertices of $G$ constructed in this way is a treedepth-$p$ coloring.

Low treedepth covers are a very elegant tool that allows breaking a problem into first solving the bounded-treedepth case, and then lifting the understanding to the bounded expansion case using covers. Consider, for instance, the {\em{Subgraph Isomorphism}} problem: Given graphs $G$ and $H$, we would like to decide whether $H$ is a subgraph of $G$. We are interested in the parameterized complexity of the problem where $p=|V(H)|$ is the parameter; thus we think of finding a small pattern graph $H$ in a large host graph~$G$. First, if the graph $G$ has treedepth at most $d$, then the problem can be solved in time $\Oh_{p,d}(n)$ using standard dynamic programming techniques. Next, if we want to solve the problem on a graph $G$ that belongs to a fixed class of bounded expansion $\Cc$, then we can first compute a treedepth-$p$ cover $\Ff$ of $G$ with $|\Ff|=\Oh_{\Cc,p}(1)$ (this can be done in time $\Oh_{\Cc,p}(n)$), and then, for each $F\in \Ff$, test whether $G[F]$ contains $H$ as a subgraph using the aforementioned dynamic programming approach. The properties of the cover $\Ff$ ensure us that if $H$ is a subgraph of $G$, then it is also a subgraph of $G[F]$ for some $F\in \Ff$. All in all, this gives an $\Oh_{\Cc,p}(n)$-time fixed-parameter algorithm for the Subgraph Isomorphism problem on any class of bounded expansion $\Cc$.

The methodology sketched above can be lifted to solve also the model-checking problem for $\FO$ on bounded expansion classes. We expand on this in \cref{sec:sparsity-mc}.

\newcommand{\Ss}{\cal S}

\paragraph*{Neighborhood complexity.} Finally, we mention one more property of classes of bounded expansion, this time expressed in terms of the complexity of set systems of neighborhoods that can be defined in them. We need a few definitions.

Let $G$ be a graph, $A$ be a subset of vertices of $G$, and $d\in \N$ be a distance parameter. The {\em{set system of distance-$d$ neighborhoods in $A$}}, denoted $\Ss^G_d(A)$, is the family of subsets of $A$ constructed by including, for every vertex $u$ of $G$, the {\em{distance-$d$ neighborhood}} of $u$ in $A$: the subset of $A$ consisting of all the vertices that are at distance at most $d$ from $u$. Formally,
$$\Ss^G_d(A)\coloneqq \{\,\Ball_d^G[u]\cap A \colon u\in V(G)\,\}.$$
Note that if two or more vertices have the same distance-$d$ neighborhood in $A$, we include this neighborhood in $\Ss^G_d(A)$ only once. In other words, we are interested in the set system of all possible ``traces'' that a radius-$d$ ball can leave on $A$.

Obviously, we always have $|\Ss^G_d(A)|\leq 2^{|A|}$, because there are that many subsets of $A$. This bound clearly can be met in general graphs. It turns out that in classes of bounded expansion, there is a surprisingly strong estimate on the size of $\Ss^G_d(A)$: we have even a {\em{linear bound}}. The following result was proved by Reidl, S\'anchez Villaamil, and Stavropoulous~\cite{ReidlVS19}.

\begin{theorem}[\cite{ReidlVS19}]\label{thm:neicmp-be}
 For every class of bounded expansion $\Cc$, distance parameter $d\in \N$, graph $G\in \Cc$, and subset of vertices $A\subseteq V(G)$, we have
 $$|\Ss^G_d(A)|\leq \Oh_{\Cc,d}(|A|).$$
\end{theorem}

Reidl et al. provide two proofs of \cref{thm:neicmp-be}. One relies on low treedepth colorings\footnote{More precisely, centered colorings, which is a closely related notion.}, and the other is an intricate reasoning involving weak coloring~numbers.

A typical way to obtain upper bounds like the one in \cref{thm:neicmp-be} would be to bound the {\em{VC dimension}} of the set system $\Ss^G_d(A)$ by some constant $c=\Oh_{\Cc,d}(1)$, and then conclude from the Sauer-Shelah Lemma~\cite{sauer1972density,shelah1972combinatorial} that $|\Ss^G_d(A)|\leq |A|^c$. This indeed can be done (and the boundedness of the VC dimension of $\Ss^G_d(A)$ follows from \cref{thm:neicmp-be}), but the bound of \cref{thm:neicmp-be} is stronger: it is linear instead of just polynomial. Thus, from the point of view of the theory of VC dimension, set systems of balls in bounded expansion classes have asymptotic growth much lower than what the standard Sauer-Shelah Lemma would~indicate.

We note that the property expressed in \cref{thm:neicmp-be}, called {\em{linear neighborhood complexity}}, is {\em{not}}\footnote{It is for monotone classes, see~\cite{ReidlVS19}.} a characterization of bounded expansion. In fact, the property of having (almost) linear neighborhood complexity is enjoyed by many ideals of classes of dense graphs, including classes of bounded twin-width~\cite{BonnetFLP24,BonnetKRTW22,Przybyszewski23} and monadically stable classes~\cite{DreierEMMPT24}, which we will discuss in \cref{sec:forbidding}. In general, the set system viewpoint introduced by \cref{thm:neicmp-be} appears to be very important in the constructed theory, as it brings along many deep techniques from discrete geometry, particularly the theory of VC dimension.

\subsubsection{Characterizations: nowhere denseness}\label{sec:char-nd}

We now proceed to nowhere dense graph classes. Recall that these are defined through the boundedness of the parameters $\omega_d$, measuring the largest size of a complete graph contained as a depth-$d$ minor. Note that at first glance, it is not at all clear why such classes should consist of sparse graphs. This is in fact a non-trivial statement, proved by Dvo\v{r}\'ak~\cite{dvorak-thesis}.

\begin{theorem}[\cite{dvorak-thesis}]\label{thm:nd-density}
 Let $\Cc$ be a nowhere dense class of graphs and $\eps>0$ be a positive real. Then for every graph $G\in \Cc$, we have
 $$|E(G)|\leq \Oh_{\Cc,\eps}(|V(G)|^{1+\eps}).$$
\end{theorem}

Thus, the number of edges is almost linear in the number of vertices. This means that the average degree is bounded {\em{subpolynomially}}: by $\Oh_{\Cc,\eps}(|V(G)|^{\eps})$ for any fixed $\eps>0$, but not necessarily by a constant. In fact, there are nowhere dense classes where the average degree is unbounded; the class discussed in \cref{thm:nd-be-example} is one of them.

Note that if for a class $\Cc$ and $h\in \N$, by $\mathsf{Minors}_h(\Cc)$ we denote the class of all depth-$h$ minors of graphs from $\Cc$, then $\Cc$ being nowhere dense entails that $\mathsf{Minors}_h(\Cc)$ is nowhere dense as well. Indeed, for every $d\in \N$ we have $\omega_d(\mathsf{Minors}_h(\Cc))\leq \omega_{2hd+h+d}(\Cc)$, for one can easily see that a depth-$d$ minor of a depth-$h$ minor is a depth-$(2hd+h+d)$ minor. So from \cref{thm:nd-density} we can immediately derive the following.

\begin{corollary}\label{thm:nd-grads}
 Let $\Cc$ be a nowhere dense class of graphs, $h\in \N$ be a distance parameter, and $\eps>0$ be a positive real. Then for every graph $G\in \Cc$, we have
 $$\nabla_h(G)\leq \Oh_{\Cc,h,\eps}(|V(G)|^{\eps}).$$
\end{corollary}

The bound on the average degrees in shallow minors provided by \cref{thm:nd-grads} can be used to lift all the characterizations and properties mentioned in \cref{sec:char-be} from the setting of bounded expansion to the setting of nowhere denseness. The statement below summarizes the counterparts of \cref{thm:wcol-be,thm:nei-cov-be,thm:low-treedepth-covers-be,thm:neicmp-be} for nowhere dense classes.

\begin{theorem}\label{thm:quant-nd}
 Let $\Cc$ be a nowhere dense class of graphs, $d,p\in \N$ be parameters, and $\eps>0$ be a positive real. Then for every graph $G\in \Cc$, we have the following:
 \begin{itemize}[nosep]
 \item $\wcol_d(G)\leq \Oh_{\Cc,d,\eps}(|V(G)|^\eps)$.
 \item $G$ has a distance-$d$ neighborhood cover of radius $2d$ and overlap $\Oh_{\Cc,d,\eps}(|V(G)|^\eps)$.
 \item $G$ has a treedepth-$p$ cover of size $\Oh_{\Cc,p,\eps}(|V(G)|^\eps)$.
 \item For every $A\subseteq V(G)$, we have $|\Ss^G_d(A)|\leq \Oh_{\Cc,d,\eps}(|A|^{1+\eps})$.
 \end{itemize}
\end{theorem}

To prove the first three points above, it suffices to inspect the proofs of \cref{thm:wcol-be,thm:nei-cov-be,thm:low-treedepth-covers-be} and verify that the relevant quantities depend polynomially on the parameters $\nabla_h(G)$. Indeed, then the $\Oh_{\Cc,h,\eps}(|V(G)|^\eps)$ bound on $\nabla_h(G)$, provided by \cref{thm:nd-grads}, implies a bound of the same form on the quantity in question, after rescaling $\eps$. The proof of the last point --- of the neighborhood complexity, proposed by Eickmeyer, Giannopoulou, Kreutzer, Kwon, Pilipczuk, Rabinovich, and Siebertz~\cite{EickmeyerGKKPRS17} --- is more complicated. The reason is that the proof of Reidl et al.~\cite{ReidlVS19} produces factors exponential in parameters $\nabla_h(G)$ (more precisely, in the weak coloring numbers of $G$). In~\cite{EickmeyerGKKPRS17}, tools from the theory of VC dimension were used to reduce this exponential dependence to polynomial.

All in all, as witnessed by \cref{thm:quant-nd}, \cref{thm:nd-grads} can be used to obtain properties of nowhere dense classes of quantitative nature, typically involving the not-so-elegant term $\Oh_{\Cc,d,\eps}(|V(G)|^{\eps})$. However, the original definition is more of a qualitative nature, as it speaks about universal boundedness of  parameters~$\omega_d$, and not about bounds of the form $\Oh_{\Cc,d,\eps}(|V(G)|^{\eps})$. We now present two other characterizations of nowhere denseness of quantitative nature: through {\em{flatness}} ({\em{aka}} {\em{uniform quasi-wideness}}) and through {\em{Splitter Game}}. These characterizations are of immense importance for the~theory.

\paragraph*{Flatness.} The intuition expressed by flatness is the following: in a huge sparse graph there are many vertices that are pairwise far from each other. This is not necessarily exactly true: consider a star. However, it turns out that this  intuition can be turned into an actual characterization of nowhere dense classes, provided we allow deletion of a bounded number of vertices. Formally, we have the following~definition.

\begin{definition}\label{def:flatness}
 A graph class $\Cc$ is {\em{flat}} if for every $d\in \N$, there exists a constant $s_d\in \N$ and a function $N_d\colon \N\to \N$ such that the following condition holds. For every graph $G\in \Cc$ and a vertex subset $A\subseteq V(G)$ with $|A|>N_d(m)$ for some $m\in \N$, there exist $S\subseteq V(G)$ with $|S|\leq s_d$ and $I\subseteq A\setminus S$ with $|I|>m$ such that the vertices of $I$ are pairwise at distance larger than $d$ in the graph $G-S$.
\end{definition}

Observe that the condition is actually somewhat stronger than the intuition phrased above: we can find a set $I$ consisting of many pairwise far vertices not only in the whole graph $G$, but even within any specified subset $A$, provided it is large enough. Note also that the bound $s_d$ on the number of vertices allowed to be deleted does {\em{not}} depend on $m$, the requested lower bound on the size of $I$.

Flatness in the exact form stated in \cref{def:flatness} was first used by Ne\v{s}et\v{r}il and Ossona de Mendez in~\cite{NPoM-nd-uqw}, but the notion has its roots in the earlier work of Atserias, Dawar, Grohe and Kolaitis~\cite{AtseriasDG08,AtseriasDK06,Dawar07,Dawar10,Dawar24} on preservation theorems; see also the discussion in~\cite{DawarK09}. In fact, one can argue that the notion was described in the context of infinite graphs by Podewski and Ziegler~\cite{podewski1978stable} already in 1978.
We remark that Ne\v{s}et\v{r}il and Ossona de Mendez~\cite{NPoM-nd-uqw}, as well as a vast majority of the existing literature on Sparsity, use the term {\em{uniform quasi-wideness}} instead of flatness. In this survey, we follow the recent trend of using {\em{flatness}} instead, because otherwise naming suitable generalizations to dense graphs, discussed in \cref{sec:forbidding}, would become cumbersome.

At first glance, the nature of flatness seems to be very different than that of nowhere denseness. It is therefore surprising that the notions actually coincide, as proved by Ne\v{s}et\v{r}il and Ossona de Mendez~\cite{NPoM-nd-uqw}.

\begin{theorem}[\cite{NPoM-nd-uqw}]\label{thm:flatness-nd}
 A graph class is nowhere dense if and only if it is flat.
\end{theorem}

The proof of \cref{thm:flatness-nd} proceeds by induction on the distance parameter $d$ and applies Ramsey's Theorem at each step. In fact, flatness itself can be regarded a property of Ramseyan nature. In applications, it is typically used to find large better-structured objects within huge worse-structured objects.

\paragraph*{Splitter Game.} Finally, we describe the game characterization of nowhere denseness. This is provided by the bounded-radius variant of the game for treedepth that we discussed in \cref{sec:treedepth}. Concretely, recall that the game is played on a graph, called the {\em{arena}}, and there are two players: Splitter and Connector. This time, there will be also a distance parameter $d\in \N$; so we will speak about the {\em{radius-$d$ Splitter Game}}. In every round, first Connector selects any vertex $u$ and the arena gets shrunk to the subgraph induced by the radius-$d$ ball around $u$. Then Splitter removes any vertex of his choice from the arena. The game finishes with Splitter's win when the arena gets empty, the goal of Splitter is to terminate the game in as few rounds as possible, and Connector's goal is the opposite.

Thus, in this bounded-radius variant, in every round Connector is forced to ``localize'' the arena to a ball of bounded radius. It turns out that with this amendment, the game exactly characterizes nowhere dense classes, as proved by Grohe, Kreutzer, and Siebertz~\cite{GroheKS17}.

\begin{theorem}[\cite{GroheKS17}]\label{thm:splitter-game}
A graph class $\Cc$ is nowhere dense if and only if for every $d\in \N$ there exists $k\in \N$ such that for every graph $G\in \Cc$, Splitter can win the radius-$d$ Splitter Game on $G$ within at most $k$ rounds.
\end{theorem}

The proof of \cref{thm:splitter-game} proposes a simple, concrete strategy for Splitter, and uses the characterization through flatness (\cref{thm:flatness-nd}) to argue that when deployed on a graph from a nowhere dense class, this strategy will lead to a win within a bounded number of rounds. In particular, every next move of Splitter in this strategy can be computed in linear time.

On the combinatorial level, the game tree of the radius-$d$ Splitter Game played on a graph $G$ provides a kind of a hierarchical decomposition of $G$ into smaller and smaller ``vicinities'', where the depth of the decomposition is bounded by a constant --- the length of the game. We will make this observation a bit more precise in \cref{sec:sparsity-mc} when discussing the algorithm for model-checking $\FO$ on nowhere dense~classes.

\subsubsection{Model-checking}\label{sec:sparsity-mc}

It turns out that the set of tools introduced for classes of bounded expansion and for nowhere dense classes can be robustly applied to the algorithmic treatment of the $\FO$ model-checking problem. This was investigated for bounded expansion classes by Dvo\v{r}\'ak, Kr\'al', and Thomas in~\cite{DvorakKT13}, and for nowhere dense classes by Grohe, Kreutzer, and Siebertz in~\cite{GroheKS17}. We now discuss these two results in order, as the approaches differ significantly and both have their merits.

\paragraph*{Classes of bounded expansion.}
The result of Dvo\v{r}\'ak et al. is the following.

\begin{theorem}[\cite{DvorakKT13}]\label{thm:be-mc}
 For every class of bounded expansion $\Cc$ there is an algorithm that given an $\FO$ sentence~$\varphi$ and a graph $G\in \Cc$, decides whether $G\models \varphi$ in time $\Oh_{\Cc,\varphi}(n)$.
\end{theorem}

A complexity remark is necessary here. The constant depending on $\Cc$ and $\varphi$ hidden in the $\Oh_{\Cc,\varphi}(\cdot)$ notation above ultimately depends on the parameter $\nabla_d(\Cc)$, for some $d$ depending (in a computable fashion) on $\phi$. Therefore, to adhere to the strict definition of fixed-parameter tractability and have this constant depend in a computable way on $\varphi$, we need to assume that the function $d\mapsto \nabla_d(\Cc)$ can be upper-bounded by a computable function. We call classes with this property of {\em{effectively bounded expansion}}, and fixed-parameter tractability in the strict sense can be stated only for those classes. The same remark applies also to the algorithms for nowhere dense classes (\cref{thm:nd-mc}) and for monadically stable classes (\cref{thm:ms-mc}).

The original proof of \cref{thm:be-mc} of Dvo\v{r}\'ak et al. relies on {\em{transitive fraternal augmentations}}, a notion similar in spirit to weak coloring numbers. Later, the approach was reworked using low treedepth covers to make it more modular, firstly by Grohe and Kreutzer~\cite{grohe2011methods} and then by Pilipczuk, Siebertz, and Toru\'nczyk~\cite{PilipczukST18}. See also the works of Kazana and Segoufin~\cite{KazanaS19} and of Toru\'nczyk~\cite{Torunczyk20} that focus on generalizations relevant for problems with database motivations. In our description here, we follow the presentation from~\cite[Chapter 3]{sparsityNotes},  which largely relies on~\cite{PilipczukST18}.

In all the works above, the main idea is to apply {\em{quantifier elimination}}: starting with the graph $G$ and the sentence $\varphi$, we iteratively enrich $G$ with more information, and at the same time we simplify $\phi$ by removing consecutive quantifiers. Eventually $\phi$ becomes very simple and its satisfaction can be tested directly. The enrichment of $G$ with more information can be regarded as a process of spreading information about the existence of witnesses for certain formulas. This information spread is performed along orientations of transitive fraternal augmentations in~\cite{DvorakKT13,KazanaS19}, and along elimination forests of subgraphs induced by elements of low treedepth covers in~\cite{grohe2011methods,PilipczukST18,Torunczyk20}.

The advantage of this approach is that it actually provides a much stronger result than just an algorithm for model-checking. Namely, quantifier elimination can be applied to any formula, even with free variables, to reduce it to an equivalent formula of a simple form, at the cost of adding more information to the graph. We now explain the outcome of quantifier elimination using the terminology from~\cite[Chapter 3]{sparsityNotes}.

The output of quantifier elimination will have the form of a so-called {\em{pointer structure}}, which is a notion that is simple, but does not exactly fit within the framework of relational structures. Namely, the signature $\Sigma$ of a pointer structure consists of some unary predicates and some {\em{unary function symbols}}. A pointer structure $\Af$ consists of a universe $U(\Af)$, a unary relation $R^\Af\subseteq U(\Af)$ for each unary predicate $R\in \Sigma$, and a function $f^{\Af}\colon U(\Af)\to U(\Af)$ for each function symbol $f\in \Sigma$. It is useful to think of those functions as pointers, hence the name. Then in $\FO$ on pointer structures, we allow applying functions to variables. For instance, we can write atomic formulas of the form $R(f(x))$ or $f(g(x))=g(f(y))$. Note that in this way, even if a formula $\varphi(\tup x)$ is quantifier-free (i.e., does not contain any quantifiers), it may speak not only about the relations between the elements of $\tup x$, but also between their images under application of functions. The {\em{Gaifman graph}} of a pointer structure is defined naturally, by putting an edge between $u$ and $f^{\Af}(u)$ for every element $u$ and function symbol $f$.

We can now state the quantifier elimination result in full~generality.

\begin{theorem}\label{thm:quantifier-elimination}
 Let $\Cc$ be a class of bounded expansion and $\varphi(\tup x)$ be an $\FO$ formula. Then there is a formula $\widehat{\varphi}(\tup x)$ and an $\Oh_{\Cc,\phi}(n)$-time algorithm that given a graph $G\in \Cc$, outputs a pointer structure $\widehat{G}$ such that the following conditions hold:
 \begin{itemize}[nosep]
  \item The universe of $\widehat{G}$ is the vertex set of $G$ and the Gaifman graph of $\widehat{G}$ is a subgraph of $G$. That is, for every vertex $u$ and function $f$ present in $\widehat{G}$, either $u=f(u)$ or $u$ and $f(u)$ are adjacent in~$G$.
  \item $\widehat{\varphi}(\tup x)$ is a quantifier-free $\FO$ formula over the signature of $\widehat{G}$.
  \item $\phi(\tup x)$ selects in $G$ exactly the same tuples as $\widehat{\phi}(\tup x)$ in $\widehat{G}$. That is, we have
  $$G\models \phi(\tup u)\quad\textrm{if and only if}\quad \widehat{G}\models \widehat{\phi}(\tup u)\qquad \textrm{for all }\tup u\in V(G)^{\tup x}.$$
 \end{itemize}
\end{theorem}

For readers with an algorithmic background, it may be instructive to think of the pointer structure output by the algorithm of \cref{thm:quantifier-elimination} as of a {\em{data structure}} for answering $\phi$-queries on the input graph~$G$. Namely, after computing $\widehat{G}$ (which is done as preprocessing in fixed-parameter linear time), every query of the form ``Given $\tup u\in V(G)^{\tup x}$, does $G\models \varphi(\tup u)$?'' can be reduced to deciding whether $\widehat{G}\models \widehat{\phi}(\tup u)$, which can be done in constant time: just follow the pointers and check relations.

\paragraph*{Nowhere dense classes.} The approach behind the proofs of \cref{thm:be-mc,thm:quantifier-elimination} seems to fundamentally fail in the nowhere dense case; or at least it is entirely unclear how to make it work. The reason is that once the relevant parameters (outdegrees in appropriate orientations, weak coloring numbers, sizes of low treedepth covers, etc.) are bounded by $\Oh_{\Cc,\phi,\eps}(n^\eps)$ instead of by a constant, the bounds during the quantifier elimination start to explode very quickly, yielding super-exponential functions already after a few steps. Therefore, a different approach is needed. This approach was proposed by Grohe, Kreutzer, and Siebertz~\cite{GroheKS17}, who proved the following result.

\begin{theorem}[\cite{GroheKS17}]\label{thm:nd-mc}
 For every nowhere dense class of graphs $\Cc$ there is an algorithm that given an $\FO$ sentence $\phi$ and a graph $G\in \Cc$, decides whether $G\models \phi$. The running time can be bounded by $\Oh_{\Cc,\phi,\eps}(n^{1+\eps})$, for any fixed~$\eps>0$.
\end{theorem}

The approach of Grohe et al. fundamentally relies on the characterization of nowhere dense classes through the Splitter Game (\cref{thm:splitter-game}). In fact, this characterization was devised with the proof of \cref{thm:nd-mc} in mind. Let $d\in \N$ be a parameter depending on the input sentence $\phi$ (roughly, $d$ is exponential in the quantifier rank of $\phi$) and consider the radius-$d$ Splitter Game played on the input graph $G$. We construct the {\em{game tree}} $T$ assuming the fixed strategy of Splitter provided by \cref{thm:splitter-game}. That is, nodes of $T$ correspond to positions in the game just before the next move of Splitter or Connector, and:
\begin{itemize}[nosep]
 \item every node corresponding to a Splitter's position ({\em{Splitter's node}}) has exactly one child, in which the move suggested by the strategy is executed; and
 \item every node corresponding to a Connector's position ({\em{Connector's node}}) has as many children as the number of vertices of the current arena, and in each child a different move of Connector is executed.
\end{itemize}
Note that $T$ has depth bounded by $k$ --- the constant upper bound on the length of the game --- but the branching at each Connector's node can be as high as $n$. Therefore, if we perform the construction just as above, $T$ may have as many as $n^k$ nodes and it is too large to be useful in any fixed-parameter algorithm.

The idea is to make $T$ significantly smaller by restricting the moves of Connector to playing, instead of radius-$d$ balls, elements of the distance-$d$ neighborhood cover with radius $2d$ and overlap $\Oh_{\Cc,d,\eps}(n^\eps)$ provided by \cref{thm:quant-nd}, second point. Observe that with such a restriction, Connector has still as much freedom as she has in the standard radius-$d$ game, because for every radius-$d$ ball she can play an element of the cover containing this ball. However, Splitter can still win within a bounded number of rounds, by playing the strategy for the radius-$2d$ game. The bound on the overlap of the neighborhood cover can be used to argue that the whole game tree $T$ has now size $\Oh_{\Cc,d,\eps}(n^{1+\eps})$, and it can be efficiently~computed.

Having computed $T$, the algorithm of \cref{thm:nd-mc} applies a bottom-up procedure that aggregates relevant information about larger and larger parts of the graph (arenas corresponding to the nodes); this can be viewed as a sort of dynamic programming. Namely, one uses locality of First-Order logic to argue that the information about arena $G_x$ at a node $x$ can be compiled from the information about local neighborhoods in $G_x$, which are exactly the arenas in the children of $x$. Once the information about the root of $T$ is computed, this is sufficient to answer whether $G\models \phi$.

This concludes the sketch of the proof of \cref{thm:nd-mc}. We note that the reasoning applies directly to the model-checking problem, and in particular does not provide any form of quantifier elimination. Obtaining a suitable analogoue of \cref{thm:quantifier-elimination} for nowhere dense classes remains open.

Finally, we note that as observed by Dvo\v{r}\'ak et al.~\cite{DvorakKT13}, no further generalization of \cref{thm:nd-mc} should be expected on monotone (i.e., subgraph-closed) classes.

\begin{theorem}[\cite{DvorakKT13}]\label{thm:nd-limit}
 Let $\Cc$ be a class of graphs that is monotone and not nowhere dense. Then the model-checking problem for $\FO$ on $\Cc$ is as hard as on general graphs, that is, $\mathsf{AW}[\star]$-hard.
\end{theorem}

\section{New concepts}\label{sec:new}

In \cref{sec:classic} we have discussed a wide range of different concepts of structure in graphs. These concepts were predominantly motivated by graph-theoretic considerations. However, it turned out that many of them also possess good model-theoretic properties, witnessed by the existence of efficient algorithms for the model-checking problem, or by closure properties with respect to transductions. In particular, we have already discovered four $\FO$ ideals:
\begin{itemize}[nosep]
 \item classes of bounded shrubdepth (\cref{cor:shrub-ideal});
 \item classes of bounded linear cliquewidth (\cref{thm:cw-lcw-ideals});
 \item classes of bounded cliquewidth (\cref{thm:cw-lcw-ideals}); and
 \item classes of bounded twin-width (\cref{thm:tww-ideal}).
\end{itemize}
If our goal is to construct a robust mathematical theory for describing the ``First-Order complexity'' of graphs, arguably one should ground such a theory in notions that are fundamentally of logical nature. In particular, the landscape presented so far has a gaping hole in that we have not seen suitable dense analogues of the notions of bounded expansion and of nowhere denseness. Based on purely graph-theoretic considerations, it is rather unclear how such notions should be constructed.

So the proposition for a research programme is to:
\begin{itemize}[nosep]
 \item adopt the notion of $\FO$ transductions as the basic embedding notion;
 \item define relevant properties of graph classes using the transduction order, and in particular propose $\FO$ ideals that should correspond to classes of bounded expansion and to nowhere dense classes;~and
 \item understand these $\FO$ ideals on the grounds of graph theory, and develop decomposition tools for working with them.
\end{itemize}
In this section we describe the recent attempts to make first steps within this programme.

As transductions give a quasi-order on graph classes, there are two natural ways to construct  $\FO$~ideals:
\begin{itemize}[nosep]
 \item {\em{Obstructions:}} If $\Dd$ is a graph class, then the set of all graph classes that do not transduce $\Dd$ is an $\FO$~ideal.
 \item {\em{Closure:}} If $\Pi$ is any set of graph classes, then the set comprising all graph classes that are transducible from any class belonging to $\Pi$ is an $\FO$ ideal.
\end{itemize}
For the latter point, if we interpret $\Pi$ as a property of graph classes, then classes transducible from classes satisfying $\Pi$ are called {\em{structurally $\Pi$}}. A typical example of this construction is closing a property of classes of sparse graphs under transductions; so we may speak, for instance, about classes of structurally bounded treewidth, structurally bounded expansion classes, structurally nowhere dense classes, etc. Moreover, note that the intersection of any set of $\FO$ ideals is again an $\FO$ ideal.

We divide our further discussion into two veins: ideals defined through obstructions, and ideals defined through closure. However, the two viewpoints will often interleave.

\subsection{Ideals defined through obstructions}\label{sec:forbidding}

The most obvious choice for an ideal defined by obstructions is to forbid transducibility of the class of all~graphs.

\begin{definition}
 A graph class $\Cc$ is {\em{monadically dependent}}\footnote{The term {\em{monadically NIP classes}} is also often used in the literature, where NIP stands for {\em{Not the Independence Property}}. Monadic dependence and monadic NIP are synonyms.} if the class of all graphs cannot be transduced from $\Cc$.
\end{definition}

To break this definition further, $\Cc$ is monadically dependent if for every fixed transduction $\Tf$, $\Tf(\Cc)$ does not contain all graphs. In other, more intuitive words, there is no fixed $\FO$-definable mechanism that allows the encoding of all graphs in colored graphs from $\Cc$. Note here that more and more complicated transductions $\Tf$ may yield larger and larger sets $\Tf(\Cc)$, but the requirement is that none of them encompasses all graphs. On high level, this is similar to the gradation-by-distance character of the definition of nowhere denseness (\cref{def:nd}). As we will see, this seems not to be a coincidence.

Thus, monadic dependence is the weakest possible restriction that can be made in the theory of transductions. Indeed, every $\FO$ ideal that is not equal to all the graph classes must be contained in the ideal of monadically dependent classes. For instance, classes of bounded twin-width are monadically dependent.

In fact, we have already come across the notion of monadic dependence in disguise several times:
\begin{itemize}[nosep]
 \item \cref{lem:rook-not-dependent} proves that the class of rook graphs is not monadically dependent. Observe that thus, the reasoning may serve as a model-theoretic proof that rook graphs do not have bounded twin-width.
 \item Naturally, the notion of monadic dependence can be defined for logics other than $\FO$. Then \cref{thm:cw-dependent} is equivalent to saying that a class of graphs $\Cc$ is monadically dependent with respect to the $\CMSO$ logic if and only if $\Cc$ has bounded cliquewidth. Brushing aside the slight fuzziness of the notion of $\MSO_2$ transductions, \cref{thm:tw-model} can be interpreted as saying that a class of graphs $\Cc$ is monadically dependent with respect to the $\MSO_2$ logic if and only if $\Cc$ has bounded treewidth. This seems to be the most compelling formal explanation of the intuition that the boundedness of cliquewidth delimits the region of tractability of $\CMSO$, and the boundedness of treewidth delimits the region of tractability of $\MSO_2$.
 \item \cref{thm:tww-dependent} states that a class $\Cc$ of ordered graphs is monadically dependent if and only if $\Cc$ has bounded twin-width.
\end{itemize}
Thus, we already have a robust combinatorial understanding of monadically dependent classes with respect to logics $\CMSO$ or $\MSO_2$, or of ordered graphs; this understanding is delivered through suitable duality theorems providing decompositions for those classes (tree decompositions, laminar decompositions, contraction sequences, etc.). The structural understanding of graph classes that are monadically dependent with respect to $\FO$ is still quite incomplete, but we will present the recent advances in this direction in \cref{sec:mon-dependent}.

\medskip

Next, it turns out that a very important $\FO$ ideal is obtained by forbidding transducibility of the class of {\em{half-graphs}} (\cref{def:half-graph}).

\begin{definition}
A graph class $\Cc$ is {\em{monadically stable}} if the class of all half-graphs cannot be transduced from $\Cc$.
\end{definition}

Obviously, every monadically stable graph class is also monadically dependent, but not vice-versa, as witnessed by the class of half-graphs themselves. Also, as half-graphs have bounded linear cliquewidth, the ideal of monadically stable classes does not contain the ideal of classes of bounded linear cliquewidth. But it does contain the ideal of classes of bounded shrubdepth, thanks to \cref{lem:half-graphs-shb}. As we will see later, monadically stable classes form an ideal that is in some sense orthogonal to the hierarchy of ideals defined by shrubdepth, linear cliquewidth, cliquewidth, and twin-width.

The reason  why forbidding transducibility of half-graphs turns out to be a fundamental property is that half-graphs are combinatorial encodings of linear orders. More formally, the classes of half-graphs and of linear orders (treated as binary structures) can be transduced from each other, so monadic stability can be equivalently defined through postulating non-transducibility of the class of linear orders. Thus, the intuition is that graphs from a monadically stable class are fundamentally ``orderless'': one cannot $\FO$-define in them any sizeable linear order, even if coloring vertices is allowed.

Finally, let us mention that as proved by Ne\v{s}et\v{r}il, Ossona de Mendez, Pilipczuk, Rabinovich, and Siebertz~\cite{NesetrilMPRS21}, within the context of monadically dependent classes, monadic stability can be characterized purely combinatorially, by exclusion of half-graphs as {\em{semi-induced subgraphs}}. Here, for a graph $G$ and a bipartite graph $H$, we say that $G$ contains $H$ as a {\em{semi-induced subgraph}} if there exist disjoint vertex subsets $A,B\subseteq V(G)$ such that the subgraph of $G$ consisting of $A$, $B$, and all the edges with one endpoint in $A$ and second in $B$, is isomorphic to $H$. Then, Ne\v{s}et\v{r}il et al. proved the~following using arguments inspired by the work of Baldwin and Shelah~\cite{baldwin1985second}.

\begin{theorem}[\cite{NesetrilMPRS21}]\label{thm:ms-es}
 Let $\Cc$ be a monadically dependent class of graphs. Then $\Cc$ is monadically stable if and only if $\Cc$ excludes some half-graph as a semi-induced subgraph.
\end{theorem}

\paragraph*{Relation to nowhere denseness.} Monadic dependence and monadic stability are concepts that were first studied in model theory, particularly within the stability theory developed by Shelah~\cite{shelah1990classification}. See also the books of Pillay~\cite{pillay} and of Tent and Ziegler~\cite{tentZiegler} for an introduction to the area. There, typically one considers dependence or stability of a single infinite model (in our understanding, an infinite graph), rather than of classes of finite models, as we do in our theory. While there are ways of formally translating results between the finite and the infinite setting (for instance, compactness, ultrafilters, or Łoś' Theorem), one can also try to understand how certain techniques work on infinite models, and then try to emulate those techniques as purely combinatorial arguments in classes of finite models. Recent developments on monadically stable and monadically dependent classes of graphs feature both these types of interactions.

Related to the above, the adjective ``monadic'' in monadic stability and monadic dependence signifies that we forbid interpretability of obstructions after adding arbitrary unary predicates to the structure. In standard dependence or stability, one would consider interpretability of obstructions in the structure alone (but also allow multi-dimensional interpretations). The theory of Shelah explains the role of stability in model theory. The particular notion of monadic stability was studied by Baldwin and Shelah in~\cite{baldwin1985second}, and this work provides a wealth of inspiration for our theory.

However, so far monadic stability and monadic dependence are just two abstract definitions, borrowed from model theory. The surprising link between those notions and the concepts discussed previously is delivered by the following result.

\begin{theorem}[\cite{adler2014interpreting,Dvorak18,podewski1978stable}]\label{thm:nd-monst}
 Every nowhere dense class is monadically stable. Moreover, every monadically dependent class that is weakly sparse, is actually nowhere dense.
\end{theorem}

\cref{thm:nd-monst} was proved in the infinite setting by Podewski and Ziegler~\cite{podewski1978stable} as early as in 1978, so roughly 30 years before the introduction of the notion of nowhere denseness by Ne\v{s}et\v{r}il and Ossona de Mendez~\cite{NPoM-nd}. In fact, Podewski and Ziegler call (infinite) nowhere dense graphs {\em{superflat}}. Also, for their proof they introduce a notion called {\em{property $(\star)$}} that in our terms is equivalent to flatness (\cref{def:flatness}). They proved the equivalence of property $(\star)$ and superflatness, and derived \cref{thm:nd-monst} from there. The applicability of the work of Podewski and Ziegler to the setting of nowhere dense classes of (finite) graphs was observed by Adler and Adler~\cite{adler2014interpreting}, who proved \cref{thm:nd-monst} in the form above, except the collapse of monadic dependence to nowhere denseness was stated only for monotone classes. That the collapse occurs also for weakly sparse classes follows from a later work of Dvo\v{r}\'ak~\cite{Dvorak18}. We also note that the analogue of \cref{thm:nd-monst} for relational structures was proved recently by Braunfeld, Dawar, Eleftheriadis, and Papadopoulos~\cite{BraunfeldDEP23}.

Intuitively, \cref{thm:nd-monst} shows that nowhere denseness is a ``shadow'' of the more general notions of monadic stability and monadic dependence, obtained by restricting attention to weakly sparse classes. Therefore, one may suspect that many combinatorial and algorithmic results that hold for nowhere dense classes, can be in fact lifted to monadically stable or monadically dependent classes. In particular, this leads to the following conjecture.

\begin{conjecture}\label{conj:main}
 The $\FO$ model-checking problem is fixed-parameter tractable on every monadically dependent class of graphs. More precisely, for every monadically dependent class $\Cc$ there is a constant $c\in \N$ and an algorithm that, given a graph $G\in \Cc$ and an $\FO$ sentence $\varphi$, decides whether $G\models \varphi$ in time $\Oh_{\Cc,\phi}(n^c)$.
\end{conjecture}

In \cref{conj:main}, we allow the constant $c$ to depend on the class $\Cc$. However, one could envision, and indeed expect, a stronger statement where $c$ is a universal constant, independent of $\Cc$.

\cref{conj:main} has emerged naturally in various forms after the dissemination of the work of Adler and Adler~\cite{adler2014interpreting}. The two earliest concrete mentions (with somewhat different formulations) are due to Toru\'nczyk~\cite{Warwick16} and to Gajarsk\'y, Hlin\v{e}n\'y, Obdr\v{z}\'alek, Lokshtanov, and Ramanujan~\cite{GajarskyHOLR20}.

At the point of writing this survey, \cref{conj:main} remains open and presents itself as the main algorithmic goal of the theory under development. In particular, it was recently confirmed for monadically stable classes (\cref{thm:ms-mc}). In the next two sections, we present a variety of structural tools for monadically stable and monadically dependent classes that were developed recently. These were admittedly in large part motivated by the work on \cref{conj:main}, but they also present a combinatorial understanding of those classes that is interesting on its~own.

\subsubsection{Monadic stability}\label{sec:mon-stable}

\paragraph*{Characterizations.}
We start by presenting the toolbox for monadically stable classes, which arguably already now provides a comprehensive combinatorial description. First, recall that in \cref{sec:char-nd} we have discussed two qualitative characterizations of nowhere denseness: through flatness and through the Splitter Game. It turns out that both these concepts can be lifted to characterizations of monadically stable classes, and the key to this is the principle that we have already seen in \cref{sec:treedepth} in the context of treedepth and shrubdepth: replace the operation of vertex deletion with the operation of applying a flip.

Recall that a {\em{$k$-flip}} of a graph $G$ is any graph $G'$ that can be obtained from $G$ by applying at most $k$ flip operations, that is, operations of replacing the adjacency relation on a vertex subset with its complement. Then the lift of the notion of flatness reads as follows.

\begin{definition}\label{def:flip-flatness}
 A graph class $\Cc$ is {\em{flip-flat}} if for every $d\in \N$, there exists a constant $k_d\in \N$ and a function $N_d\colon \N\to \N$ such that the following condition holds. For every graph $G\in \Cc$ and a vertex subset $A\subseteq V(G)$ with $|A|>N_d(m)$ for some $m\in \N$, there exist a $k_d$-flip $G'$ of $G$ and a subset $I\subseteq A$ with $|I|>m$ such that the vertices of $I$ are pairwise at distance larger than $d$ in $G'$.
\end{definition}

Note that flip-flatness is a purely combinatorial notion, there is no logic involved in the definition. Yet, as proved by Dreier, M\"ahlmann, Siebertz, and Toru\'nczyk~\cite{DreierMST23}, this notion exactly characterizes monadically stable classes.

\begin{theorem}[\cite{DreierMST23}]\label{thm:flip-flatness}
 A graph class is monadically stable if and only if it is flip-flat.
\end{theorem}

On high level, the proof of \cref{thm:flip-flatness} follows the inductive template of the proof of \cref{thm:flatness-nd}. Yet, it is far more involved and relies on tools borrowed from model theory, in particular the concept of {\em{indiscernible sequences}}.

Next, we have already seen how the Splitter Game can be lifted: in \cref{sec:shrubdepth} we have discussed the radius-$\infty$ Flipper Game that exactly characterizes the graph parameter SC-depth, which in turn is functionally equivalent to shrubdepth, in the sense of being bounded on the same graph classes. Then we can consider the radius-$d$ variant of the Flipper Game with the following rules:
\begin{itemize}[nosep]
 \item In every round, first Keeper restricts the arena to the subgraph induced by a radius-$d$ ball around a vertex of her choice, and the Flipper applies a flip of his choice.
 \item The game ends with Flipper's victory once the arena becomes a single vertex.
 \item Flipper's goal is to win in as few rounds as possible, and Keeper's goal is to avoid losing for as long as possible.
\end{itemize}
Again, Gajarsk\'y, Ohlmann, M\"ahlmann, McCarty, Pilipczuk, Przybyszewski, Siebertz, Soko\l{}owski, and Toru\'nczyk~\cite{GajarskyMMOPPSS23} proved that this purely combinatorial game can be used to characterize monadic stability.

\begin{theorem}[\cite{GajarskyMMOPPSS23}]\label{thm:flipper-game-monstable}
 A graph class $\Cc$ is monadically stable if and only if for every $d\in \N$ there exists $k\in \N$ such that for every graph $G\in \Cc$, Flipper can win the radius-$d$ Flipper Game on $G$ within at most $k$ rounds.
\end{theorem}

For the difficult direction of \cref{thm:flipper-game-monstable} (from monadic stability to the existence of a strategy), Gajarsk\'y et al. actually gave two proofs. The first one is non-constructive, relies on model-theoretic methods, and has the advantage of providing certain combinatorial obstructions to monadic stability, called {\em{rocket patterns}}. The second one is more direct, and in particular it provides a concrete strategy for Flipper, in which every next move can be computed in time $\Oh_{\Cc,d}(n^2)$.

\paragraph*{Model-checking.}
Recall that Splitter Game was the key decompositional tool used by Grohe et al.~\cite{GroheKS17} to prove fixed-parameter tractability of $\FO$ model-checking on nowhere dense classes (\cref{thm:nd-mc}). However, one more ingredient was needed to trim the game tree: sparse neighborhood covers. It turns out that these also exist in monadically stable classes, as proved by Dreier, Eleftheriadis, M\"ahlmann, McCarty, Pilipczuk, and Toru\'nczyk~\cite{DreierEMMPT24} in the statement below. However, we need to slightly relax the requirements: we say that a neighborhood cover $\Ff$ of a graph $G$ has {\em{weak radius $r$}} if for every $F\in \Ff$ there exists $u\in V(G)$ such that $F\subseteq \Ball_r^G[u]$. In other, informal words, the distances witnessing the locality of $F$ are measured in the whole graph $G$ instead of the induced subgraph $G[F]$, and in particular we do not even require $G[F]$ to be connected.

\begin{theorem}[\cite{DreierEMMPT24}]\label{thm:nei-cov-ms}
Let $\Cc$ be a monadically stable class of graphs, $d\in \N$ be a distance parameter, and $\eps>0$ be a positive real. Then every graph $G\in \Cc$ admits a neighborhood cover of weak radius $2d$ and overlap $\Oh_{\Cc,d,\eps}(n^\eps)$, where $n=|V(G)|$. Moreover, such a neighborhood cover can be computed in time $\Oh_{\Cc,d,\eps}(n^{4+\eps})$.
\end{theorem}

Recall that in the nowhere dense case, the existence of sparse neighborhood covers was an easy consequence of the bounds on weak coloring numbers (see \cref{thm:nei-cov-be} and the discussion around \cref{thm:quant-nd}). This approach is currently not available in the monadically stable case, due to the lack of a robust analogue of weak coloring numbers; developing such an analogue is a notorious open problem. Therefore, in their proof of \cref{thm:nei-cov-ms}, Dreier et al. took a very different route. We briefly explain this route now, as it brings another set of interesting tools to the picture.

The first step is to prove that monadically stable classes enjoy almost linear neighborhood complexity.

\begin{theorem}[\cite{DreierEMMPT24}]\label{thm:nei-comp-ms}
  For every monadically stable graph class $\Cc$, distance parameter $d\in \N$, positive real $\eps>0$, graph $G\in \Cc$, and a subset of vertices $A\subseteq V(G)$, we have
 $$|\Ss^G_d(A)|\leq \Oh_{\Cc,\eps,d}(|A|^{1+\eps}).$$
\end{theorem}

Note that in the proof of \cref{thm:nei-comp-ms} it suffices to consider the case $d=1$, for the operation of taking the $d$th power of a graph is a transduction, hence the class of $d$th powers of graphs from a monadically stable class $\Cc$ is again monadically stable. The proof of \cref{thm:nei-comp-ms} in \cite{DreierEMMPT24} makes this assumption, and essentially provides a reduction to the nowhere dense case (\cref{thm:quant-nd}, last point). This reduction relies both on sampling methods from discrete geometry, and on a model-theoretic characterization of stability through the notion of {\em{branching index}}.

The next step is to apply a result of Welzl~\cite{Welzl89} on orderings with low {\em{crossing number}}. To introduce this result, we need a few definitions. Consider a set system $\Ss=(U,\Ff)$, where $U$ is a universe and $\Ff$ is a family of subsets of $U$. The set system {\em{dual}} to $\Ss$ is the set sytem $\Ss^\star=(\Ff,U^\star)$, where the universe is the family $\Ff$ and for every $u\in U$ we introduce $u^\star\coloneqq \{F\in \Ff~|~u\in F\}\subseteq \Ff$; then $U^\star\coloneqq \{u^\star\colon u\in U\}$.
Next, for a subset $X\subseteq U$, by $\Ss|_X$ we denote the set system $(X,\Ff_X \coloneqq \{F\cap X\colon F\in \Ff\})$. Then for a function $\pi\colon \N\to \N$, we say that $\Ss$ has {\em{growth}} $\pi$ if $|\Ff_X|\leq \pi(|X|)$, for all $X\subseteq U$. Observe that \cref{thm:nei-comp-ms} implies that, in the notation from the statement, the set system $(V(G),\{\Ball_d^G[u]\colon u\in V(G)\})$ has growth\footnote{More formally, has growth $\pi(t)$ for some function $\pi(t)\in \Oh_{\Cc,\eps,d}(t^{1+\eps})$.} $\Oh_{\Cc,\eps,d}(t^{1+\eps})$, and also note that this set system is self-dual.

Next, we define the crossing number of an ordering of the universe with respect to a set system. Intuitively, an ordering with a low crossing number provides a very basic decomposition for a set system, in which every set can be decomposed into a small number of intervals.

\begin{definition}
Let $\Ss=(\Ff,U)$ be a set system and let $\leq$ be an ordering of $U$. The {\em{crossing number}} of $\leq$ with respect to $\Ff$ is the least $k\in \N$ such that for every set $F\in \Ff$, there are at most $k$ pairs $u,v$ of elements of $U$ such that $u$ and $v$ are consecutive in $\leq$ and exactly one of them belongs to $F$. (Note that this implies that $F$ can be written as the union of $\lceil k/2\rceil$ sets that are convex in $\leq$.)
\end{definition}

With these definitions, the result of Welzl reads as follows. Intuitively, it says that having low growth allows one to construct orderings with low crossing number, where almost linear growth implies subpolynomial crossing number.

\begin{theorem}[\cite{Welzl89}]\label{thm:welzl}
 Let $\Ss=(U,\Ff)$ be a set system and suppose $\Ss^\star$ has growth $\Oh(t^{d})$, for some real $d>1$. Then there is an ordering $\leq$ of $U$ whose crossing number with respect to $\Ff$ is bounded by $\Oh(|U|^{1-1/d}\log |U|)$. Moreover, such an ordering can be computed in time $\Oh((|U|+|\Ff|)^{3+d})$.
\end{theorem}

Thus, by combining \cref{thm:nei-comp-ms} with \cref{thm:welzl} we immediately obtain the following.

\begin{theorem}[\cite{DreierEMMPT24}]\label{thm:welzl-ms}
  For every monadically stable graph class $\Cc$, distance parameter $d\in \N$, positive real $\eps>0$, and a graph $G\in \Cc$, there exists an vertex ordering $\leq$ of $G$ such that every radius-$d$ ball $\Ball^G_d[u]$, $u\in V(G)$, can be written as the union of $\Oh_{\Cc,\eps,d}(n^\eps)$ sets that are convex in $\leq$, where $n=|V(G)|$. Moreover, such a vertex ordering can be computed in time $\Oh_{\Cc,d,\eps}(n^{4+\eps})$.
\end{theorem}

The final touch is a simple greedy procedure that constructs a neighborhood cover based on a vertex ordering with low crossing number with respect to the family of balls. This procedure is described by the following statement.

\begin{theorem}[\cite{DreierEMMPT24}]\label{thm:rose-construction}
 Let $G$ be a graph, $d\in \N$ be a distance parameter, and $\leq$ be a vertex ordering of $G$ such that the crossing number of $\leq$ with respect to the family of radius-$d$ balls $\{\Ball^G_d[u]\colon u\in V(G)\}$ is at most~$k$. Then $G$ admits a distance-$d$ neighborhood of weak radius $2d$ and overlap $k+1$, which moreover can be computed from $G$ and $\leq$ in time $\Oh(n^3)$.
\end{theorem}
\begin{proof}[Proof sketch]
 A set $I\subseteq V(G)$ that is convex in $\leq$ shall be called an {\em{interval}}. An interval $I$ is {\em{compact}} if there exists $u\in V(G)$ such that $I\subseteq \Ball^G_d[u]$.

 The idea is to greedily partition $V(G)$ into compact intervals along the ordering $\leq$. More precisely, we construct a partition $\cal I$ of $V(G)$ into compact intervals as follows:
 \begin{itemize}[nosep]
  \item Start with $\cal I\coloneqq \emptyset$ and $B\coloneqq V(G)$. $B$ will always be a suffix of $\leq$.
  \item As long as $B$ is not empty, let $I$ be the longest compact interval that is a prefix of $B$. Remove $I$ from $B$ and add $I$ to $\cal I$.
 \end{itemize}
 Now, let $\Ff\coloneqq \{\Ball_d^G[I]\colon I\in \cal I\}$, where we denote $\Ball_d^G[I]\coloneqq \bigcup_{u\in I}\Ball_d^G[u]$. That $\Ff$ is a distance-$d$ neighborhood cover of weak radius $2d$ follows directly from the construction. Moreover, the bound on the crossing number of $\leq$ with respect to the family of radius-$d$ balls, together with the maximality of the intervals $I$ extracted in the construction of $\cal I$, can be used to argue that the overlap of $\Ff$ is bounded by~$k+1$. We leave the easy details as an exercise for the reader.
\end{proof}

Now, combining \cref{thm:welzl-ms} with \cref{thm:rose-construction} gives distance-$d$ neighborhood covers of weak radius $2d$ and overlap $\Oh_{\Cc,d,\eps}(n^\eps)$, computable in time $\Oh_{\Cc,d,\eps}(n^{4+\eps})$, for any monadically stable class $\Cc$ and parameters~$d,\eps$. This proves \cref{thm:nei-cov-ms}.

Now that both \cref{thm:flipper-game-monstable,thm:nei-cov-ms} are established, we have suitable analogues of all the tools that were needed in the nowhere dense case to prove \cref{thm:nd-mc}. And indeed Dreier, M\"ahlmann, and Siebertz~\cite{DreierMS23} showed that the reasoning from the proof of \cref{thm:nd-mc} can be suitably lifted to the setting of the Flipper Game, which completes the proof of fixed-parameter tractability of $\FO$ model-checking on monadically stable classes.

\begin{theorem}[\cite{DreierEMMPT24,DreierMS23,GajarskyMMOPPSS23}]\label{thm:ms-mc}
 For every monadically stable class of graphs $\Cc$ there is an algorithm that given an $\FO$ sentence $\phi$ and a graph $G\in \Cc$, decides whether $G\models \phi$. The running time can be bounded by $\Oh_{\Cc,\eps}(n^{6+\eps})$, for any fixed~$\eps>0$.
\end{theorem}

Note that we chose to attribute \cref{thm:ms-mc} to all the three articles~\cite{DreierEMMPT24,DreierMS23,GajarskyMMOPPSS23}, as all of them contributed with major tools needed in the proof. The result was reported in the last article~\cite{DreierEMMPT24}, because chronologically, sparse neighborhood covers were the last piece of the puzzle missing. The earlier wrap-up provided by Dreier, M\"ahlmann, and Siebertz in~\cite{DreierMS23} assumed the existence of sparse neighborhood covers, and actually showed a method for efficiently approximating them (thereby reducing the need of finding them to proving their existence). This method is interesting on its own and relies on rounding a linear relaxation of the natural IP formulation of the problem; see~\cite{DreierMS23} for details.

\medskip

The {\em{Welzl orders}}, that is, orderings with low crossing number provided by \cref{thm:welzl}, served as a crucial ingredient of the proof of \cref{thm:nei-cov-ms}. Their role in the whole theory seems to be far more significant, but is not yet fully understood at this point. Nevertheless, the combination of \cref{thm:welzl,thm:rose-construction} provides a very transparent reduction of the task of (algorithmically) finding sparse neighborhood covers to proving almost linear neighborhood complexity. Note that this reduction can be also applied in the monadically dependent case.

\subsubsection{Monadic dependence}\label{sec:mon-dependent}

Compared to the monadically stable classes discussed in \cref{sec:mon-stable}, monadically dependent classes are much more poorly understood. However, the recent work of Dreier, M\"ahlmann, and Toru\'nczyk~\cite{DreierMT24} sheds some light on their structural properties. In this section we briefly describe their findings.

\medskip

On the positive side, Dreier et al. managed to characterize monadically dependent classes as those the enjoy the property of {\em{flip-breakability}}, defined as follows.

\begin{definition}\label{def:flip-breakability}
 A graph class $\Cc$ is {\em{flip-breakable}} if for every $d\in \N$, there exists a constant $k_d\in \N$ and a function $N_d\colon \N\to \N$ such that the following condition holds. For every graph $G\in \Cc$ and a vertex subset $A\subseteq V(G)$ with $|A|>N_d(m)$ for some $m\in \N$, there exist a $k_d$-flip $G'$ of $G$ and two subsets $B,C\subseteq A$ with $|B|>m$ and $|C|>m$, such that for every $b\in B$ and $c\in C$, the distance between $b$ and $c$ in $G'$ is larger than $d$.
\end{definition}

\begin{theorem}[\cite{DreierMT24}]\label{thm:md-fb}
 A graph class is monadically dependent if and only if it is flip-breakable.
\end{theorem}

Flip-breakability can be understood as a certain ``balanced separator lemma'' for monadically dependent classes, with the following caveats:
\begin{itemize}[nosep]
 \item the separator takes the form of applying a $k_d$-flip to the given graph $G$; and
 \item the separator destroys only short cross-connections between two sizeable subsets, $B$ and $C$.
\end{itemize}
In this vein, it is instructive to compare flip-breakability to flip-flatness (\cref{def:flip-flatness}). In flip-breakability, the separator (flip) destroys short paths connecting a vertex of $B$ and a vertex of $C$, for two sizeable subsets $B$ and $C$ of $A$; so a biclique of connections is destroyed. In flip-flatness, we find a sizeable subset $I$ of $A$ such that the separator (flip) destroys all the short paths between all the pairs of vertices from $I$; so a clique of connections is destroyed. The latter is a stronger property, and indeed monadic stability is more restrictive than monadic dependence.

While the characterization through flip-breakability of \cref{thm:md-fb} might suggest the existence of global decompositions for graphs from monadically dependent classes, no such decompositions have been described so far. We also note that one can consider also a definition of {\em{breakability}} that is analogous to flip-breakability, just with flips replaced with vertex deletions. Indeed, Dreier et al. proved that, as expected, this notion of breakability exactly characterizes nowhere dense classes.

On the negative side, Dreier et al. proved that classes that are {\em{monadically independent}} (i.e., not monadically dependent) can be characterized by the existence of certain combinatorial obstructions in the form of induced subgraphs of $k$-flips, for a constant $k$. More precisely, for any distance parameter $d\in \N$, they considered four families of {\em{patterns}}:
\begin{itemize}[nosep]
 \item {\em{star $d$-crossings}};
 \item {\em{clique $d$-crossings}};
 \item {\em{half-graph $d$-crossings}}; and
 \item {\em{comparability grids}}.
\end{itemize}
A star $d$-crossing of order $n$ is just the $d$-subdivision of the biclique $K_{n,n}$. Clique $d$-crossings and half-graph $d$-crossings are minor variations of such subdivisions, obtained by altering the adjacencies near the principal vertices by introducing either cliques or half-graphs; see \cite{DreierMT24} for details. Comparability grids are a special family that can be seen as a product of two half-graphs: the comparability grid of order $n$ has vertex set $\{1,\ldots,n\}\times \{1,\ldots,n\}$, and two distinct vertices $(a,b)$ and $(a',b')$ are adjacent if and only if $(a-a')(b-b')\geq 0$.

Assuming $d\geq 1$, each of the families of patterns described above forms a sequence of larger and larger grid-like structures from which the class of all graphs can be transduced, similarly as it was done in the proof of \cref{lem:rook-not-dependent} for rook graphs. Therefore, if a class $\Cc$ transduces any of these families of patterns, then $\Cc$ is automatically monadically independent. However, it turns out that any monadically independent class in fact contains one of those families in a very simple way: as induced subgraphs under a $k$-flip, for a constant~$k$.

\begin{theorem}[\cite{DreierMT24}]\label{thm:patterns-md}
 Let $\Cc$ be a monadically independent graph class. Then at least one of the following assertions holds:
 \begin{itemize}[nosep]
  \item For some $k,d\geq 1$, every star $d$-crossing is an induced subgraph of a $k$-flip of a graph from $\Cc$.
  \item For some $k,d\geq 1$, every clique $d$-crossing is an induced subgraph of a $k$-flip of a graph from $\Cc$.
  \item For some $k,d\geq 1$, every half-graph $d$-crossing is an induced subgraph of a $k$-flip of a graph from $\Cc$.
  \item For some $k\in \N$, every comparability grid is an induced subgraph of a $k$-flip of a graph from $\Cc$.
 \end{itemize}
\end{theorem}

Using the fact that the patterns found by \cref{thm:patterns-md} are induced subgraphs of $k$-flips, Dreier et al. were able to derive a complexity lower bound analogous to \cref{thm:md-limit}: as far as hereditary graph classes are concerned, no tractability of $\FO$ model-checking should be expected beyond monadic dependence.

\begin{theorem}[\cite{DreierMT24}]\label{thm:md-limit}
 Let $\Cc$ be a class of graphs that is hereditary and monadically independent. Then the model-checking problem for $\FO$ on $\Cc$ is as hard as on general graphs.
\end{theorem}

Finally, we remark that in their study of monadically stable classes, Dreier et al.~\cite{DreierEMMPT24} obtained a similar characterization for those classes. Not surprisingly, the class of half-graphs enters the scene as one of the families of forbidden patterns. Then half-graph $d$-crossings and comparability grids are no longer necessary, as they contain arbitrarily large induced half-graphs.

\begin{theorem}[\cite{DreierEMMPT24}]\label{thm:patterns-ms}
 Let $\Cc$ be a graph class that is not monadically stable. Then one of the following assertions~holds:
 \begin{itemize}[nosep]
  \item For some $k,d\geq 1$, every star $d$-crossing is an induced subgraph of a $k$-flip of a graph from $\Cc$.
  \item For some $k,d\geq 1$, every clique $d$-crossing is an induced subgraph of a $k$-flip of a graph from $\Cc$.
  \item For some $k\in \N$, every half-graph is an induced subgraph of a $k$-flip of a graph from $\Cc$.
 \end{itemize}
\end{theorem}

\subsection{Ideals defined through closure}\label{sec:closure}

We now proceed to the other side of the coin: analyzing the structure in $\FO$ ideals defined by closing graph properties under transductions. More formally, recall that if $\Pi$ is a property of graph classes, then we say that a graph class $\Cc$ is {\em{structurally $\Pi$}} if $\Cc$ can be transduced from a class enjoying $\Pi$. Thus, classes that are structurally $\Pi$ form an $\FO$ ideal, and it is natural to ask for an understanding of the structure of graphs that are structurally $\Pi$. We will typically be interested in $\Pi$ ranging over classic properties of classes of sparse graphs, such as classes of bounded treewidth, bounded expansion classes, or nowhere dense classes.

\subsubsection{Structurally sparse classes}\label{sec:structurally-sparse}

The chronologically first result in this direction was delivered by Gajarsk\'y, Hlin\v{e}n\'y, Obdr\v{z}\'alek, Lokshtanov, and Ramanujan~\cite{GajarskyHOLR20}, who characterized classes of structurally bounded maximum degree. It turns out that these are just classes of bounded maximum degree obfuscated by a bounded number of flips.

\begin{theorem}[\cite{GajarskyHOLR20}]\label{thm:str-bnd-deg}
 A class of graphs $\Cc$ has structurally bounded maximum degree if and only if there exists $k\in \N$ and a class $\Dd$ of bounded maximum degree such that every graph in $\Cc$ is a $k$-flip of a graph in $\Dd$.
\end{theorem}

The proof of \cref{thm:str-bnd-deg} is not too difficult, but it introduced important technical methods for working with transductions, connected to locality of First-Order logic. More precisely, the idea of the proof (even though it is not that explicit in \cite{GajarskyHOLR20}) is to show that every transduction can be decomposed into a {\em{local part}} that can introduce new edges only between vertices that are close to each other in the graph, and a {\em{long-distance part}} that essentially can just apply some flips between vertices that are far from each other. This methodology turned out to be important in later works, see for instance~\cite{BonnetDGKMST22,BraunfeldNOS22transductions,GajarskyGK22}.

Next, a significant amount of work has been devoted to understanding structurally bounded expansion classes. Recall that in \cref{thm:low-treedepth-covers-be}, classes of bounded expansion were characterized as those that admit low treedepth covers. Gajarsk\'y, Kreutzer, Ne\v{s}et\v{r}il, Ossona de Mendez, Pilipczuk, Siebertz, and Toru\'nczyk~\cite{GajarskyKNMPST20} proved that classes of structurally bounded expansion can be characterized by the existence of {\em{low shrubdepth covers}} in the following sense.

\begin{theorem}[\cite{GajarskyKNMPST20}]\label{thm:low-shrubdepth-covers-sbe}
 A class of graphs $\Cc$ has structurally bounded expansion if and only if the following condition holds: for every $p\in \N$, there is a constant $K_p\in \N$ and a graph class $\Dd_p$ of bounded shrubdepth such that for every graph $G\in \Cc$, there is a family $\Ff$ of subsets of $V(G)$ with $|\Ff|\leq K_p$ such that:
 \begin{itemize}[nosep]
  \item for every $F\in \Ff$, we have $G[F]\in \Dd_p$; and
  \item for every subset of vertices $X\subseteq V(G)$ with $|X|\leq p$, there exists $F\in \Ff$ such that $X\subseteq F$.
 \end{itemize}
\end{theorem}

A family $\Ff$ satisfying the two conditions above will be called a {\em{$p$-cover}} of $G$ {\em{captured}} by $\Dd_p$.

Further structural characterizations of structurally bounded expansion classes were proposed by Dreier in~\cite{Dreier23} (through {\em{lacon-}}, {\em{shrub-}}, and {\em{parity-decompositions}}), and by Dreier, Gajarsk\'y, Kiefer, Pilipczuk, and Toru\'nczyk in~\cite{DreierGKPT22} (through {\em{bushes}}). Roughly speaking, all  those characterizations have a similar spirit: they show that if $\Cc$ is a class of structurally bounded expansion, then for every graph $G\in \Cc$ one can find a sparse graph $H_G$ such that (i) $G$ can be transduced from $H_G$ using a concrete, very simple transduction, and (ii) the class $\{H_G\colon G\in \Cc\}$ has bounded expansion. Actually, without going into technical details, the characterization of \cref{thm:low-shrubdepth-covers-sbe} can be also understood in this way. Thus, all the results mentioned above have both a {\em{decompositional}} character --- $H_G$ may serve as a sparse skeleton of a decomposition in which $G$ can be encoded --- and a {\em{normalization}} character --- they show that if $\Cc$ can be transduced from a class of bounded expansion, then $\Cc$ can be transduced from a very specific class of bounded expansion using a very specific~transduction.

Due to the technical character of the characterizations discussed above, in this survey we refrain from giving their precise description. An interested reader is invited to the works~\cite{Dreier23,DreierGKPT22}.

Finally, inspired by the characterization of structurally bounded expansion classes through bushes, Dreier et al.~\cite{DreierGKPT22} proved a decomposition result for structurally nowhere dense classes that involves a related notion of {\em{quasi-bushes}}. This result turned out to be somewhat important in the later works, hence we present quasi-bushes in full formality.

\begin{definition}
 A {\em{quasi-bush}} $B$ consists of
 \begin{itemize}[nosep]
   \item a rooted tree $T$;
   \item a finite set of labels $\Lambda$;
   \item a labeling $\lambda\colon \mathsf{Leaves}(T)\to \Lambda$, where $\mathsf{Leaves}(T)$ is the set of leaves of $T$;
   \item a set $D$ of arcs called {\em{pointers}}, each with tail in a leaf of $T$ and head in an internal vertex of $T$; and
   \item a labeling $\lambda^D\colon D\to 2^\Lambda$.
 \end{itemize}
 We require that for each leaf $u\in \mathsf{Leaves}(T)$, there is a pointer $(u,r)$, where $r$ is the root of $T$. The {\em{depth}} of $B$ is the depth of $T$.

 A quasi-bush $B$ {\em{represents}} the graph $G(B)$ defined as follows:
 \begin{itemize}[nosep]
  \item The vertex set of $G(B)$ is $\mathsf{Leaves}(T)$.
  \item For two distinct $u,v\in \mathsf{Leaves}(T)$, let $x$ be the deepest (furthest from the root) ancestor of $u$ such that $(v,x)\in D$, and $y$ be the deepest ancestor of $v$ such that $(u,y)\in D$. Then $u$ and $v$ are adjacent in $G(B)$ if and only if $\lambda(u)\in \lambda^D(v,x)$ and $\lambda(v)\in \lambda^D(u,y)$.
 \end{itemize}
\end{definition}

The reader may think of a quasi-bush as of a structure similar to a tree-model (which we have seen in the context of shrubdepth, \cref{def:tree-model}), except that besides the bounded-depth tree there is also a net of pointers that provide additional information across different branches.

It is not hard to see that as long as a quasi-bush $B$ has constant depth, the graph $G(B)$ represented by $B$ can be transduced from $B$ (where $B$ is suitably represented as a relational structure). Indeed, the mechanism of finding the lowest ancestor of one given vertex that has a point from another given vertex, and reading the label of this pointer, can be clearly expressed in $\FO$.

We need one more definition. We say that a graph class $\Cc$ is {\em{almost nowhere dense}} if it satisfies the conclusion of \cref{thm:nd-grads}; that is,
 $$\nabla_h(G)\leq \Oh_{\Cc,h,\eps}(|V(G)|^{\eps})\qquad\textrm{for every }G\in \Cc, h\in \N, \eps>0.$$
Thus, almost nowhere dense classes satisfy the qualitative properties described in \cref{thm:quant-nd} (except the last point, the neighborhood complexity), but not necessarily the quantitative properties such as flatness or short termination of the Splitter Game.

With all these definitions, we can finally state the result of Dreier et al.

\begin{theorem}[\cite{DreierGKPT22}]\label{thm:snd-qbushes}
 Let $\Cc$ be a structurally nowhere dense class of graphs. Then there exists a class of quasi-bushes $\Bb$ such that
 \begin{itemize}[nosep]
  \item the quasi-bushes from $\Bb$ have bounded depth (there is $h$ such that each $B\in \Bb$ has depth at most~$h$);
  \item the class of Gaifman graphs of the quasi-bushes from $\Bb$ is almost nowhere dense; and
  \item for every $G\in \Cc$, there is a quasi-bush $B\in \Bb$ such that $G=G(B)$.
 \end{itemize}
\end{theorem}

Dreier et al. used this result to show that structurally nowhere dense classes admit low shrubdepth covers of size $\Oh_{\Cc,p,\eps}(n^\eps)$. This gives the expected generalization of the characterization of \cref{thm:low-shrubdepth-covers-sbe} to the setting of structurally nowhere dense classes.

\begin{theorem}[\cite{DreierGKPT22}]\label{thm:snd-low-shb-covers}
 Let $\Cc$ be a structurally nowhere dense class of graphs. Then for every $p\in \N$, there is a graph class $\Dd_p$ of bounded shrubdepth such that for every graph $G\in \Cc$ and $\eps>0$, $G$ admits a $p$-cover governed by $\Dd_p$ of size $\Oh_{\Cc,p,\eps}(|V(G)|^\eps)$.
\end{theorem}

\subsubsection{Sparsification conjecture}

In \cref{sec:structurally-sparse}, we have discussed that in many classes of graphs that are {\em{structurally sparse}} --- can be transduced from classes of well-behaved sparse graphs --- one can find and describe the structure through various decomposition statements. Another question is the following: How do those classes fit into the large picture? How do they relate to the $\FO$ ideals that were discussed before, such as monadically stable classes or classes of bounded twin-width or cliquewidth?

To make the discussion more concrete, let us consider the $\FO$ ideal of structurally nowhere dense graph classes. Observe that since every nowhere dense class is monadically stable (\cref{thm:nd-monst}), also every structurally nowhere dense class is monadically stable, for monadic stability is closed under applying transductions. Hence, structurally nowhere dense classes are a subset of monadically stable classes. Could those notions actually coincide? This is the question asked in the following {\em{Sparsification Conjecture}}.

\begin{conjecture}[Sparsification Conjecture]\label{conj:sparsification}
 Every monadically stable graph class is structurally nowhere~dense.
\end{conjecture}

Thus, on a more intuitive level, \cref{conj:sparsification} postulates that graphs from any monadically stable class $\Cc$ are just ``sparse graphs in disguise'' in the following sense: For any graph $G\in \Cc$ one can find a ``skeleton'' graph $H_G$ so that $G$ can be encoded in $H_G$ using a fixed $\FO$-definable mechanism (formally, $G$ can be transduced from $H_G$) and graphs $H_G$ are sparse (formally, the class $\{H_G\colon G\in \Cc\}$ is nowhere dense).

Similarly as \cref{conj:main}, also \cref{conj:sparsification} had circulated in various forms in the community before being codified in published literature. The first concrete mention known to the author is due to Ossona de Mendez~\cite{POM21}.

Very recently, major progress towards \cref{conj:sparsification} was made by Braunfeld, Ne\v{s}et\v{r}il, Ossona de Mendez, and Siebertz~\cite{BraunfeldNOS22decompositions}. Namely, they used the characterization of monadically stable classes through the Flipper Game~\cite{GajarskyMMOPPSS23} together with a structural understanding of the sparse neighborhood covers developed in~\cite{DreierEMMPT24} to prove that monadically stable classes admit almost nowhere dense quasi-bushes. This generalizes \cref{thm:snd-qbushes} from the structurally nowhere dense case to the monadically stable case.

\begin{theorem}[\cite{BraunfeldNOS22decompositions}]\label{thm:ms-qbushes}
 Let $\Cc$ be a monadically stable class of graphs. Then there exists a class of quasi-bushes $\Bb$ such that
 \begin{itemize}[nosep]
  \item the quasi-bushes from $\Bb$ have bounded depth;
  \item the class of Gaifman graphs of the quasi-bushes from $\Bb$ is almost nowhere dense; and
  \item for every $G\in \Cc$, there is a quasi-bush $B\in \Bb$ such that $G=G(B)$.
 \end{itemize}
\end{theorem}

Recall that decoding a graph from a quasi-bush representing it can be done by an $\FO$ transduction. Therefore, \cref{thm:ms-qbushes} falls short of confirming \cref{conj:sparsification} by providing a class of quasi-bushes that is only {\em{almost}} nowhere dense, instead of being actually nowhere dense. Strengthening \cref{thm:ms-qbushes} by making (the Gaifman graphs of) $\Bb$ nowhere dense would imply \cref{conj:sparsification}.

Let us note that similarly to \cref{thm:snd-qbushes}, it can be derived from \cref{thm:ms-qbushes} that monadically stable classes admit low shrubdepth covers of a subpolynomial size; this generalizes \cref{thm:snd-low-shb-covers}.

\begin{theorem}[\cite{BraunfeldNOS22decompositions}]\label{thm:ms-low-shb-covers}
 Let $\Cc$ be a monadically stable class of graphs. Then for every $p\in \N$, there is a graph class $\Dd_p$ of bounded shrubdepth such that for every graph $G\in \Cc$ and $\eps>0$, $G$ admits a $p$-cover governed by $\Dd_p$ of size $\Oh_{\Cc,p,\eps}(|V(G)|^\eps)$.
\end{theorem}

Sparsification Conjecture, as stated in \cref{conj:sparsification}, concerns only structurally nowhere dense classes. However, similar questions can be asked also about more restrictive properties of classes of sparse graphs. Consider, for instance, classes of {\em{structurally bounded treewidth}}, that is, classes transducible from classes of bounded treewidth. On one hand, boundedness of treewidth implies monadic stability, hence classes of structurally bounded treewidth are monadically stable. On the other hand, boundedness of treewidth implies boundedness of cliquewidth (\cref{thm:tw-cw}) and classes of bounded cliquewidth form an $\FO$ ideal (\cref{thm:cw-lcw-ideals}), hence classes of structurally bounded treewidth have bounded cliquewidth. But are those two constraints exhaustive, that is, does every monadically stable class of bounded cliquewidth actually have structurally bounded treewidth? Same question can be asked about other properties of classes of sparse graphs, such as classes of bounded treedepth, pathwidth, or sparse twin-width.

The answer for treedepth is trivial: classes of bounded shrubdepth are monadically stable (\cref{lem:half-graphs-shb}) and can be $\FO$-transduced from trees of bounded height (\cref{thm:shrub-logic}), hence the property of having structurally bounded treedepth coincides with the property of having bounded shrubdepth, which in turn implies monadic stability. Somewhat surprisingly, for pathwidth, treewidth, and sparse twin-width, the answer also turns out to be positive, but in a highly non-trivial way. The following results were proven in order by Ne\v{s}et\v{r}il, Ossona de Mendez, Rabinovich and Siebertz~\cite{NesetrilMRS21}, by Ne\v{s}et\v{r}il, Ossona de Mendez, Pilipczuk, Rabinovich, and Siebertz~\cite{NesetrilMPRS21}, and by Gajarsk\'y, Pilipczuk, and Toru\'nczyk~\cite{GajarskyPT22}.

\begin{theorem}[\cite{NesetrilMRS21}]\label{thm:spars-lcw}
 A class of graphs has structurally bounded pathwidth if and only if it is edge-stable and has bounded linear cliquewidth.
\end{theorem}

\begin{theorem}[\cite{NesetrilMPRS21}]\label{thm:spars-cw}
 A class of graphs has structurally bounded treewidth if and only if it is edge-stable and has bounded cliquewidth.
\end{theorem}

\begin{theorem}[\cite{GajarskyPT22}]\label{thm:spars-tww}
 A class of graphs has structurally bounded sparse twin-width if and only if it is edge-stable and has bounded twin-width.
\end{theorem}

Recall here that a class $\Cc$ is {\em{edge-stable}} if $\Cc$ excludes some half-graph as a semi-induced subgraph. Thus, in general edge-stability is a weaker assumption than monadic stability, but the notions coincide under the assumption of monadic dependence (\cref{thm:ms-es}). Hence, the appearance of this weaker assumption in \cref{thm:spars-lcw,thm:spars-cw,thm:spars-tww} should not come as a surprise.

We remark that in fact, the proofs of \cref{thm:spars-lcw,thm:spars-cw,thm:spars-tww} provide a significantly stronger assertion. Taking \cref{thm:spars-tww} as an example, the argument shows the following: For any edge-stable class $\Cc$ of bounded twin-width, with every graph $G\in \Cc$ one can associate a graph $H_G$ so that
\begin{itemize}[nosep]
 \item $H_G$ can be transduced from $(G,\leq)$, where $\leq$ is a vertex ordering witnessing a bound on the twin-width of $G$ in the sense of \cref{thm:tww-mixed};
 \item $G$ can be transduced back from $H_G$; and
 \item the class $\{H_G\colon G\in \Cc\}$ has bounded sparse twin-width.
\end{itemize}
Thus, one may think of $H_G$ as of a sparse ``decomposition'' of $G$ that on one hand can be computed from $(G,\leq)$ by means of a transduction, and on the other hand, encodes $G$ so that $G$ can be decoded from $H_G$ by means of a transduction. Similar statements can be extracted from the proofs of \cref{thm:spars-lcw,thm:spars-cw}\footnote{In fact, such statements also follow from the statement for twin-width discussed here. This is because classes of bounded (linear) cliquewidth form $\FO$ ideals (\cref{thm:cw-lcw-ideals}), hence assuming $\Cc$ has bounded (linear) cliquewidth, the class $\{H_G\colon G\in \Cc\}$ will have bounded pathwidth/treewidth.}.

The proof of \cref{thm:spars-lcw} relies on {\em{Simon's Factorization Theorem}}~\cite{Kufleitner08,Simon90}: a Ramseyan tool originating in the theory of formal languages. This tool allows one to break a given linear laminar decomposition of bounded diversity into smaller and smaller subdecompositions in a hierarchical way so that (i) the hierarchy has a bounded depth, and (ii) each break in the hierarchy is in some sense very well-behaved. Then suitable transductions are constructed by an induction on the depth of the hierarchy. The proof of \cref{thm:spars-cw} follows a similar path, but uses the tree version of Simon's Factorization, due to Colcombet~\cite{Colcombet07}, and is far more technical. Finally, the proof of \cref{thm:spars-tww} is quite different, and relies on an induction on (roughly speaking) the largest order of a half-graph that can be found as a semi-induced subgraph, combined with a delicate analysis of a contraction sequence witnessing the bound on twin-width. Thus, this argument shows that tools related to Simon's Factorization are not really necessary for obtaining results of the form of \cref{thm:spars-lcw,thm:spars-cw,thm:spars-tww}.

\section{Outlook}\label{sec:outlook}

Within the discussion presented in \cref{sec:classic,sec:new} we mentioned a number of concrete open problems. Particularly, \cref{conj:main,conj:sparsification} are two questions that serve as the motivation for a large part of research in the area. In this concluding section, we would like to highlight three particular directions that were not explicitly mentioned before, but deserve a closer look.

\paragraph*{Obstructions for (linear) cliquewidth.} Recall that through \cref{thm:shrub-obstr}, Ossona de Mendez, Pilipczuk, and Siebertz~\cite{OssonaPS21} gave an obstruction characterization of classes of bounded shrubdepth: these are exactly classes that do not transduce the class of all paths. Observe that this implies the following:
\begin{quote}
Classes of bounded shrubdepth are the largest $\FO$ ideal $\Pi$ such that every weakly sparse class belonging to $\Pi$ has bounded treedepth.
\end{quote}
This formulation is remarkable for the following reason: it enables us to recover the notion of having bounded shrubdepth purely from the notion of having bounded treedepth, thus turning a sparse notion (treedepth) into an analogous dense notion (shrubdepth).

In~\cite{GajarskyPT22}, Gajarsk\'y, Pilipczuk, and Toru\'nczyk conjectured that the same recipe can be applied to recover linear cliquewidth from pathwidth, and cliquewidth from treewidth. That is, they hypothesized that classes of bounded cliquewidth are the largest $\FO$ ideal whose intersection with weakly sparse classes are classes of bounded treewidth; and the same for linear cliquewidth and pathwidth. As noted in~\cite{GajarskyPT22}, these hypotheses are equivalent to the following conjectures about obstruction characterizations of classes of bounded (linear) cliquewidth.

\begin{conjecture}[\cite{GajarskyPT22}]\label{conj:lcw-obst}
 A class of graphs $\Cc$ has unbounded linear cliquewidth if and only if $\Cc$ transduces a class $\Dd$ that contains some subdivision of every binary tree.
\end{conjecture}

\begin{conjecture}[\cite{GajarskyPT22}]\label{conj:cw-obst}
 A class of graphs $\Cc$ has unbounded cliquewidth if and only if $\Cc$ transduces a class $\Dd$ that contains some subdivision of every wall.
\end{conjecture}

Recall that a {\em{wall}} is the subcubic grid-like structure presented in \cref{fig:wall}.

It is instructive to compare \cref{conj:lcw-obst,conj:cw-obst} with \cref{con:lcw-model} and \cref{thm:cw-dependent}. In short, \cref{conj:lcw-obst,conj:cw-obst} are much stronger statements, because they concern the more restrictive notion of $\FO$ transductions, instead of $\MSO$ or $\CMSO$ transductions. Indeed, a positive resolution of \cref{conj:lcw-obst} would immediately imply \cref{con:lcw-model}, and a positive resolution of \cref{conj:cw-obst} would strengthen \cref{thm:cw-dependent} by replacing $\CMSO$ transductions with $\MSO$ transductions.

It seems that \cref{conj:lcw-obst,conj:cw-obst} are much closer in spirit to the search for a characterization of classes of bounded treewidth through induced subgraph obstructions. This question has recently been actively researched within the structural graph theory community, see for instance~\cite{BonamyBDEGHTW24,Korhonen23} and references therein. Though, \cref{conj:lcw-obst,conj:cw-obst} might be significantly easier to answer, as they concern the more relaxed $\FO$ transduction quasi-order, instead of the extremely rigid induced subgraph order.

\paragraph*{Flip-width and dense analogues of sparse notions.} Throughout this survey, we have seen multiple dense analogues of sparse notions: shrubdepth is a dense analogue of treedepth, cliquewidth is a dense analogue of treewidth, and monadic dependence is a dense analogue of nowhere denseness. However, for some sparse notions, their dense analogues are so far only partially understood, or not understood at all.

The most important example here are classes of bounded expansion, for which we only discussed their closure under transductions (structurally bounded expansion classes), but no analogue on the ``monadically dependent level'' was mentioned. Such an analogue was very recently proposed by Toru\'nczyk~\cite{Torunczyk23} in the form of classes of {\em{bounded flip-width}}.

The idea of Toru\'nczyk is to take inspiration from the classic {\em{Cops \& Robber Game}} that characterizes treewidth, and vary the rules to define a hierarchy of graph parameters that on one hand generalize cliquewidth, and on the other hand are bounded on classes of bounded expansion. The proposed game, which for the purpose of this survey we will call the {\em{distance-$d$ Flip-width Game}}, is a variation of the Flipper Game that we discussed in \cref{sec:mon-stable}. There are two players, {\em{Flipper}} and {\em{Runner}}. Besides the fixed distance parameter $d\in \N$, there is an additional parameter $k\in \N$ that signifies the strength of Flipper. (Roughly, $k$ corresponds to the number of cops in the Cops \& Robber Game.)
The game is played in rounds on a graph $G$. The arena after the $i$th round will be denoted by $G_i$, and $G_i$ will always be a $k$-flip of $G$; initially we set $G_0\coloneqq G$. Also, Runner always stands on some vertex of the graph. By $u_i$ we denote the Runner's position after the $i$th round, and we let Runner choose $u_0$ arbitrarily before the game begins. With this setup, the $i$th round of the game ($i\geq 1$) proceeds as follows:
\begin{itemize}[nosep]
 \item Flipper announces the next graph $G_i$ that has to be a $k$-flip\footnote{In~\cite{Torunczyk23}, Toru\'nczyk uses bipartite flips defined by two vertex subsets, instead of simple flips defined by one vertex subset that we adopted as the main definition here. This is an immaterial detail.} of $G$.
 \item Runner moves from $u_{i-1}$ to any vertex $u_i$ of her choice such that the distance between $u_i$ and $u_{i-1}$ in $G_{i-1}$ is at most $d$.
\end{itemize}
The game finishes with Flipper's victory once $u_i$ becomes isolated in the graph $G_i$. The goal of Runner is to avoid losing indefinitely. Thus, compared to the Flipper Game considered in \cref{sec:mon-stable}, there is no bound on the duration of the game, but Flipper's flips do not ``stack up'': in every round he needs to propose a new set of $k$ flips, while the flips applied in the previous round get forgotten.

\newcommand{\fw}{\mathsf{fw}}

With this definition of the game, we may define a graph parameter {\em{distance-$d$ flip-width}}, denoted $\fw_d(G)$, as the minimum $k\in \N$ such that Flipper may win the distance-$d$ Flip-width Game on $G$ when given strength $k$. Then we say that a graph class $\Cc$ has {\em{bounded flip-width}} if $\Cc$ has bounded distance-$d$ flip-width for every $d\in \N$; that is, $\fw_d(\Cc)$ is finite for every $d\in \N$.

As proved by Toru\'nczyk, this notion seems to perfectly fit into the place of the  analogue of bounded expansion on the monadically dependent level. Precisely, we have the following.

\begin{theorem}[\cite{Torunczyk23}]
 The following assertions hold:
 \begin{itemize}[nosep]
  \item Classes of bounded flip-width form an $\FO$-ideal. In particular, every class of bounded flip-width is monadically dependent.
  \item Every class of bounded twin-width also has bounded flip-width.
  \item Every class of bounded expansion has bounded flip-width. Moreover, every weakly sparse class of bounded flip-width in fact has bounded expansion.
 \end{itemize}
\end{theorem}

Moreover, in~\cite{Torunczyk23} it is shown that the distance-$\infty$ variant of the game (where the Runner may move within a connected component of $G_{i-1}$) exactly characterizes classes of bounded cliquewidth. Also, the variant of the game where the operation of applying a flip is replaced by the operation of {\em{isolating a vertex}} (i.e., removing all edges adjacent to it) characterizes classes of bounded expansion.

While the concept of flip-width seems to fit the large picture surprisingly well, it has certain limitations which show that there is more to be understood. Most importantly, classes of bounded flip-width are so far only defined through the existence of strategies in the Flip-width Games, which provides only a loose grasp on their structural properties. In particular, the following questions are open:
\begin{itemize}[nosep]
 \item It is unknown whether the $\FO$ model-checking problem can be solved in fixed-parameter time on classes of bounded flip-width, even under the assumption of being provided suitable Flipper's strategies through oracles.
 \item It is unknown whether edge-stable classes of bounded flip-width coincide with classes of structurally bounded expansion.
 \item It is unknown whether classes of bounded flip-width are $\chi$-bounded\footnote{A class $\Cc$ is {\em{$\chi$-bounded}} if there is a function $f\colon \N\to \N$ such that $\chi(G)\leq f(\omega(G))$ for each $G\in \Cc$, where $\chi(G)$ and $\omega(G)$ denote the chromatic number and the clique number of $G$, respectively.}.  This property is known to hold both for classes of bounded twin-width~\cite{BonnetG0TW21} and for structurally bounded expansion classes (this follows easily from~\cref{thm:low-shrubdepth-covers-sbe}).
\end{itemize}
In order to answer these and other related questions, it seems necessary to develop some form of a global, static decomposition for classes of bounded flip-width. No such decomposition is known at this point.

We note that in~\cite{Torunczyk23}, Toru\'nczyk asked a number of other interesting questions concerning flip-width, particularly about combinatorial obstructions. We refer the reader to~\cite{Torunczyk23} for details.

Besides classes of bounded expansion, the other sparse concept for which we do not have a robust dense analogue are minor-free classes. Here, the author is not aware of any work on {\em{structurally minor-free classes}}, nor of any propositions for dense analogues on the monadically dependent level (that is, fitting between classes of bounded cliquewidth and classes of bounded twin-width). In the context of structural graph theory, the notion of {\em{vertex-minors}} seems to provide a good dense analogue of the minor order, but unfortunately this notion does not seem to fit well our theory here. The reason is that forbidding a vertex-minor does not entail monadic dependence. A classic example here is the class of {\em{circle graphs}} (intersection graphs of chords of a circle), which are defined through forbidding vertex-minors~\cite{Bouchet94}, but are monadically independent~\cite{HlinenyPR19}. There is a notion of {\em{shallow vertex-minors}} introduced by Ne\v{s}et\v{r}il, Ossona de Mendez, and Siebertz~\cite{NesetrilOdMS22}, which can be used to give obstruction characterizations for structurally bounded expansion classes~\cite{NesetrilOdMS22} as well as, as shown very recently by Buffi\`ere, Kim, and Ossona de Mendez~\cite{BuffiereKOdM24}, monadically stable and  monadically dependent classes. However, so far it is unclear whether the notion of vertex-minors, without any constraints on the depth, can be combined with the theory of monadically dependent graph classes in any meaningful way.

\paragraph*{Fine understanding of the transduction order.} Finally, recall that due to the compositionality of $\FO$ transductions (\cref{lem:trans-composition}), the relation of transducibility $\tleq_\FO$ forms a quasi-order on graph classes. While the coarser quasi-order given by $\MSO$ transducibility is largely understood thanks to the work of Blumensath and Courcelle~\cite{BlumensathC10}, the $\FO$ transducibility quasi-order appears to be much more complex. A wealth of insight has been provided by the recent work of Braunfeld, Ne\v{s}et\v{r}il, Ossona de Mendez, and Siebertz~\cite{BraunfeldNOS22transductions,NesetrilMS22}, but there is still much to be explored and understood.

Particularly, we are currently lacking robust tools for proving negative results: that some graph class $\Dd$ cannot be transduced from another graph class $\Cc$. This applies even to very basic graph classes, studied through and through from the point of view of graph theory. A particular setting that we would like to propose here is that of graphs embeddable in a fixed surface.

As in \cref{sec:minors}, by a surface we mean a compact $2$-dimensional manifold $\Sigma$ without boundary, and we consider the standard notion of embedding: vertices of a graph are mapped to distinct points of the surface, and edges are mapped to internally-disjoint curves connecting respective endpoints. Recall that the Classification Theorem for Closed Surfaces states that every surface is homeomorphic to one of the following surfaces:
\begin{itemize}[nosep]
 \item the sphere;
 \item the surface $\Sigma_{k\times \mathbf{T}}$ obtained from the sphere by gluing in $k\geq 1$ handles; and
 \item the surface $\Sigma_{\ell\times \mathbf{P}}$ obtained from the sphere by gluing in $\ell\geq 1$ crosscaps.
\end{itemize}
We also write $\Sigma_{0\times \mathbf{T}}=\Sigma_{0\times \mathbf{P}}$ for the sphere.
These surfaces are pairwise non-homeomorphic, but
\begin{itemize}[nosep]
 \item if $0\leq k\leq k'$, then every graph embeddable in $\Sigma_{k\times \mathbf{T}}$ is also embeddable in $\Sigma_{k'\times \mathbf{T}}$;
 \item if $0\leq \ell\leq \ell'$, then every graph embeddable in $\Sigma_{\ell\times \mathbf{P}}$ is also embeddable in $\Sigma_{\ell'\times \mathbf{P}}$; and
 \item if $0\leq k$, then every graph embeddable in $\Sigma_{k\times \mathbf{T}}$ is also embeddable in $\Sigma_{(2k+1)\times \mathbf{P}}$.
\end{itemize}
More generally, we write $\Sigma \preceq \Gamma$ if every graph embeddable in $\Sigma$ is also embeddable in $\Gamma$. By $\Ee_\Sigma$ we denote the class of graphs embeddable in surface $\Sigma$.

It is unclear how the classes $\Ee_\Sigma$ for different surfaces $\Sigma$ relate to each other in the $\FO$ transduction quasi-order, besides the obvious inclusion of classes following from the partial order $\preceq$ on the surfaces. We conjecture that no more transducibility relations are present between those classes.

\begin{conjecture}\label{conj:surfaces}
 Let $\Sigma$ and $\Gamma$ be surfaces such that $\Sigma \npreceq \Gamma$. Then $\Ee_\Sigma \not\sqsubseteq_\FO \Ee_\Gamma$.
\end{conjecture}

\cref{conj:surfaces} is not known to hold even when $\Sigma$ is a torus and $\Gamma$ is a sphere. In this simple setting, the question boils down to the following: Is the class of toroidal graphs transducible from the class of planar~graphs?

To address this question, one could first focus on a very specific form of a transduction, namely taking {\em{congested shallow minors}}. Here, we say that a graph $H$ is a {\em{congestion-$c$ depth-$d$ minor}} of $G$ if there is a {\em{congestion-$c$ depth-$d$ minor model}} of $H$ in $G$, which is a mapping $\eta$ satisfying the following:
\begin{itemize}[nosep]
 \item for every vertex $u\in V(H)$, $\eta(u)$ is a connected subgraph of $G$ of radius at most $d$;
 \item for every edge $uv\in E(G)$, either $\eta(u)$ and $\eta(v)$ share a vertex, or there is an edge of $G$ with one endpoint in $\eta(u)$ and second in $\eta(v)$; and
 \item every vertex $w\in V(G)$ belongs to at most $c$ subgraphs $\eta(u)$, for $u\in V(H)$.
\end{itemize}
It is not hard to prove that for every class $\Cc$ of bounded expansion and constants $c,d\in \N$, there is a transduction $\Tf_{\Cc,c,d}$ (with copying) such that $\Tf_{\Cc,c,d}(\Cc)$ contains all congestion-$c$ depth-$d$ minors of $\Cc$. It seems that general transductions between classes of sparse graphs cannot do much more than taking congested shallow minors, hence it makes sense to focus on this purely combinatorial operation first. Thus, based on joint discussions with Jakub Gajarsk\'y and Szymon Toru\'nczyk, we pose the following question as a potential stepping stone towards \cref{conj:surfaces}.

\begin{conjecture}
 There is no constant $k\in \N$ with the following property: every toroidal graph is a congestion-$k$ depth-$k$ minor of a planar graph.
\end{conjecture}

\bibliographystyle{abbrv}
\bibliography{references}

\end{document}